\documentclass[10pt,leqno]{amsart}
\usepackage{graphicx}

\usepackage[utf8]{inputenc}	
\usepackage{url}
\usepackage{xcolor, soul}
\sethlcolor{green}

\usepackage{tikz}
\usetikzlibrary{positioning, shapes.geometric, calc}

\usepackage{amsmath,amsthm,amsfonts,amssymb,amscd,bezier,dsfont,bm}
\allowdisplaybreaks
\usepackage{nicematrix}
\usepackage{graphicx}
\usepackage{mathtools}
\usepackage{empheq} 
\usepackage{lipsum}
\usepackage{bbold}

\usepackage{indentfirst}

\usepackage{subcaption} 
\usepackage{float} 

\usepackage[most]{tcolorbox}
\usepackage{xcolor}

\usepackage{tikz}
\usetikzlibrary{shapes.geometric, arrows,matrix}

\usepackage{hyperref}

\theoremstyle{definition}
\newtheorem{defn}{Definition}
\newtheorem{teo}{Theorem}
\newtheorem{prop}{Proposition}
\newtheorem{lema}{Lemma}

\newtheorem{obs}{Remark}
\newtheorem{assum}{Assumption}

\newtheorem{claim}{Claim}
\newtheorem*{proofclaim}{Proof of the Claim}
\newcommand{\claimqed}{\hfill \ensuremath{\blacksquare}}
\newtheorem{theorem}{Theorem}

\newtheorem{definition}{Definition}

\newcommand{\E}{\mathbb{E}}
\renewcommand{\P}{\mathbb{P}}

\newcommand{\V}{\mathbb{V}}
\newcommand{\1}{\mathds{1}}
\DeclareMathOperator{\tr}{tr}

\DeclareMathOperator*{\argmax}{arg\,max}

\DeclarePairedDelimiter\floor{\lfloor}{\rfloor}
\newcommand{\block}[1]{
\underbrace{\begin{matrix}\underbrace{1}_{\text{$sm$ times}} & -1 & 1 & \cdots & 1\end{matrix}}_{#1}
}

\numberwithin{equation}{section}

\begin{document}

\title[Counting communities in weighted Stochastic Block Models via semidefinite programming]{Counting communities in weighted Stochastic Block Models via semidefinite programming}

\author{Deborah Oliveira}
\address{IMPA, Rio de Janeiro, Brazil. }
\email{deborah.oliveira@impa.br}

\author{Andressa Cerqueira}
\address{UFSCar, S\~{a}o Carlos, Brazil.}
\email{acerqueira@ufscar.br}

\author{Roberto Oliveira}
\address{IMPA, Rio de Janeiro, Brazil. }
\email{deborah.oliveira@impa.br}

\date{\today}

\begin{abstract}
We consider the problem of estimating the number of communities in a weighted balanced Stochastic Block Model. We construct hypothesis tests based on semidefinite programming and with a statistic coming from a GOE matrix to distinguish between any two candidate numbers of communities. This is possible due to a universality result for a semidefinite programming-based function that we also prove. The tests are then used to form a sequential test to estimate the number of communities. Furthermore, we also construct estimators of the communities themselves.
\end{abstract}
\maketitle

\section{Introduction}\label{sec:intro}
\vspace{0.1in}
\begin{center}
    {\sf First draft. Comments are welcome.}
\end{center}
\vspace{0.1in}
The Stochastic Block Model (SBM), introduced in \cite{holland1983stochastic},  is a widely used probabilistic framework for modeling networks with community structure. It generates networks by randomly assigning vertices to communities based on specified probabilities and then creating edges between them according to defined connection probabilities, which reflect how likely it is for vertices to connect within or between communities.

Community detection is arguably one of the most studied problems related to the SBM, which involves recovering the community assignments of the nodes that generated the observed network. 
Classical maximum likelihood methods are hard to apply to the SBM, as optimizing the likelihood is NP-hard, making it intractable in practice. Consequently, several methods have been proposed to recover the communities, such as, likelihood methods based on variational approximations \cite{daudin2008mixture}, MCMC methods \cite{Snijders2001,handcock2007model}, pseudo-likelihood approach \cite{amini2013pseudo}, belief propagation algorithms \cite{Decelle2011belief}, Bayesian methods \cite{VanderPas2018,Priebe2016bayes}, spectral methods \cite{Rohe2011spectral,Rohe2013spectral,lei2015consistency}, and semidefinite programming (SDP) relaxations \cite{montanari, perry2017semidefinite, hajek2016achieving}.

Despite its simplicity, SBM presents many interesting theoretical properties, in particular in terms of phase transitions and sharp boundaries on what is theoretically achievable in community detection. A comprehensive review covering these phenomena is provided in \cite{abbe2018community}. In practice, however, SBM can be restrictive when modeling complex real-world network data. To address these limitations, extensions of the SBM have been proposed to change the probabilistic mechanism for assigning communities to the nodes, including degree correction to the nodes \cite{karrer2011stochastic}, as well as allowing mixed membership \cite{airoldi2008mixed} and overlapping communities \cite{Latouche2011overlap}.

In this work, we consider the case of a weighted Stochastic Block Model \cite{aicher2015learning}, that is, attached to each present edge we have a random real number coming from a given probability distribution. The motivation for considering the weighted case comes from the fact that in real-world network data, sometimes it is possible to observe not only the existence of connections between nodes but also the strength of those connections \cite{cerqueira2023,xu2018optimal}.

We consider here not only the problem of estimating community assignments but also, and with more emphasis, that of estimating the number of communities in the network. Even though, in practice, it is more reasonable to assume that the true number of communities in a network is unknown, the problem of estimating this number has not been studied as extensively as the community detection problem. Some recent works in the literature on this topic include \cite{lei2016_test, chen2018network, le2022estimating, wang2017likelihood}. In this work, for both problems, we make use of the mentioned SDP relaxation. The use is justified by the fact that our main interest will be the more difficult case of sparse random graphs, and in this case, SDP-based methods have been shown to be successful compared with others, such as the spectral one \cite{montanari}.

Finally, we also deal with a balanced planted partition model, i.e., all communities have the same size. Moreover, as already mentioned, we consider a weighted network. Here, the edge weights among vertices within the same community are drawn from probability distribution with the same mean, as well as the edge weights between vertices in different communities having distributions with the same mean. We assume that all weight distributions, when centered, are \textit{sub-gamma}. Specifically, our first assumption regarding the distribution of the weights of the entries of the weighted adjacency matrix is that they satisfy the following condition.

\begin{assum}\label{assum:1}
There exist $c_{ij},\nu_{ij}>0$ such that for every $p\in \mathbb{N}\setminus\{1\}$ we have
\begin{equation}
    \E[|W_{Gij}-\E W_{Gij}|^{p}] \leq \frac{c_{ij}^{p-2}\nu_{ij}p!}{2}
\end{equation}
and, consequently, we obtain $\V[W_{ij}-\E W_{Gij}] \leq \nu_{ij}$ choosing $p=2$. In particular, this implies we have a sub-gamma distribution given by $({W}_{Gij}-\E[W_{Gij}])\sim sub\Gamma(\nu_{ij},c_{ij})$ by Theorem $2.12$ in \cite{bousquet2002concentration}. The $\nu$ parameter is called the \textbf{variance proxy} and $c$ is a \textbf{scale parameter}.
\end{assum}

This assumption is fairly general. It includes the case of weights given by a Bernoulli random variable times an independent Gaussian one, that is, the usual case of binary edges, but with a Gaussian weight attached to the edge when it is present. We show this in full detail in Section \ref{sec:application}, where we apply our main results to this representative case.

Besides the previous assumption, we also consider, throughout the text, the following two other assumptions.

\begin{assum}\label{assum:2} Let $G$ be a random graph with $n$ nodes that are divided into $K$ communities $\mathcal{C}_1,\dots,\mathcal{C}_K \subset [n]$  such that $|\mathcal{C}_i|=n/K$, $i=1,\dots,K$ and \textbf{$K$ does not depend on $n$}. We consider the weighted adjacency matrix $W_{G}$ of $G$, with the expected weight between nodes
$i$ and $j$, conditionally on the communities assignments of the nodes, given by
\begin{equation}
    \E[W_{G{ij}}] = \begin{cases}
        M_{in}, \quad \text{if  $i \sim j$}\\
        M_{out}, \quad \text{if  $i \not\sim j$}
    \end{cases}
\end{equation}
where the relation $i \sim j$ means that vertices $i$ and $j$ are in the same community. Finally, we also consider that $M_{in}> M_{out}$, that is, we are in a \textbf{mean-separated model}.
\end{assum} 
\begin{assum} \label{assum:3}
    There exists $w_{+} \in (4,n\cdot constant)$ with $w_{+}\xrightarrow{n\to \infty}\infty$ such that
    \begin{equation}
        n\max_{1\leq i\leq j \leq n}\nu_{ij} \leq w_{+}
    \end{equation}
and 
\begin{equation}
     \theta \coloneqq \frac{\max_{1\leq i \leq j \leq n}c_{ij}}{\sqrt{w_{+}}} \xrightarrow{n\to \infty} 0.
\end{equation}
This assumption states how large the variance proxy and scale parameters are in relation to each other and turns out to be a sufficient condition for the theorems we prove.
\end{assum}

The goal of community detection is to recover the communities $\mathcal{C}_1,\dots,\mathcal{C}_K$ from a single realization of the random graph $G$, or equivalently, $W_G$. For this purpose, we divided the problem into two cases: $K=2$ and $K>2$. When $K = 2$, the community detection problem is equivalent to estimating the community membership vector $x_0 \in \{-1, 1\}^n$ where $x_{0i} = x_{0j}$ if $i \sim j$. When $K>2$, we want to estimate the entries of the community membership matrix $Z_{0}\in \{0,1\}^{n\times n}$, with $Z_{0ij}=1$ if $i \sim j$ and $Z_{0ij}=0$ if $i \not\sim j$. Finally, the problem of estimating the true number of communities in our case is the one of estimating the number $K$.

\subsection{Main results}

We present here the main results of the paper. Before that, we need to define the SDP-based function introduced in \cite{montanari} and used on the statement of some of the results.

\begin{definition} Let $M \in \mathbb{R}^{n \times n}$ be any real-valued matrix, we define
\begin{equation}\label{def:SDP}
    \text{SDP}(M) \coloneqq \max{\{\langle M,X \rangle : X \succeq 0, X_{ii}=1 \, \forall i \in [n]\}}
\end{equation}
and denote the subset where the maximum is taken by $\text{PSD}_{1}(n)$.
\end{definition}

An important tool used throughout the text is a universality type result for this function. This result is is fully stated in Theorem \ref{Theorem1} in the following section \ref{sec:semidef} and proved in Appendix \hyperref[sec:appA]{A}.

The first main result concerns the construction of a generalization of the hypothesis test from \cite{montanari} to weighted graphs. In subsection \hyperref[subsection:3.1]{$3.1$}, we consider the case of distinguishing between one (homogeneous Erdös-Renyi graph) and two communities. Let $\P_{0}$ be the probability under the null hypothesis H$_{0}$ that there is one community and $\P_{1}$ be the probability under hypothesis H$_{1}$ that there are two communities. Our null hypothesis has the same meaning as the one from \cite{montanari}, that is, we suppose we have a equal-sized mixture of the distributions from the edges within and between communities such that we still have a sub-gamma distribution when we do this (Proposition $5.1$ in \cite{zhang2020concentration}). Given a $\delta>0$, for a matrix $X \in \mathbb{R}^{n \times n}$, we define the test
\begin{equation}
    T(X;\delta) \coloneqq \mathbb{1}_{\text{SDP}(X-M_{0})> 2n(1 +\delta)\sqrt{w_{+}}}
\end{equation}
where $M_{0}$ is $\E[W_{G}]$ under H$_{0}$ and $w_{+}$ is the number defined in \ref{assum:3}. We have then $T(X;\delta) = 1$ corresponding to rejecting the null hypothesis and $ T(X;\delta) =0$ to not rejecting it. The following result is stated as Theorem \hyperref[Theorem2]{2} later.
\begin{theorem}
In the setting above, if
\begin{equation}\label{eq:29}
    \frac{n^2}{2}(M_{in}-M_{out})>4n(1 +\delta)\sqrt{w_{+}}
\end{equation}
the Type I and Type II errors of the test  $T(X;\delta)$ approach zero asymptotically.
\end{theorem}

When we consider the problem of estimating the communities, we want to estimate the ground truth vector denoted by $x_{0} \in \{-1,+1\}^{n}$, \textit{i.e.}, the vector representing the communities membership. If we could estimate it perfectly with a vector $\hat{v}$, we would have $\langle \hat{v},x_{0}\rangle =n$. Then, the quality of our estimation can be measured by how close \( n - \langle \hat{v}, x_0 \rangle \) is to zero. The following theorem, corresponding to Theorem \ref{Theorem3}, summarizes the result we found.
\begin{theorem}
Let $\hat{v}$ be the estimator constructed using SDP and define an error term as
\begin{equation}
    Err(W_{G}) \coloneqq \sup\{\tr((W_{G}-\E[W_{G}])X): \, X\in \text{PSD}_{1}(n)\}
\end{equation}
then, under the hypothesis that
\begin{equation}
    Err(W_{G})\leq \frac{(M_{in}-M_{out})n^{2}}{8}
\end{equation}
we obtain with high probability 
\begin{equation}
   n- |\langle \hat{v},x_{0}\rangle| \leq \frac{64}{M_{in}-M_{out}}\sqrt{w_{+}} + o(\sqrt{w_{+}})
\end{equation}
then achieving a partial recovery of the communities as, by assumption, $w_{+}=o(n)$.
\end{theorem}

In Subsection \hyperref[subsection:3.2]{$3.2$}, we further generalize to the case of distinguishing between $r$ and $s$ communities where $r>s\geq 2$. We define the tests
\begin{equation}
    T(X;\delta) \coloneqq \mathbb{1}_{\text{SDP}(X-M_{s})> 2n(1 +\delta)\sqrt{w_{+}}}
\end{equation}
where $M_{s}$ is $\E[W_{G}]$ under the hypothesis $H_{s}$ that we have $s$ communities.

Analogously to the previous result we obtain the following for this case, stated as Theorem \ref{Theorem4} in Subsection \hyperref[subsection:3.2.2]{$3.2.2$}.
\begin{theorem}
In the previous setting, if
\begin{equation}\label{eq:175}
    \frac{2n^2}{r^2s^2}(M_{in}-M_{out}) > 4n(1 +\delta)\sqrt{w_{+}}
\end{equation}
the Type I and Type II errors of the test  $T_{n,s}(X;\delta)$ approach zero asymptotically.
\end{theorem}

The main difficulty of this generalization will lie on the calculation of a lower bound for $\text{SDP}(M_{r}-M_{s})$. This requires a careful combinatorial analysis of the relative structure of communities and is done in Appendix \hyperref[sec:appB]{B}.

The estimation procedure for the communities is now done with the membership matrix $Z_{0}\in \{0,1\}^{n \times n}$ defined as
\begin{equation}
        Z_{0 ij} = \begin{cases}
            1, \, i\sim j\\ 
            0, \, i\not\sim j.
        \end{cases}
\end{equation}
instead of a ground truth vector. 

In this case, we can use as a measure of error the count of how many entries from the membership matrix our estimator $\hat{Z_{0}}$ got wrong, that is
\begin{equation}
     e(\hat{Z}_{0}) \coloneqq \sum_{ij}|\hat{Z}_{0ij} - Z_{0ij}| 
\end{equation}
Again, in the following proposition, we obtain a result that gives us a partial recovery of the communities and this theorem corresponds to Theorem \ref{Theorem5}.
\begin{theorem}
Let $\hat{Z}_{0}$ be the membership matrix estimator constructed using $\text{SDP}$. We obtain with high probability
\begin{equation}
e(\hat{Z}) \leq \frac{32n\sqrt{w_{+}}}{M_{in}-M_{out}} + o(\sqrt{w_{+}})
\end{equation}
then achieving a partial recovery of the communities, as by assumption, $w_{+}=o(n)$ and now the maximum error is $n^2$ instead of $n$.
\end{theorem}

Finally, we construct a sequential hypothesis test inspired by \cite{lei2016_test} to find the true number of communities. The idea is again to define hypothesis tests based on SDP. An important difference is that now the tests are not only theoretical but can actually be done in practice, that is, we use estimated quantities in the test. For this purpose, for $\varepsilon >0$ and $K_{0}\in \mathbb{N}$, we define a hypothesis test $\hat{T}_{n, K_0}$ as follows
\begin{equation}
    \hat{T}_{n,K_0}(W_{G};\varepsilon) \coloneqq \mathbb{1}\{ \text{SDP}(W_{G} - \hat{M}_{K_0}) > 2n(1+\epsilon)\sqrt{w_{+}}\}.
\end{equation}
with null hypothesis $H_{0,K_{0}}$ and alternative hypothesis $H_{a,K_{0}}$ given by
\begin{equation}
\left\lbrace\begin{array}{c}
    H_{0,K_0} : \#\text{communities} = K_0\\
    H_{a,K_0} : \#\text{communities} > K_0
\end{array} \right. 
\end{equation}
where $\hat{M_{0}}$ is $M_{0}$ estimated with our SDP estimators.

These tests are performed sequentially for $K_0=1,2,\dots$ until failing to reject the null hypothesis, that is, until $H_{0,K^{*}_0}=0$. The sequential test estimator is then defined by
$$\hat K_n = \inf \{ K_0\geq 1 : \hat{T}_{n,K_0}=0\}\,.$$

We show that $\hat K_n$ is consistent, that is,
\begin{equation}\label{eq:consistency}
    \P(\hat K_n = K)\to 1 \text{ as } n\to \infty
\end{equation}
by showing the two following theorems
\begin{theorem} \textbf{Non-underestimation:} $\P(\hat K_n < K)\to 0 \text{ as } n\to \infty$,\end{theorem}

\begin{theorem}
 \textbf{Non-overestimation:} $ \P(\hat K_n > K)\to 0 \text{ as } n\to \infty$ if
\begin{equation}
    M_{in}-M_{out}>\min\left\{\frac{2K^{4}(K+1)^{2}(1+\varepsilon)\sqrt{w_{+}}}{n},\frac{K(K-1)\sqrt{w_{+}}\log(2(K-1))\sqrt{w_{+}}}{n}\right\}.
\end{equation}
where $\varepsilon>0$ comes from the test $\hat{T}_{n,K_0}(W_{G};\varepsilon)$.
\end{theorem}

These theorems are stated as Theorems \ref{Theorem6} and \ref{Theorem7} and their proofs are done in full detail in Section \ref{sec:estimagingK}.

\begin{obs}
    The above results assume that $w_{+}$ is a known quantity. However, we can use in its place an estimator $\hat{w}_{+,n}$ such that $\P(\hat{w}_{+,n}\geq w_{+}) \xrightarrow{n \to \infty} 1$  and the proofs still work.
\end{obs}

\begin{obs}
It is important to emphasize that all previous results also apply to the widely studied case of the unweighted SBM, providing new insights in this setting as well.
\end{obs}

\subsection{Organization of the paper} In Section \ref{sec:semidef}, we present the result concerning the universality of the SDP-based function which will be crucial in the following sections. Section \ref{sec:communitydet} is devoted to the problem of distinguishing between two different number of communities using a hypothesis test and also the one of estimating the communities themselves. It is subdivide such that in Subsection \ref{subsection:3.1} we consider the distinction between a homogeneous Erdös-Renyi graph and a graph with $2$ communities and Subsection \ref{subsection:3.2} is dedicated to the case of distinguishing between $r$ and $s$ communities with $r>s\geq2$. We then get to the main problem of estimating the number of communities by a sequential hypothesis test in Section \ref{sec:estimagingK}. In Section \ref{sec:application}, we do an application of our results to the case of weights given by a Bernoulli times independent Gaussian distribution. Section \ref{sec:simulations} presents some computational simulations of the tests and estimations. Finally, in Section \ref{sec:discussion}, we have some discussion and final remarks. The appendices contain the more extensive calculations, with the proof of the universality result in Appendix \hyperref[sec:appA]{A} and the lower bound for $\text{SDP}(M_{r}-M_{s})$ in Appendix \hyperref[sec:appB]{B}.

\subsection{Acknowledgments} D.O. was supported by by Petrobr\'{a}s CENPES grant 2023/00172-8 and a CAPES fellowship for graduate studies. A.C.  was supported by the S\~{a}o Paulo Research Foundation (FAPESP) grant 2023/05857-9. R.O. was supported by  a ``Bolsa de Produtividade em Pesquisa''~ and a ``Projeto Universal'' (432310/2018-5) from CNPq, Brazil; and by a ``Cientista do Nosso Estado'' grant (E26/200.485/2023) from FAPERJ, Rio de Janeiro, Brazil.

\section{Universality in Semidefinite Programming} \label{sec:semidef}

As previously mentioned, we are interested here in the more challenging case of the sparse regime, where all vertices have bounded average degree. Spectral methods, which use concentration inequalities to say that the adjacency matrix $A$ is close in operator norm to $\E[A]$ and Davis-Kahan theorem to recover the communities labels from the eigenvectors of the observed $A$, have difficulties in this setting \cite{bordenave2015,feige2005spectral,krzakala2013spectral}. This approach fails in the sparse regime because, in this case, $\lVert A- \E[A] \rVert$ is proportional to the highest degree and then concentration may not hold. Changing the approach completely to a semidefinite programming relaxation of the problem has advantages \cite{abbe2015exact,hajek2016achieving}. This is essentially due to the fact that it makes a constraint where the degrees of the community estimator need to be more uniform, allowing for concentration to hold. In fact, in \cite{montanari} it was proved that SDP has also the advantage of being nearly optimal for weak recovery (identify more than $50\%$ of the nodes correctly) in the large degree asymptotic when working with sparse graphs. Our work is a generalization of this result in two ways: we work with the more general case of weighted graphs and we do hypothesis test to distinguishes  between any two different number of communities, instead of only distinguishing between an homogeneous Erdös-Renyi graph and some number of community (as done in subsection $4.3$ in \cite{montanari}).

When the true number of communities is considered to be known, there are a lot of examples of successful use of the SDP approach in community detection with two or more communities, balanced or not, such as \cite{abbe2015exact,agarwal2017multisection,perry2017semidefinite}. In contrast, the problems of distinguishing between hypothesis for the number of communities and learning the true number of communities were not much explored using SDP. The first problem is known as distinguishability and the second is an instance of learnability as denoted in the famous review \cite{abbe2018community}.

In this work, the semidefinite program approach is used in all of the analysis: to construct a useful statistic in the hypothesis tests and as a way to estimate the communities themselves. To construct the tests, we first need to approximate the SDP-based function of the centralized weighted adjacency matrix to the one of a matrix for which we know its asymptotic behavior, in our case, a GOE matrix. Remember that a GOE matrix is defined as a random symmetric matrix with independent centered Gaussian entries on and above the diagonal with variance on entry $ij$ given by $\frac{1}{n}(1+\delta_{ij})$ with $\delta$ being the Kronecker symbol.

The theorem of this section is a step towards this goal: here we approximate the SDP function of a symmetric matrix to the corresponding Gaussian matrix. Then, we can use part of the proof to also approximate this symmetric matrix to a corresponding GOE matrix, as GOE is a special type of matrix with Gaussian entries. 

As mentioned on the introduction, we will only consider the case of sub-gamma weights. With that, we state the following theorem considering a symmetric matrix with centralized entries which have sub-gamma distribution.

\begin{teo}\label{Theorem1}
Let $X \in \mathbb{R}^{n\times n}$ be a symmetric random matrix with independent centered entries $X_{ij} \sim sub\Gamma(\nu_{ij},c_{ij})$  with variance $\sigma_{ij}^2$ and let $D \in \mathbb{R}^{n \times n}$ be independent of $X$ and the analogous Gaussian matrix, \textit{i.e.}, a symmetric random matrix with independent entries $D_{ij}$ of distribution $\mathcal{N}(0,\sigma_{ij}^{2})$. Then, there exist $C>0$ such that for $8\leq k \leq n$, $\epsilon \in (0,\frac{1}{2})$, $\delta \in (0,1)$ and $\beta>0$, with probability at least $1-\max(4^{-n},n^{-\delta})$,  we have
\begin{equation}\label{eq:52}
\begin{split}
    \frac{1}{n}|\text{SDP}(X) - \text{SDP}(D)| &\lesssim \left[\frac{1}{k-1}+\epsilon\right]\left(\sqrt{\sum_{i\leq j}\frac{\nu_{ij}}{n}} + c\right) + \frac{k-1}{2\beta}\log{\left(\frac{C}{\epsilon}\right)}\\
    &+ \sqrt{\sum_{i\leq j}\nu_{ij}n^{\delta-2}} + \frac{\beta^{2}}{n}\sum_{i\leq j}(c_{ij}\nu_{ij}+\nu_{ij}^{\frac{3}{2}})
\end{split}
\end{equation}
where $c \coloneqq \max_{1\leq i \leq j}c_{ij}$ and $\lesssim$ means that the inequality is valid up to a multiplicative universal constant.
\end{teo} 

The proof is done in full details in Appendix \hyperref[sec:appA]{A}. In our case, we choose $k,\beta, \epsilon$ and $\delta$ in such a way that the bound is or order $o(\sqrt{w_{+}})$. This is exactly what we need to construct the consistent hypothesis tests as we will see on the next section.

Let us now state how exactly this theorem will be used in our results. We need to define the following two functions before proceeding (defined also in \cite{montanari}).

\begin{defn}Let $M \in \mathbb{R}^{n \times n}$, and $k\in \mathbb{N}$, we define
\begin{equation}
    \text{OPT$_{k}$}(M) := \max{\{\langle M,X \rangle : X \succeq 0,  X_{ii}=1 \forall i \in [n], \text{rank}(X)\leq k\}}.
\end{equation}
\end{defn}

\begin{defn} Let $M \in \mathbb{R}^{n \times n}$, $\beta > 0$ and $k \in \mathbb{N}$, we define
\begin{equation}
    \Phi(\beta,k;M) := \frac{1}{\beta} \log{\left(\int\exp\left(\beta \sum_{i,j=1}^{n} M_{ij} \langle \sigma_{i},\sigma_{j}\rangle\right)d\nu(\mathbf{\sigma})\right)}
 \end{equation}
where $\sigma=(\sigma_1, \sigma_2, ..., \sigma_n) \in (\mathbb{S}^{k-1})^{n}$ and $d\nu$ is the uniform and normalized measure on $(\mathbb{S}^{k-1})^{n}$.
\end{defn}

\begin{obs}
    We observe that $\text{OPT$_{k}$}(M)$ can be seen as a low rank version of the SDP function and $\Phi(\beta,k;M)$ as a smooth version of it inspired by statistical mechanics. More discussion about these functions can be found in \cite{montanari}.
\end{obs}

First, for comparison, we have below the approximation scheme of SDP given by Montanari and Sen's paper \cite{montanari} using their lemmas and theorems:

\begin{center}
\begin{tikzpicture}
    \matrix(m)[matrix of math nodes, column sep=3em, row sep=2em]
    {
    \text{SDP}(\bar{A}_{G}) & \text{OPT}_{k}(\bar{A}_{G}) & \Phi(\beta,k;\bar{A}_{G}) & \E[\Phi(\beta,k;\bar{A}_{G})] \\
    \text{SDP}(B) &  \text{OPT}_{k}(B) & \Phi(\beta,k;B) & \E[\Phi(\beta,k;B)] \\
    };
    \node at ($(m-1-1)!0.5!(m-1-2)$) {$\approx$};
    \node[above=2pt] at ($(m-1-1)!0.5!(m-1-2)$) {\tiny Teo. 4};

    \node at ($(m-1-2)!0.5!(m-1-3)$) {$\approx$};
    \node[above=2pt] at ($(m-1-2)!0.5!(m-1-3)$) {\tiny Lem. 3.3 $+$ A.3};

    \node at ($(m-1-3)!0.5!(m-1-4)$) {$\approx$};
    \node[above=2pt] at ($(m-1-3)!0.5!(m-1-4)$) {\tiny Lem. A.1};

    \node at ($(m-2-1)!0.5!(m-2-2)$) {$\approx$};
    \node[above=2pt] at ($(m-2-1)!0.5!(m-2-2)$) {\tiny Teo. 4};

    \node at ($(m-2-2)!0.5!(m-2-3)$) {$\approx$};
    \node[above=2pt] at ($(m-2-2)!0.5!(m-2-3)$) {\tiny Lem. 3.3 $+$ A.3};

    \node at ($(m-2-3)!0.5!(m-2-4)$) {$\approx$};
    \node[above=2pt] at ($(m-2-3)!0.5!(m-2-4)$) {\tiny Lem. A.2};

    \node at ($(m-1-4)!0.5!(m-2-4)$) {$\approx$};
    \node[right=2pt] at ($(m-1-4)!0.5!(m-2-4)$) {\tiny Lem. 3.3 (E.1 $+$ E.2)};
\end{tikzpicture}
\end{center}
where $\bar{Y}\coloneqq Y-\E Y$ for any random variable $Y$, $A_{G}$ is the adjacency matrix in their case and $B$ their corresponding GOE matrix.

For our case, we have the following corespondent lemmas in Appendix \hyperref[sec:appA]{A} of this paper:

\begin{center}
\begin{tikzpicture}
    \matrix(m)[matrix of math nodes, column sep=3em, row sep=2em]
    {
    \text{SDP}(\bar{W}_{G}) & \text{OPT}_{k}(\bar{W}_{G}) & \Phi(\beta,k;\bar{W}_{G}) & \E[\Phi(\beta,k;\bar{W}_{G})] \\
    \text{SDP}(\sqrt{w_{+}}B) &  \text{OPT}_{k}(\sqrt{w_{+}}B) & \Phi(\beta,k;\sqrt{w_{+}}B) & \E[\Phi(\beta,k;\sqrt{w_{+}}B)] \\
    };
    \node at ($(m-1-1)!0.5!(m-1-2)$) {$\approx$};
    \node[above=2pt] at ($(m-1-1)!0.5!(m-1-2)$) {\tiny Lem. 1};

    \node at ($(m-1-2)!0.5!(m-1-3)$) {$\approx$};
    \node[above=2pt] at ($(m-1-2)!0.5!(m-1-3)$) {\tiny Lem. 2};

    \node at ($(m-1-3)!0.5!(m-1-4)$) {$\approx$};
    \node[above=2pt] at ($(m-1-3)!0.5!(m-1-4)$) {\tiny Lem. 3};

    \node at ($(m-2-1)!0.5!(m-2-2)$) {$\approx$};
    \node[above=2pt] at ($(m-2-1)!0.5!(m-2-2)$) {\tiny Lem. 1};

    \node at ($(m-2-2)!0.5!(m-2-3)$) {$\approx$};
    \node[above=2pt] at ($(m-2-2)!0.5!(m-2-3)$) {\tiny Lem. 2};

    \node at ($(m-2-3)!0.5!(m-2-4)$) {$\approx$};
    \node[above=2pt] at ($(m-2-3)!0.5!(m-2-4)$) {\tiny Lem. 3};

    \node at ($(m-1-4)!0.5!(m-2-4)$) {$\approx$};
    \node[right=2pt] at ($(m-1-4)!0.5!(m-2-4)$) {\tiny Lemma 4 + $\E[\Phi(\beta,k;D)] + $Prop. 1};
\end{tikzpicture}
\end{center}
where $w_{+}$ is the variance factor defined in Assumption \ref{assum:3} and Proposition \ref{prop:1} is given below. 

In the proof of the universality result in Appendix \hyperref[sec:appA]{A} we stated the Lemmas considering we had the Gaussian equivalent matrix $D$ all along. However, as seen above, we actually only approximate by the corresponding $D$ matrix in Lemma \ref{lemma:4}. Then, we use Proposition \ref{prop:1} below to approximate the Gaussian matrices $D$ and $\sqrt{w_{+}}B$ and use $\sqrt{w_{+}}B$ on the other half of the approximation. This is valid, as will be seen on the proofs, because all matrices -- $\bar{W}_{G}, D$, and  $\sqrt{w_{+}}B$ -- have sub-gamma distribution and we simply need to appropriately adjust the variance proxy, i.e., the parameter $\nu$, for each matrix.

We now state the proposition that connects matrix $D$ and $\sqrt{w_{+}}B$.

\begin{prop}\label{prop:1}
    Let $G_{1},G_{2}$ be symmetric random matrices with centered Gaussian entries which are independent up to the symmetry condition. Moreover, we assume $\V[G_{1ij}]\leq \V[G_{2ij}]$ for all $1\leq i \leq j \leq n$. 

    Then, for all convex functions $f:\mathbb{R}^{n\times n}\to \mathbb{R}$ we obtain
    \begin{equation}
        \E[f(G_{1})] \leq \E[f(G_{2})].
    \end{equation}
\end{prop}

\begin{proof}
    Let $G_{3}$ be a symmetric random matrix with independent centered Gaussian entries (up to the symmetry condition), independent of $G_{1}$ and with $\V[G_{3ij}]=\V[G_{2ij}] - \V[G_{1ij}]$. Then, by properties of sum of independent Gaussian variables we obtain in distribution $G_{3}+G_{1} \sim G_{2}$ such that
    \begin{equation}
        \E[f(G_{1})] =\E[f(\E[G_{1}+G_{3}|G_{1}])] \leq \E[\E[f(G_{1}+G_{3})|G_{1}]] = \E[f(G_{1}+G_{3})] = \E[f(G_{2})]
    \end{equation}
    where we used the conditional Jensen's inequality.
\end{proof}

In our case, we use Proposition \ref{prop:1} with the convex function $\Phi(\beta,k;\cdot): \mathbb{R}^{n\times n} \to \mathbb{R}$ above to obtain
\begin{equation}\label{eq:24}
   \E[\Phi(\beta,k;D)] \leq \E[\Phi(\beta,k;\sqrt{w_{+}}B)] 
\end{equation}
because by Assumptions \ref{assum:1} and \ref{assum:3}  we have $\sigma_{ij}^{2} = \V[D_{ij}]\leq \nu_{ij} \leq \frac{w_{+}}{n} = \V[\sqrt{w_{+}}B]$ for all $1\leq i \leq j \leq n$.

To make it clear, we use the above scheme in the following way
\begin{equation}\label{eq:scheme}
\begin{split}
    \text{SDP}(\bar{W}_{G}) - \text{SDP}(\sqrt{w_{+}}B) &= \text{SDP}(\bar{W}_{G}) - \text{OPT}_{k}(\bar{W}_{G})\\
    &+ \text{OPT}_{k}(\bar{W}_{G}) - \Phi(\beta,k;\bar{W}_{G})\\
    &+ \Phi(\beta,k;\bar{W}_{G}) - \E[\Phi(\beta,k;\bar{W}_{G})]\\
    &+ \E[\Phi(\beta,k;\bar{W}_{G})] - \E[\Phi(\beta,k;D)]\\
    &+ \E[\Phi(\beta,k;D)] - \E[\Phi(\beta,k;\sqrt{w_{+}}B)]\\
    &+\E[\Phi(\beta,k;\sqrt{w_{+}}B)] - \Phi(\beta,k;\sqrt{w_{+}}B)\\
    &+\Phi(\beta,k;\sqrt{w_{+}}B) - \text{OPT}_{k}(\sqrt{w_{+}}B)\\
    &+\text{OPT}_{k}(\sqrt{w_{+}}B) - \text{SDP}(\sqrt{w_{+}}B) \leq n\Lambda
\end{split}
\end{equation}
where $\E[\Phi(\beta,k;D)] - \E[\Phi(\beta,k;\sqrt{w_{+}}B)]$ is negative by \ref{eq:24} and we denote
\begin{equation}\label{eq:Lambda}
    \Lambda \coloneqq \frac{RHS-(\E[\Phi(\beta,k;D)] - \E[\Phi(\beta,k;\sqrt{w_{+}}B)])}{n}.
\end{equation}
where $RHS$ is the right-hand side of the equality in \ref{eq:scheme}.

With that, we are approximating the SDP-based function of the weighted centered adjacency matrix $\bar{W}_{G}$ with the corresponding GOE matrix, in this case $\sqrt{w_{+}}B$. We then use this approximation to construct the hypothesis test and estimation procedures for the communities on the next section.

\section{Community detection}\label{sec:communitydet}

In this section, we consider the problem of community detection, that is, given a single realization of the random graph modeled by SBM, we want to recover with high probability the communities. Before we consider this, we deal with the problem of distinguishing between two hypothesis for the number of communities. We subdivide this problem and the estimation of communities into two cases: distinguishing between $1$ and $2$ communities and distinguishing between $r$ and $s$ communities with $r>s\geq 2$. 

In \cite{montanari} it was shown that it is possible to construct a hypothesis test to distinguish between a homogeneous Erdös-Renyi graph and a balanced SBM with $2$ communities. They also showed that their test achieves the information theoretical optimal threshold \cite{mossel2012stochastic, mossel2018proof, massoulie2014community}, that is, they can estimate the communities in the large degree asymptotic (see Remark $1.4$ there) when for some $\epsilon>0$
\begin{equation}
    \frac{a-b}{2\sqrt{a+b}} \geq 1 + \epsilon
\end{equation}
where the connectivity probability inside communities is $\frac{a}{n}$, outside is $\frac{b}{n}$, and they are considering the usual setting of  Bernoulli weights. In our case of more general weighted graphs, as will be shown below, we obtain an asymptotically good estimation with a similar condition but with the probabilities being substituted by the means of the weights on the edges and the quantity $w_{+}$ (see equations \ref{eq:29} and \ref{eq:74} below). 

The strategy for the hypothesis tests will be the same as the one in \cite{montanari}: we construct a hypothesis test using a statistic provided by SDP$(W_{G}-\E[W_{G}])$, approximate it by SDP of a GOE matrix with our universality result and use our knowledge of its asymptotical behavior to calculate the error.

\subsection{Case I: Distinguishing between 1 and 2 communities}\label{subsection:3.1}

For the hypothesis tests here, we closely follow the analysis done in \cite{montanari} with  our main contribution being the generalization to the weighted case. However, when estimating the communities themselves, we depart from the analysis done in the same paper and consider a relatively simple one.

\subsubsection{Hypothesis test for Case I}\label{subsection:3.1.1}

We want to know if our random graph $G$ has a community structure. For this, we construct a hypothesis test that takes into account the value of \text{SDP}$(W_{G}-\E[W_{G}])$ where $W_{G}$ is the weighted adjacency matrix. The main idea is to use a statistic from to the approximate GOE matrix given by \ref{Theorem1} to construct the test. 

Our goal is to have the errors of false positive and false negative asymptotically going to zero, that is
\begin{equation}
    \P_{0}(T_{n}(W_{G};\delta)=1)+\P_{1}(T_{n}(W_{G};\delta)=0) \xrightarrow{n \to \infty} 0
\end{equation}
where $\P_{0}$ is the probability under hypothesis H$_{0}$: there is no community, and $\P_{1}$ is the probability under hypothesis H$_{1}$: there are two communities, and the test $T_{n}(\cdot; \delta)$ for $\delta>0$ is defined as 
\begin{equation}
    T(X;\delta) \coloneqq \mathbb{1}_{\text{SDP}(X-M_{0})> 2n(1 +\delta)\sqrt{w_{+}}}
\end{equation}
where  $M_{0}$ is $\E[W_{G}]$ under H$_{0}$ and $w_{+}$ is the variance factor define in Assumption \ref{assum:3}. The test accepts the null hypothesis $H_{0}$ if $T(X;\delta)=0$ and rejects it otherwise. 
 
This choice for the test is justified by Theorem $5$ of \cite{montanari} with $\lambda=0$. In this case,for all $\delta>0$ we get
\begin{equation}\label{GOEtheorem}
\P\left(\text{SDP}\left(\sqrt{w_{+}}\text{GOE}\right)\leq 2n(1 +\delta)\sqrt{w_{+}}\right) \xrightarrow{n \to \infty} 1
\end{equation}
for $\P=\P_{0}$ or $\P_{1}$ in our case. The mentioned theorem in \cite{montanari} is a rank $1$ version of a known result about finite rank perturbations of GOE matrices, the Baik-Ben Arous-Peché (BBAP) phase transition \cite{baik2005phase}. Here, with their $\lambda=0$ we do not have any perturbation, just the known semicircle distribution of eigenvalues of a GOE matrix. 

The idea, as is always the case for hypothesis tests, is to have a test statistic that is independent of the hypothesis that truly holds and we use the GOE matrix asymptotic behavior for this purpose. We see below how this equation appear on the error estimates of the test.

\begin{teo}\label{Theorem2}
For every $\delta>0$ , if we have
\begin{equation}\label{eq:29}
    \frac{n^2}{2}(M_{in}-M_{out})>4n(1 +\delta)\sqrt{w_{+}}
\end{equation}
then
\begin{equation}
    \P_{0}(T_{n}(W_{G};\delta)=1)+\P_{1}(T_{n}(W_{G};\delta)=0) \xrightarrow{n \to \infty} 0.
\end{equation}
\end{teo}

\begin{proof} 
Calculating each error separately we obtain the following.

\underline{\textbf{Type I Error:}} Under hypothesis H$_{0}$ we have
\begin{equation}
    \P_{0}(T_{n}(W_{G};\delta)=1) = \P_{0}(\text{SDP}(W_{G}-M_{0})> 2n(1 +\delta)\sqrt{w_{+}})
\end{equation}
Defining the events
\begin{equation}
    \begin{split}
        E' &= \{\text{SDP}(W_{G}-M_{0})>2n(1+\delta)\sqrt{w_{+}}\}\\
        E'' &= \{\text{SDP}(W_{G}-M_{0}) \leq \text{SDP}(\sqrt{w_{+}}B)+n\Lambda_{0}\}
    \end{split}
\end{equation}
where $\text{SDP}(W_{G}-M_{0}) - \text{SDP}(\sqrt{w_{+}}B) \leq n\Lambda_{0}$ and $\Lambda_{0}$ is $\Lambda$ in equation \ref{eq:Lambda} in the case where the mean of $W_{G}$ is $M_{0}$. That is, it is given by
\begin{equation}
\begin{split}
    \Lambda_{0} &= \left[\frac{1}{k-1} + \epsilon \right]\left(\sqrt{\sum_{i\leq j}\frac{\left(\nu_{ij}+\frac{w_{+}}{n}\right)}{n}} + c \right) + \frac{k-1}{\beta}\log{\left(\frac{C}{\epsilon}\right)} + \sqrt{\sum_{1\leq i \leq j \leq n}\left(\nu_{ij}+\frac{w_{+}}{n}\right)n^{\delta-2}} + \frac{\beta^{2}}{n}\sum_{1\leq i \leq j \leq n}(c\nu_{ij}+\nu_{ij}^{\frac{3}{2}})   \\
    &\lesssim \left[\frac{1}{k-1} + \epsilon \right]\left(\sqrt{w_{+}} + \theta\sqrt{w_{+}} \right) + \frac{k-1}{\beta}\log{\left(\frac{C}{\epsilon}\right)} + \beta^{2}\left(\theta w_{+}^{\frac{3}{2}}+ \frac{w_{+}^{\frac{3}{2}}}{n^{\frac{1}{2}}}\right) \\
    &\lesssim [\rho^{\frac{1}{8}} + \rho^{\frac{1}{8}}(\log{(C)}+\log{\rho^{-1}}) + \rho^{\frac{1}{2}}]\sqrt{w_{+}} \coloneqq \lambda_{0}
\end{split}
\end{equation}
where we made the choices
\begin{equation}
\begin{cases} 
    \beta = \frac{1}{\sqrt{w_{+}}\rho^{\frac{1}{4}}}\\
    \epsilon = \min\{\rho,\frac{1}{2}\}\\
    \delta = \frac{1}{2}\\
    k = \max\{8,\lfloor \rho^{-\frac{1}{8}}\rfloor\}
\end{cases}
\end{equation}
with $\rho = \max\{ \theta, n^{-8}\}\}$ to obey $k\leq n$ such that
\begin{equation}\label{eq:lambda}
    \lambda_{0} = o(\sqrt{w_{+}}).
\end{equation}
With that, we obtain
\begin{equation}
    E'\cap E'' \subset \{\text{SDP}(\sqrt{w_{+}}B) > 2n(1+\delta)\sqrt{w_{+}}-n\lambda_{0}\}
\end{equation}
such that
\begin{equation}
    \begin{split}
        \P_{0}(E') &\leq \P_{0}(E'\cap E'') + \P_{0}(E''^{c})\\
        &\leq \P_{0}\left(\text{SDP}(\sqrt{w_{+}}B) > \left(2n(1+\delta)-\frac{n\lambda_{0}}{\sqrt{w_{+}}}\right)\sqrt{w_{+}}\right) + \P_{0}(E''^{c}) = o(1)
    \end{split}
\end{equation}
because for $n$ large enough we have $\frac{\lambda_{0}}{\sqrt{w_{+}}}\leq \delta$ and then
\begin{equation}
    \P_{0}\left(\text{SDP}(\sqrt{w_{+}}B) \geq 2n\left(1+\frac{\delta}{2}\right)\sqrt{w_{+}}\right) = o(1).
\end{equation}

\underline{\textbf{Type II Error:}} Under hypothesis H$_{1}$ we have by same previous reasoning
\begin{equation}
\begin{split}
    \P_{1}(T_{n}(W_{G};\delta)=0) &= \P_{1}(\text{SDP}(W_{G}-M_{0})\leq 2n(1+\delta)\sqrt{w_{+}})\\ 
    &\leq \P_{1}(\text{SDP}(M_{1}-M_{0})-\text{SDP}(M_{1}-W_{G})\leq 2n(1+\delta)\sqrt{w_{+}}) \\
    &= \P_{1}(\Delta - 2n(1+\delta)\sqrt{w_{+}} \leq \text{SDP}(M_{1}-W_{G}))\\
    & \leq \P_{1}(2n(1+\delta)\sqrt{w_{+}} \leq \text{SDP}(M_{1}-W_{G}))\\
\end{split}
\end{equation}
where we used $M_{1}$ is $\E[W_{G}]$ under hypothesis H$_{1}$, $\Delta=\text{SDP}(M_{1}-M_{0})=\frac{n^2}{2}(M_{in}-M_{out})$, $M_{in}\geq M_{out}$, $M_{0}=\frac{M_{out}+M_{in}}{2}\mathbb{1}\mathbb{1}^{T}$ and the hypothesis that we have $\Delta > 4n(1+\delta)\sqrt{w_{+}}$. Continuing, we define similarly 
\begin{equation}
    \begin{split}
        F' &= \{\text{SDP}(W_{G}-M_{1})>2n(1+\delta)\sqrt{w_{+}})\}\\
        F'' &= \{\text{SDP}(W_{G}-M_{1}) \leq \text{SDP}(\sqrt{w_{+}}B)+n\Lambda_{1}\}
    \end{split}
\end{equation}
where $\Lambda_{1}$ is $\Lambda$ in equation \ref{eq:scheme} in the case where the mean of $W_{G}$ is $M_{1}$. The proof ends with the reasoning as in the case of Type I error.
\end{proof}

\subsubsection{Community estimation for case I}\label{subsection:3.1.2}
As we are in the simple case of two communities, we can use a binary estimator of $n$ coordinates to detect the communities. We define a ground truth vector $x_{0} \in \{+1,-1\}^{n}$ and proceed to estimate it.

First, we observe that we can write
\begin{equation}
    \E[W_{G}] = \frac{M_{in}+M_{out}}{2}\mathbb{1}\mathbb{1}^{T} + \frac{M_{in}-M_{out}}{2}x_{0}x_{0}^{T}
\end{equation}
where $\mathbb{1} \in \mathbb{R}^{n}$ is the vector with all entries equal to one and $I \in \mathbb{R}^{n\times n}$ is the identity matrix.

Then we observe that $\E[W_{G}]$ does not change when we invert the sign of $x_{0}$. We obtain that our goal is to estimate $x_{0}$ up to a sign. Observe that $\forall x \in \{+1,-1\}^{n}$ 
\begin{equation}
    -n \leq \langle x, x_{0} \rangle \leq n
\end{equation}
with the bound being attained when $x=\pm x_{0}$. In that way, to have a good estimator we need to obtain a good approximation of $|\langle x,x_{0}\rangle|$ with $n$.

Let us define our ideal optimization problem as
\begin{equation}
    \begin{aligned}
        &\text{maximize} \quad && \tr(\E[W_G] X) \\
        &\text{subject to} \quad && X \in \mathcal{D} =\text{PSD}_1(n)\cap\{\tr{(\mathbb{1}\mathbb{1^{T}}X)=0}\}
    \end{aligned}
\end{equation}
that is, we want to find $X \in\text{PSD}_{1}(n)$ that gives the value of $\text{SDP}(\E[W_{G}])$.

The following proposition shows that the only optimal solution to it is $x_{0}x_{0}^{T}$ and for any other $X$ inside $\text{PSD}_{1}(n)$ it gives a quantification of how good its top eigenvector $v_{1}(X)$ estimates the communities.
\begin{prop}
Let $X\in \text{PSD}_{1}(n)$. Defining
\begin{equation}
    \xi \coloneqq \frac{2}{n^{2}(M_{in}-M_{out})}[\tr(\E[W_{G}]x_{0}x_{0}^{T})-\tr(\E[W_{G}]X)]
\end{equation}    
\end{prop}
we obtain 
\begin{enumerate}
    \item $0 \leq n^2 - \langle x_{0}, Xx_{0} \rangle = \frac{2}{M_{in}-M_{out}}[\tr(\E[W_{G}]x_{0}x_{0}^{T})-\tr(\E[W_{G}]X)]=\xi n^{2}$, in particular, $x_{0}x_{0}^{T}$ is the only solution to the ideal optimization problem.
    \item If $\xi \leq \frac{1}{2}$ then normalizing $\lVert v_{1}(X) \rVert^{2} = n$ gives us 
    \begin{equation}
        n - |\langle v_{1}(X),x_{0}\rangle| \leq 2\xi n.
    \end{equation}
\end{enumerate}

\begin{proof}
Let $X \in\text{PSD}_{1}(n)$ and denote the eigenvalues of X as $\lambda_{1} \geq \lambda_{2} \geq ... \geq \lambda_{n} \geq 0$, then we can write
\begin{equation}
    X = \sum_{i=1}^{n}\lambda_{i}y_{i}y_{i}^{T}
\end{equation}
where $\{y_{i}\}_{i=1}^{n}$ is a orthonormal basis of eigenvectors.
Moreover
\begin{equation}
    \tr(X) = \sum_{i}\lambda_{i} = \sum_{i} X_{ii} = n
\end{equation}
and then 
\begin{equation}\label{eq:45}
    n \geq \lambda_{1} = \max_{\lVert x \rVert =1}\langle x, Xx\rangle \geq \frac{\langle x_{0},Xx_{0}\rangle}{n}
\end{equation}
because $\lVert\frac{1}{\sqrt{n}}x_{0}\rVert = 1$. We have equality if and only if $n=\lambda_{1}$ and $\frac{x_{0}}{\sqrt{n}}$ is eigenvector of $X$ corresponding to $\lambda_{1}$.

Now, observe that 
\begin{equation} \label{eq:48}
\begin{split}
    \tr(\E[W_{G}]X) &= \frac{M_{in}-M_{out}}{2} \langle x_{0}, Xx_{0}\rangle\\
    &\leq \frac{(M_{in}-M_{out})n}{2}\lambda_{1}(X)\\
    &\leq \frac{(M_{in}-M_{out})n^{2}}{2} = \tr(\E[W_{G}]x_{0}x_{0}^{T})
\end{split}
\end{equation}
this implies item $1$ of the proposition and, in particular, the equality is achieved when $X=x_{0}x_{0}^{T}$, \textit{i.e.}, this gives the maximum. 
Observe that
\begin{equation}
    \xi = 1 - \frac{\langle x_{0},Xx_{0}\rangle}{n^{2}}.
\end{equation}
Denoting $p_{i} := \frac{1}{n}\langle y_{i}, x_{0}\rangle^{2}$ we obtain
\begin{equation} 
    \sum_{i}p_{i} = \frac{\lVert x_{0}\rVert^{2}}{n}=1
\end{equation}
because $\{y_{i}\}_{i}$ is a orthonormal basis.

Then, using the spectral decomposition of $X$ we obtain
\begin{equation}
\begin{split}
    \xi n &= n - \frac{\langle x_{0},Xx_{0} \rangle}{n}= n - \sum_{i}\frac{\lambda_{i}\langle x_{0},y_{i}y_{i}^{T}x_{0}\rangle}{n} = n - \sum_{i}\frac{\lambda_{i}\langle y_{i},x_{0}\rangle^{2}}{n}\\
    &= n -\sum_{i}\lambda_{i}p_{i} = n - \lambda_{1}p_{1} - \sum_{i\geq2}\lambda_{i}p_{i} \geq n - \lambda_{1}p_{1} - \sum_{i\geq2}\lambda_{i}\sum_{j\geq2}p_{j}\\
    &= n - \lambda_{1}p_{1} - (n-\lambda_{1})(1-p_{1}) = np_{1} + \lambda_{1} - 2\lambda_{1}p_{1}
\end{split}
\end{equation}
Then
\begin{equation}
    (2\lambda_{1}-n)p_{1} \geq \lambda_{1} - \xi n = (1-\xi)n - (n-\lambda_{1}) \geq (1-2\xi)n
\end{equation}
where in the last inequality we used $\tr(X) - \lambda_{1} = n - \lambda_{1} \leq n\xi$
by equations \ref{eq:48} and \ref{eq:45}. 

Using $\lambda_{1}\leq n$, we deduce
\begin{equation}
    p_{1} \geq \frac{(1-2\xi)n}{2\lambda_{1}-n} \geq \frac{(1-2\xi)n}{2n-n} = (1-2\xi) \implies \langle y_{1},x_{0}\rangle^{2} \geq (1-2\xi)n
\end{equation}
Taking the re-normalized eigenvector $v_{1}(X)=\sqrt{n}y_{1}$ and remembering we are considering $\xi<\frac{1}{2}$ (just to have a positive number inside the below square root) we obtain
\begin{equation}
|\langle v_{1}(X),x_{0}\rangle| = \sqrt{n}|\langle y_{1},x_{0}\rangle| \geq \sqrt{n}\sqrt{(1-2\xi)n} = \sqrt{(1-2\xi)}n
\end{equation}
and then 
\begin{equation}\label{eq:52}
    n - |\langle v_{1}(X),x_{0}\rangle| \leq (1 - \sqrt{1-2\xi})n \leq (1 - \sqrt{1-2\xi})(1 + \sqrt{1-2\xi})n = 2\xi n
\end{equation}
as we wanted.
\end{proof}

As our goal is to discover $\E[W_{G}]$ we obviously do not have access to it to solve the ideal optimization problem. However, we can define the equivalent problem using the observed $W_{G}$.
\begin{equation}
    \begin{aligned}
        &\text{maximize} \quad && \tr([W_G] X) \\
        &\text{subject to} \quad && X \in \mathcal{D} =\text{PSD}_1(n)\cap\{\tr{(\mathbb{1}\mathbb{1^{T}}X)=0}\}
    \end{aligned}
\end{equation}
The following proposition gives us how good is the estimator based on the above optimization problem.
\begin{prop}
    Let $\hat{X}$ be a solution of the above optimization problem and $\hat{v} = v_{1}(\hat{X})$ be its top eigenvector normalized such that $\lVert v_{1}(\hat{X}) \rVert^{2} =n$. Defining the rounding operator as
    \begin{equation}
    \begin{split}
        s: \mathbb{R}^{n} \to \{-1,+1\}^{n}\\
        s(v)_{i} = sign(v_{i})
    \end{split}
    \end{equation}
    and an error term as
    \begin{equation}
        Err(W_{G}) \coloneqq \sup\{\tr((W_{G}-\E[W_{G}])X): \, X\in \mathcal{D})\}
    \end{equation}
    then, if
\begin{equation}
    Err(W_{G})\leq \frac{(M_{in}-M_{out})n^{2}}{8}
\end{equation}
    we obtain
    \begin{equation}
        n - |\langle s(v_{1}(\hat{X})),x_{0}\rangle| \leq \frac{32 Err(W_{G})}{(M_{in}-M_{out})n}
    \end{equation}
\end{prop}
\begin{proof}
As $s(\hat{v})$ and $x_{0}$ have entries in $\{-1,+1\}$ we obtain
\begin{equation}
    (s(\hat{v})_{i}-x_{0i})^{2} = 4\mathbb{1}\{s(\hat{v})_{i}\neq x_{0i}\}.
\end{equation}
Observing that $(\hat{v}_{i} - x_{0i})^{2} \geq 1$ when $\hat{v}_{i} \leq 0$ and $x_{0i}=1$ or when $\hat{v}_{i} \geq 0$ and $x_{0i}=-1$, that is, when $s(\hat{v})_{i}\neq x_{0i}$, we conclude $(\hat{v}_{i}-x_{0i})^{2} \geq \mathbb{1}\{s(\hat{v})_{i}\neq x_{0i}\}$ and then
\begin{equation}
\begin{split}
    (s(\hat{v})_{i}-x_{0i})^{2} &\leq 4(\hat{v}-x_{0i})^{2}\\
    \lVert s(\hat{v}) - x_{0} \rVert^{2} &\leq 4\lVert \hat{v} - x_{0} \rVert^{2}
\end{split}
\end{equation}
using that $\lVert \hat{v} \rVert^{2}=\lVert x_{0}\rVert^{2} = \lVert s(\hat{v}) \rVert^{2}=n$ we obtain
\begin{equation}
\begin{split}
    n - 2\langle s(\hat{v}),x_{0} \rangle +n &\leq 4(n-2\langle \hat{v},x_{0}\rangle +n)\\
    2n-2\langle s(\hat{v}),x_{0}\rangle &\leq 4(2n -2\langle \hat{v}, x_{0}\rangle)\\
    n-\langle s(\hat{v}),x_{0}\rangle &\leq 4(n-\langle \hat{v},x_{0}\rangle)
\end{split}
\end{equation}
Now, by the previous proposition, as $\hat{X}$ belongs to the same set we obtain by \ref{eq:52}
\begin{equation}
    n-\langle s(\hat{v}),x_{0}\rangle \leq 4(n-\langle \hat{v},x_{0}\rangle) \leq 8\xi n
\end{equation}
as long as we have
\begin{equation}
    \xi := \frac{2}{(M_{in}-M_{out})n^{2}}(\tr(\E[W_{G}]x_{0}x_{0}^{T})-\tr(\E[W_{G}]\hat{X})) < \frac{1}{2}.
\end{equation}
Again, as $\hat{X}$ attains the maximum of $\tr(W_{G}X)$ in the viable set and we have
\begin{equation}
    \tr(W_{G}x_{0}x_{0}^{T}) - \tr(W_{G}\hat{X}) \leq 0 \implies \tr(\E[W_{G}]x_{0}x_{0}^{T})-\tr(\E[W_{G}]\hat{X})\leq 2Err(W_{G})
\end{equation}
then
\begin{equation}
    \xi \leq \frac{4Err(W_{G})}{(M_{in}-M_{out})n^{2}}
\end{equation}
by the hypothesis on $Err(W_{G})$ we obtain $\xi < \frac{1}{2}$ and the bound follows
\begin{equation}
   n-|\langle s(\hat{v})),x_{0}\rangle| \leq \frac{32 Err(W_{G})}{n(M_{in}-M_{out})}
\end{equation}
\end{proof}

We observe that $\mathcal{D}\subset \text{PSD}_{1}$ such that the following is valid with high probability by Theorem \ref{Theorem1} and equation \ref{eq:lambda}
\begin{equation}\label{eq:122}
    \textit{Err}(W_{G}) \leq \text{SDP}(W_{G}-\E[W_{G}]) \lesssim \text{SDP}\left(\sqrt{w_{+}}B \right) + o(\sqrt{w_{+}}) \leq 2n(1+\delta)\sqrt{w_{+}} + o(\sqrt{w_{+}})
\end{equation}
for any $\delta>0$. Then, choosing the variances and means such that
\begin{equation}
2n\sqrt{w_{+}} + o(\sqrt{w_{+}}) \leq \frac{(M_{in}-M_{out})n^2}{8} 
\end{equation}
we have what we wanted.

All the results above and the value of $\textit{Err}(W_{G})$ gives us the following theorem.
\begin{teo}\label{Theorem3}
Let $\hat{v}$ be the previous estimator. We obtain asymptotically with high probability 
\begin{equation}
   n- |\langle \hat{v},x_{0}\rangle| \leq \frac{64}{M_{in}-M_{out}}\sqrt{w_{+}} + o(\sqrt{w_{+}})
\end{equation}
such that we obtain a partial recovery of the communities as, by assumption, $w_{+}=o(n)$.
\end{teo}

\subsection{Case II: Distinguishing between r and s communities}\label{subsection:3.2}

Now, we consider the case of trying to distinguish between any two number of communities greater than two. Our main contribution will be to find a lower bound with appropriate order for \text{SDP}$(M_{r}-M_{s})$ where $M_{i}$ is the matrix of means in the hypothesis that we have $i\geq 2$ communities. It is necessary to do the distinction between case II and case I  because our theorem that gives this lower bound (proved in full details in appendix \hyperref[sec:appA]{A}) do not have a natural generalization to case I.

\subsubsection{Hypothesis tests for case II}\label{subsection:3.2.1}

We denote the number of communities by $r$ or $s$ with $r>s\geq 2$ such that $r$ communities represents the alternative hypothesis.

We define the following test
\begin{equation}
    T_{n,s}(W_{G};\delta) = \mathbb{1}_{\text{SDP}(X-M_{s})>  2n(1 +\delta)\sqrt{w_{+}}}.
\end{equation}
where $M_{s}$ is $\E[W_{G}]$ in the hypothesis we have $s$ communities.
 
Similar to previous calculations, let $\delta>0$ and $T_{n,s}(\cdot;\delta)$ be the test that identifies the community structure. We want again to have the errors of false positive and false negative going to zero asymptotically, that is
\begin{equation}
    \P_{0}(T_{n,s}(W_{G};\delta)=1)+\P_{1}(T_{n,s}(W_{G};\delta)=0) \xrightarrow{n \to \infty} 0
\end{equation}
where $\P_{0}$ is the probability under hypothesis  $H_{s}$: there are s communities and $\P_{1}$ is the probability under hypothesis $H_{r}$: there are r communities.

Analogously to the previous section, we have the following theorem.
\begin{teo}\label{Theorem4} For $\delta>0$, if we have
\begin{equation}\label{eq:74}
\frac{2n^2}{r^2s^2}(M_{in}-M_{out}) >4n(1+\delta)\sqrt{w_{+}} 
\end{equation}
then
\begin{equation}
    \P_{0}(T_{n,s}(W_{G};\delta)=1)+\P_{1}(T_{n,s}(W_{G};\delta)=0) \xrightarrow{n \to \infty} 0.
\end{equation}
\end{teo}

\begin{proof} Again, using theorem $5$ of \cite{montanari} with $\lambda=0$ we obtain
\begin{equation}
    \P( \text{SDP}\left(\sqrt{w_{+}}B\right) \leq 2n(1+\delta)\sqrt{w_{+}}) = 1 - o(1),
\end{equation}
$\P=\P_{0}$ or $\P=\P_{1}$. With the calculations below we will see what conditions do we need to have the two types of errors going to zero asymptotically.

The analysis is similar to the one done in previous section.

\underline{\textbf{Type I Error:}} Under hypothesis H$_{s}$ we have
\begin{equation}
    \P_{0}(T_{n,s}(W_{G};\delta)=1) = \P_{0}(\text{SDP}(W_{G}-M_{s})> 2n(1 +\delta)\sqrt{w_{+}})
\end{equation}
Defining the events
\begin{equation}
    \begin{split}
        E' &= \{\text{SDP}(W_{G}-M_{s})>2n(1+\delta\sqrt{w_{+}})\}\\
        E'' &= \{\text{SDP}(W_{G}-M_{s}) \leq \text{SDP}(\sqrt{w_{+}}B)+n\Lambda_{s}\}
    \end{split}
\end{equation}
where $\text{SDP}(W_{G}-M_{s}) - \text{SDP}(\sqrt{w_{+}}B) \leq n\Lambda_{s}$ and $\Lambda_{s}$ is $\Lambda$ in the upper bound given by equation \ref{eq:scheme} in the case where the mean of $W_{G}$ is $M_{s}$. This is also upper bounded by $\lambda_{0}$ of the previous section. 

With that, we obtain again
\begin{equation}
    E'\cap E'' \subset \{\text{SDP}(\sqrt{w_{+}}B) > 2n(1+\delta)\sqrt{w_{+}}-n\lambda_{0}\}
\end{equation}
such that
\begin{equation}
    \begin{split}
        \P_{0}(E') &\leq \P_{0}(E'\cap E'') + \P_{0}(E''^{c})\\
        &\leq \P_{0}\left(\text{SDP}(\sqrt{w_{+}}B) > \left(2n(1+\delta)-\frac{n\lambda_{0}}{\sqrt{w_{+}}}\right)\sqrt{w_{+}}\right) + \P_{0}(E''^{c}) = o(1)
    \end{split}
\end{equation}
because for $n$ large enough we have $\frac{\lambda_{0}}{\sqrt{w_{+}}}\leq \delta$ and then
\begin{equation}
    \P_{0}\left(\text{SDP}(\sqrt{w_{+}}B) \geq 2n\left(1+\frac{\delta}{2}\right)\sqrt{w_{+}}\right) = o(1).
\end{equation}

\underline{\textbf{Type II Error:}} Under hypothesis H$_{r}$ we have again by same reasoning of previous subsection
\begin{equation}
\begin{split}
    \P_{1}(T_{n,s}(W_{G};\delta)=0) &= \P_{1}(\text{SDP}(W_{G}-M_{s})\leq 2n(1+\delta)\sqrt{w_{+}})\\ 
    &\leq \P_{1}(\text{SDP}(M_{r}-M_{s})-\text{SDP}(M_{r}-W_{G})\leq 2n(1+\delta)\sqrt{w_{+}}) \\
    &= \P_{1}(\Delta -2n(1+\delta)\sqrt{w_{+}} \leq \text{SDP}(M_{r}-W_{G}))\\
    &\leq \P_{1}(2n(1+\delta)\sqrt{w_{+}} \leq \text{SDP}(M_{r}-W_{G})) 
\end{split}
\end{equation}
where in the last inequality we used the result of Appendix \hyperref[sec:appB]{B} 
\begin{equation}
\Delta=\text{SDP}(M_{r}-M_{s}) \geq \frac{2n^2}{r^2s^2}(M_{in}-M_{out}),
\end{equation}
together with the hypothesis that we have $\Delta > 4n(1+\delta)\sqrt{w_{+}}$. Again, we define
\begin{equation}
    \begin{split}
        F' &= \{\text{SDP}(W_{G}-M_{r})>2n(1+\delta)\sqrt{w_{+}})\}\\
        F'' &= \{\text{SDP}(W_{G}-M_{r}) \leq \text{SDP}(\sqrt{w_{+}}B)+n\Lambda_{r}\}
    \end{split}
\end{equation}
where $\Lambda_{r}$ is $\Lambda$ in equation \ref{eq:scheme} in the case where the mean of $W_{G}$ is $M_{r}$. The proof ends again with the same reasoning as of Type I error.

\end{proof}

\subsubsection{Community estimation for case II}\label{subsection:3.2.2}

We conduct a similar analysis as in the case of two communities. However, instead of using a ground truth vector \( x_0 \in \{-1, +1\}^n \), we employ the membership matrix defined in \cite{chretien2021learning}, the definition of which we repeat below.

\begin{defn}
    Let $C_{1},C_{2},...,C_{k}$ be a partition of the vertex set. We denote $i \sim j$ if vertex $i$ and $j$ are in the same community and $i \not\sim j$ otherwise. The membership matrix $Z_{0}$ is defined as
    \begin{equation}
        Z_{0 ij} \coloneqq \begin{cases}
            1, \, i\sim j\\ 
            0, \, i\not\sim j.
        \end{cases}
    \end{equation}
\end{defn}
Following the proof of lemma $6.1.$ in \cite{guedon2016community} this matrix is a solution of a semidefinite programming, explicitly:
\begin{equation}
    Z_{0} \in \argmax_{Z \in \mathcal{C}} \langle \E[W_{G}],Z\rangle
\end{equation}
where  the set of constraints is
\begin{equation}
    \mathcal{C} = \left\{Z \in \mathbb{R}^{n\times n}, Z \succeq 0, Z \geq 0, \text{diag}(Z) = I_{n}, \sum_{i,j=1}^{n}Z_{ij} = \lambda\right\} \quad \text{and} \quad \lambda = \sum_{i,j}^{n}Z_{0ij} = \sum_{l=1}^{K} |C_{l}|^{2}.
\end{equation}

\begin{obs}
    Here we substitute the condition on $\text{diag}(Z) \preceq I_{n}$ in the definition of $\mathcal{C}$ in \cite{guedon2016community} by following its Remark $1.7$.
\end{obs}

In our case of $K$ equally sized communities we have
\begin{equation}
    \lambda = K\left(\frac{n}{K}\right)^{2} = \frac{n^2}{K}
\end{equation}

A natural error measure $e(\hat{Z})$ for the estimation is counting how many wrong pair of vertices we estimate, that is
 \begin{equation}\label{eq:def_error}
     e(\hat{Z}) := \sum_{ij}|r(\hat{Z})_{ij} - Z_{0ij}| 
 \end{equation}
where we do a rounding procedure
\begin{equation}
    r(\hat{Z}) \coloneqq \begin{cases}
        1, \, \text{if} \, \hat{Z}_{ij} > \frac{1}{2} \\
        0, \, \text{if} \, \hat{Z}_{ij} \leq \frac{1}{2}.
    \end{cases}
\end{equation}
Defining again an error related to the SDP as
\begin{equation}
    \textit{Err}(W_{G}) \coloneqq \sup\{|\tr((W_{G}-\E[W_{G}])X)|:X\in \mathcal{C}\} 
\end{equation}
we obtain the following result concerning the estimation of the membership matrix
\begin{prop} Given the functions $e$ and $Err$ defined above and the following estimator for the membership matrix
\begin{equation}
    \hat{Z} \in \argmax_{Z \in \mathcal{C}}\langle W_{G}, Z \rangle
\end{equation}
we obtain
\begin{equation}
    e(\hat{Z}) \leq \frac{16 Err(W_{G})}{M_{in} - M_{out}}
\end{equation}
\end{prop}

\begin{obs}
    We observe that $\mathcal{C}\subset \text{PSD}_{1}(n)$ so that the supremum in $Err(W_{G})$ with feasible set $\mathcal{C}$ is smaller than with feasible set $\text{PSD}_{1}$ . Then we can use again equation \ref{eq:122} when obtaining the bound for the error. 
\end{obs}

\begin{proof} Similar for what we did on the previous case of $2$ communities using the ground truth vector, we write now $\E[W_{G}]$ in terms of the membership matrix as following
\begin{equation}
\begin{split}
    \E[W_{G}] &= M_{in}Z_{0} + M_{out}(\mathbb{1}\mathbb{1}^{T}-Z_{0}) \\
    &= M_{out}\mathbb{1}\mathbb{1}^{T} + (M_{in}-M_{out})Z_{0}
\end{split}
\end{equation}
and calculate 
\begin{equation}
\begin{split}
    \tr{(\E[W_{G}]Z_{0})} = \langle \E[W_G],Z_{0} \rangle &= M_{out}\langle \mathbb{1}\mathbb{1}^{T},Z_{0} \rangle + (M_{in}-M_{out})\langle Z_{0},Z_{0} \rangle \\
    &= M_{out}\lambda +(M_{in}-M_{out})\lambda\\
    &= M_{in}\lambda
\end{split}
\end{equation}
and for any $Z \in \mathcal{C}$
\begin{equation}
\begin{split}
    \tr{(\E[W_{G}]Z)} = \langle \E[W_G], Z\rangle &= M_{out}\langle \mathbb{1}\mathbb{1}^{T}, Z \rangle + (M_{in}-M_{out})\langle Z_{0},Z\rangle\\
    &= M_{out}\lambda + (M_{in}-M_{out})\langle Z_{0},Z \rangle
\end{split}
\end{equation}
then
\begin{equation}
\begin{split}
    \tr{(\E[W_{G}]Z_{0})}-\tr{(\E[W_{G}]Z)} &= M_{in}\lambda - M_{out}\lambda - (M_{in}-M_{out})\langle Z_{0},Z \rangle \\
    &= (M_{in}-M_{out})(\lambda - \langle Z_{0},Z \rangle).
\end{split}  
\end{equation}
Defining
\begin{equation}
    \delta := \frac{\tr{(\E[W_{G}]Z_{0})}-\tr{(\E[W_G]Z)}}{\lambda(M_{in}-M_{out})}
\end{equation}
we obtain
\begin{equation}
    \lambda - \langle Z_{0}, Z \rangle = \delta \lambda.
\end{equation}
Let us define an estimator for $Z_{0}$ as
\begin{equation}
    \hat{Z} \in \argmax_{Z \in \mathcal{C}}\langle W_{G}, Z \rangle,
\end{equation}
in that case, we obtain
\begin{equation}
    \tr{(W_{G}Z_{0})}-\tr{(W_{G}\hat{Z})} \leq 0 \Rightarrow \tr{(\E[W_{G}]Z_{0})} - \tr{(\E[W_{G}]\hat{Z})} \leq 2Err(W_{G})
\end{equation}

Using $Z=\hat{Z}$ in the definition of $\delta$, by the previous calculations we obtain
\begin{equation}
    1 - \frac{\langle Z_{0},\hat{Z} \rangle}{\lambda} \leq \frac{2Err(W_{G})}{\lambda(M_{in}-M_{out})}
\end{equation}
We argue that if
\begin{equation}
    1 - \frac{\langle Z_{0},\hat{Z} \rangle}{\lambda} = 0
\end{equation}
for $\hat{Z} \in \mathcal{C}$, then we can recover the communities exactly. We have the following
\begin{equation}
    1 - \frac{\langle Z_{0}, \hat{Z} \rangle}{\lambda} = 0 \Rightarrow \langle Z_{0}, \hat{Z} \rangle = \lambda \Rightarrow \sum_{i,j=1}^{n}Z_{0ij}\hat{Z}_{ij} = \lambda \Rightarrow \sum_{i \sim j}\hat{Z}_{ij} = \lambda
\end{equation}
also
\begin{equation}
    \hat{Z} \in \mathcal{C} \Rightarrow \sum_{i,j=1}^{n}\hat{Z}_{ij} = \lambda \Rightarrow \sum_{i \sim j}\hat{Z}_{ij} + \sum_{i \not\sim j}\hat{Z}_{ij} = \lambda \Rightarrow \sum_{i \not\sim j}\hat{Z}_{ij} = 0
\end{equation}
where the last implication follows from to the previous equation. Then, by the condition $\hat{Z}_{ij}\geq 0 \, \forall i,j$ we obtain $\hat{Z}_{ij}=0$ for all $i\not\sim j$. It follows that
\begin{equation}
    \sum_{i j}\hat{Z}_{ij} = \sum_{i \sim j}\hat{Z}_{ij} = \lambda.
\end{equation}
where we  have $\lambda$ terms in the second sum and by the definition of $\mathcal{C}$, $0 \leq \hat{Z}_{ij} \leq 1$. Therefore, all of them are exactly $1$ and then $\hat{Z}=Z_{0}$ as we wished. In fact, if one of them were not $1$, for example, $\hat{Z}_{i^{*}j^{*}} = c <1$ then we would have
\begin{equation}
    \sum_{i \sim j}\hat{Z}_{ij} \leq (\lambda -1) + c < \lambda 
\end{equation}
a contradiction. 

Now, we show that if $1 - \frac{\langle Z_{0},\hat{Z} \rangle}{\lambda}$ is small, we can still recover the communities well such that a good bound on the $Err(W_{G})$ would be sufficient for the estimation. For this purpose, we want to find a procedure on how to estimate the communities given the entries of $\hat{Z}$ and then a error measure in the estimation such that the term $1 - \frac{\langle Z_{0},\hat{Z} \rangle}{\lambda}$ or something proportional to it such as $\lambda - \langle Z_{0}, \hat{Z} \rangle$ appear.

We know that the entries of matrices in $\mathcal{C}$ are in $[0,1]$, then it is natural to choose the estimation procedure as a rounding $r(\hat{Z})$.
 
 We show that $e(\hat{Z})\leq 8(\lambda - \langle Z_{0}, \hat{Z} \rangle)$ as we wanted. In fact, 

\underline{if $\hat{Z}_{ij}>\frac{1}{2}$}:
\begin{equation}
\begin{split}
    (r(\hat{Z})_{ij}-Z_{0ij})^{2} = (1-Z_{0ij})^{2} &= \begin{cases}
        0, \, \text{if} \, Z_{0ij}=1\\
        1, \, \text{if} \, Z_{0ij}=0
    \end{cases}\\
    (\hat{Z}_{ij}-Z_{0ij})^{2} &= \begin{cases}
        <1/4, \, \text{if} \, Z_{0ij}=1\\
        > 1/4, \, \text{if} \, Z_{0ij}=0
    \end{cases}
\end{split}
\end{equation}

\underline{if $\hat{Z}_{ij}\leq \frac{1}{2}$}:
\begin{equation}
\begin{split}
    (r(\hat{Z})_{ij}-Z_{0ij})^{2} = (0-Z_{0ij})^{2} &= \begin{cases}
        1, \, \text{if} \, Z_{0ij}=1\\
        0, \, \text{if} \, Z_{0ij}=0
    \end{cases}\\
    (\hat{Z}_{ij}-Z_{0ij})^{2} &= \begin{cases}
        \geq 1/4, \, \text{if} \, Z_{0ij}=1\\
        \leq 1/4, \, \text{if} \, Z_{0ij}=0
    \end{cases}
\end{split}
\end{equation}

In both cases we obtain
\begin{equation}
\begin{split}
    e(\hat{Z}) =      \sum_{ij}|r(\hat{Z})_{ij} - Z_{0ij}| =   \sum_{ij}(r(\hat{Z})_{ij}-Z_{0ij})^{2}&\leq 4\sum_{ij}(\hat{Z}_{ij}-Z_{0ij})^{2}\\
    &\leq 4 \left(\sum_{ij}\hat{Z}_{ij}^{2} - 2\langle Z_{0},\hat{Z}\rangle + \sum_{ij}Z_{0ij}\right)\\
    &\leq 4 \left(\sum_{ij}\hat{Z}_{ij} - 2\langle Z_{0},\hat{Z}\rangle + \lambda\right)\\
    & \leq 4 \left(2\lambda - 2\langle Z_{0},\hat{Z}\rangle\right)\\
    &= 8(\lambda - \langle Z_{0}, \hat{Z} \rangle)
\end{split}
\end{equation}
as we wanted. 
\end{proof}

Summarizing the previous result we obtain the following theorem.
\begin{teo} \label{Theorem5}
If $\hat{Z}$ is the previous estimator, then by \ref{eq:122} we have asymptotically with high probability
\begin{equation}
    e(\hat{Z}) \leq \frac{32n\sqrt{w_{+}}}{M_{in}-M_{out}} + o(\sqrt{w_{+}}).
\end{equation}
obtaining again a partial estimation of the communities.
\end{teo}

\section{Estimating $K$} \label{sec:estimagingK}

We now focus on the main problem we are interested in this paper. Building on the previously defined hypothesis tests, we construct a sequential test, inspired by \cite{lei2016_test}, to determine the number of communities.. In contrast to the previous tests, the sequential test here is defined with estimated quantities such that they can actually be done in practice. We illustrate our findings in Section \ref{sec:simulations} with simulations of the test in the case we have zero-inflated Gaussian weights which were defined in \ref{sec:application}.

Let $z\in [K]^n$ be the membership vector of our weighted SBM with $K$ communities and $n$ nodes, that is, the $i$-th entry of $z$ gives the label from $1$ to $K$ of the corresponding community. As we assume the case of balanced communities, we have $|\{i: z_i=a\}|=n/K$, for $a=1,\dots,K$.

For $K_{0}\in\mathbb{N}$, a candidate for the number of communities, and $\varepsilon>0$ we create the hypothesis 
\begin{equation}
\left\lbrace\begin{array}{c}
    H_{0,K_0} : \#\text{communities} = K_0\\
    H_{a,K_0} : \#\text{communities} > K_0
\end{array} \right. 
\end{equation}
and statistical tests $\hat{T}_{n,K_0}(W_{G};\varepsilon)$, where $\hat{T}_{n,K_0}(W_{G};\varepsilon) = 1$ means the hypothesis $H_{0,K_0}$ is rejected and $\hat{T}_{n,K_0}(W_{G};\varepsilon) = 0$ means $H_{0,K_0}$ is not rejected.

The test is performed sequentially for $K_0=1,2,\dots$ until it does not reject $H_{0,K_0}$ (following the idea of \cite{lei2016_test}). The sequential testing estimator is then defined by
\begin{equation}\label{def:hat_k}
    \hat K_n = \min \{ K_0\geq 1 : \hat{T}_{n,K_0}(W_{G};\varepsilon)=0\}\,.
\end{equation}

To prove that $\hat K_n$ is consistent, we need to show that

\begin{equation}\label{eq:consistency}
    \P(\hat K_n = K)\to 1 \text{ as } n\to \infty.
\end{equation}

To prove this, it suffices to show the following two things.
\begin{itemize}
    \item \textbf{non-overestimation:} $ \P(\hat K_n > K)\to 0 \text{ as } n\to \infty$. To show this, observe that
\begin{equation}\label{eq:nonover}
    \begin{split}
         \P(\hat K_n > K) &\leq \P(\text{rejected for all $K_0 \leq K$})\\
          &\leq  \P(\text{rejected for $K_0= K$})\\
          & = \P(\hat{T}_{n,K}(W_{G};\varepsilon)=1) \\
    \end{split}
\end{equation}
    \item \textbf{non-underestimation:} $\P(\hat K_n < K)\to 0 \text{ as } n\to \infty$. To show this, observe that
\begin{equation}\label{eq:nonunder}
    \begin{split}
         \P(\hat K_n < K) &= \P(\text{for some $K_0 \leq K-1$ the test was not rejected})\\
          &= \P(\cup_{K_0=1}^{K-1} \{\hat{T}_{n,K_0}(W_{G};\varepsilon)=0\})\\
          & \leq \sum\limits_{K_0=1}^{K-1}\P(\hat{T}_{n,K_0}(W_{G};\varepsilon)=0) \\
    \end{split}
\end{equation}
\end{itemize}

Let $\hat{z}_{K_{0}} \in [K_0]^n$ be any consistent estimator of the communities in the case we have $K_{0}$ of them, that is, we have $\P(\hat{z}_{K_0}=z)\xrightarrow{n\to \infty} 1$ where the equality is up to permutations. For $K_0\in\{1,2,\dots\}$, let $\hat{n}_{K_{0},in}=|\{(i,j): 1\leq i,j \leq n, \, \hat{z}_{K_{0}i}=\hat{z}_{K_{0}j}\}|$ be the set of pairs of nodes estimated to be in the same community and $\hat{n}_{K_{0},out}=|\{(i,j): 1\leq i,j \leq n, \,\hat{z}_{K_{0}i}\neq\hat z_{K_{0}j}\}|$ to be the pair of nodes estimated to be in different communities. Define

\begin{equation}
    \hat{M}_{K_{0},in} \coloneqq \frac{1}{\hat{n}_{K_{0},in}}\sum\limits_{i,j=1}^n W_{Gij}\mathbb{1}\{\hat z_{K_{0}i} = \hat z_{K_{0}j}\}
\end{equation}
and 
\begin{equation}
    \hat{M}_{K_{0},out} \coloneqq \frac{1}{\hat{n}_{K_{0},out}}\sum\limits_{i,j=1}^n W_{Gij}\mathbb{1}\{\hat z_{K_{0}i} \neq \hat z_{K_{0}j}\}.
\end{equation}
Using these, the $n\times n$ matrix $ \hat M_{K_0}$ of the estimated means is constructed as
\begin{equation}
    (\hat M_{K_0})_{ij} =\left\lbrace \begin{array}{cc}
        \hat{M}_{K_{0},in}, & \hat z_{K_{0}i} = \hat z_{K_{0}j} \\
        \hat{M}_{K_{0},out}, & \hat z_{K_{0}i} \neq \hat z_{K_{0}j} 
    \end{array}\right.
\end{equation}

The idea is again to have the hypothesis tests based on the SDP.  With previously defined tests in mind, we define analogously
\begin{equation}
    \hat{T}_{n,K_0}(W_{G};\varepsilon) \coloneqq \mathbb{1}\{ \text{SDP}(W_{G} - \hat{M}_{K_0}) > 2n(1+\varepsilon)\sqrt{w_{+}}\}.
\end{equation}
where $\hat{M}_{K_{0}}$ is the equivalent estimated quantity. 

We begin proving the non-overestimation bound.
\begin{teo}\label{Theorem6}(\textbf{Non-overestimation}). For $\hat{K}_{n}$ given previously we have
\begin{equation}
    \P(\hat K_n > K)\to 0 \text{ as } n\to \infty.
\end{equation}
\end{teo}

\begin{proof}
By equation \ref{eq:nonover}, we only need to understand what happens when $K_{0}=K$. In this case, we denote $\hat{z}_{K} \in [K]^n$ a consistent estimator of $z$ and we also define the following.
\begin{equation}
\begin{split}
    M_{K,in} \coloneqq \frac{1}{n_{K,in}}\sum\limits_{i,j=1}^n W_{Gij}\mathbb{1}\{z_{i} = z_{j}\} \xrightarrow{n \to \infty} M_{in}\\
    M_{K,out} \coloneqq \frac{1}{n_{K,out}}\sum\limits_{i,j=1}^n W_{Gij}\mathbb{1}\{z_{i} \neq z_{j}\}\xrightarrow{n \to \infty} M_{out}.    
\end{split}
\end{equation}
where the limits come from the law of large numbers. We construct similarly
\begin{equation}
    (M_{K})_{ij} =\left\lbrace \begin{array}{cc}
        M_{K,in}, & z_{i} = z_{j} \\
        M_{K,out}, & z_{i} \neq z_{j}. 
    \end{array}\right.
\end{equation}

Then we obtain 
\begin{equation}
    \begin{split}
        \P(\hat{T}_{n,K}(W_{G},\varepsilon)=1) &= \P(\text{SDP}(W_{G} - \hat{M}_{K}) > 2n(1+\varepsilon)\sqrt{w_{+}})\\
        & \leq \P(\text{SDP}(W_{G}- \E[W_{G}])+ \text{SDP}( \E[W_{G}] - \hat{M}_{K}) > 2n(1+\varepsilon)\sqrt{w_{+}})\\
        &\leq \P(\text{SDP}(W_{G}- \E[W_{G}])+ K_{G} \lVert \hat{M}_{K}- \E[W_{G}]\rVert_{\infty \to 1}> 2n(1+\varepsilon)\sqrt{w_{+}}) \\
        &\leq \P(\text{SDP}(W_{G}- \E[W_{G}])+ K_{G} \lVert M_{K}- \E[W_{G}]\rVert_{\infty \to 1}> 2n(1+\varepsilon)\sqrt{w_{+}}) + \P(M_{K}\neq \hat{M}_{K})\\
        &\leq \P(\text{SDP}(W_{G}- \E[W_{G}])+ K_{G} \lVert M_{K}- \E[W_{G}]\rVert_{\infty \to 1}> 2n(1+\varepsilon)\sqrt{w_{+}}) + o(1)
    \end{split}
\end{equation}
where in the third line we used \ref{eq:154} of Appendix \hyperref[sec:appA]{A} below, and in the last line we used the fact that, if we define the event $G^{(n)}\coloneqq \{z^{(n)}=\hat{z}_{K}^{(n)}\}$, we obtain $G \subset \{M_{K} =\hat{M}_{K}\}$ and then $ \P(M_{K}\neq \hat{M}_{K}) \leq \P(G^{(n)^{c}})=o(1)$ by the consistency of $\hat{z}_{K}$.

Now, observe that the entries of $M_{K}-\E[W_{G}]$ are of two possible types
\begin{equation}
    (M_{K}-\E[W_{G}])_{ij} =\left\lbrace \begin{array}{cc}
       Y_{in}\coloneqq \frac{1}{n_{in}}\sum_{ij}(W_{Gij}-M_{in})\mathbb{1}\{z_{i}=z_{j}\}, & z_{i} = z_{j} \\
        Y_{out}\coloneqq \frac{1}{n_{out}}\sum_{ij}(W_{Gij}-M_{out})\mathbb{1}\{z_{i}\neq z_{j}\}, & z_{i} \neq z_{j}
    \end{array}\right.
\end{equation}
such that 
\begin{equation}\label{eq:125}
    \sum_{ij}|(M_{K}-\E[W_G])_{ij}| \leq n^{2}\max\{|Y_{in}|,|Y_{out}|\}
\end{equation}
and this fact is going to be used in the calculation below.

Continuing, using the fact that for any matrix $M\in\mathbb{R}^{n\times n}$ we have $\lVert M\rVert_{\infty \to 1}\leq \sum_{ij}|M_{ij}|$ we obtain
\begin{equation}
    \begin{split}
        \P(\hat{T}_{n,K}(W_{G},\varepsilon)=1) &\leq \P(\text{SDP}(W_{G}- \E[W_{G}])+ K_{G} \lVert M_{K}- \E[W_{G}]\rVert_{\infty \to 1}> 2n(1+\varepsilon)\sqrt{w_{+}}) + o(1)\\
        &\leq \P\left(\text{SDP}(W_{G}- \E[W_{G}])+ K_{G} \sum_{1\leq i \leq j \leq n}|\hat{M}_{Kij}-\E[W_{Gij}]|> 2n(1+\varepsilon)\sqrt{w_{+}}\right)+o(1) \\
        &\leq \P\left(\text{SDP}(W_{G}- \E[W_{G}])+ K_{G}n^{2}|Y_{in}|> 2n(1+\varepsilon)\sqrt{w_{+}}\right)\\
        &+\P\left(\text{SDP}(W_{G}- \E[W_{G}])+ K_{G}n^{2}|Y_{out}|> 2n(1+\varepsilon)\sqrt{w_{+}}\right)+o(1)
    \end{split}
\end{equation}
Now, remembering we are considering $W_{Gij}$ to be sub-gamma random variables when centered for all $1\leq i,j \leq n$, we can use again item (b) from Proposition $5.1$ in \cite{zhang2020concentration} in the last line below and get
\begin{equation}
\begin{split}
\P(|Y_{in}|>\varepsilon_{in}) =\P(|M_{K,in} - M_{in}|> \varepsilon_{in}) &= \P\left(|n_{K,in}M_{K,in} - n_{K,in}M_{in}|> n_{K,in}\varepsilon_{in}\right)\\
&= \P\left(\left|\sum_{i \sim j}(W_{Gij}-\E[W_{Gij}])\right|> n_{K,in}\varepsilon_{in}\right)\\
&\leq 2\exp{\left(\frac{-n_{K,in}^{2}\varepsilon_{in}^{2}}{C[\sum_{i \sim j, i\leq j}\nu_{ij} + n_{K,in}\varepsilon_{in}\max_{i\sim j,i\leq j}c_{ij}]}\right)} 
\end{split}
\end{equation}
where $C$ is come constant related to the fact we have independence just up to symmetry. We have similarly for $|M_{K,out}-M_{out}|$.
By already done calculations, we know that we can have $\P(|Y_{in}|>\varepsilon_{in})\leq \delta_{in}$ (same for \textit{out}) if we choose
\begin{equation}
    \varepsilon_{in} \coloneqq \frac{C}{n_{K,in}}\sqrt{\sum_{\substack{i \sim j \\ i \leq j}}\nu_{ij}\log{\left(\frac{1}{\delta_{in}}\right)}} + \frac{C}{n_{K,in}}\max_{\substack{i \sim j \\ i \leq j}}c_{ij}\log{\left(\frac{1}{\delta_{in}}\right)}
\end{equation}
where $C$ is another constant. Using assumption \ref{assum:3} we obtain
\begin{equation}
    \varepsilon_{in}  \lesssim n^{-\frac{3}{2}}\sqrt{w_{+}\log{\left(\frac{1}{\delta_{in}}\right)} }+n^{-2}\theta\sqrt{w_{+}}\log{\left(\frac{1}{\delta_{in}}\right)} 
\end{equation}
because $n_{K,in} \approx n^{2}$ and we do similarly for \textit{out}. Choosing $\delta_{in}=\delta_{out}=\frac{1}{n}$, we finally get by using  the fact given in equation \ref{eq:125}
\begin{equation}
   \begin{split}
       \P(\hat{T}_{n,K}(W_{G},\varepsilon)=1) &\leq \P\left(\text{SDP}(W_{G}- \E[W_{G}])+ K_{G}n^{2}|Y_{in}|> 2n(1+\varepsilon)\sqrt{w_{+}}\right)\\ &+ \P\left(\text{SDP}(W_{G}- \E[W_{G}])+ K_{G}n^{2}|Y_{out}|> 2n(1+\varepsilon)\sqrt{w_{+}}\right)+o(1)\\
       &\leq 2\P\left(\text{SDP}(W_{G}- \E[W_{G}])+ K_{G}n^{2}(n^{-\frac{3}{2}}\log{n} +n^{-2}\theta\log{n} )\sqrt{w_{+}}> 2n(1+\varepsilon)\sqrt{w_{+}}\right)+o(1)\\
       &\leq 2\P(\text{SDP}(\sqrt{w_{+}}B)>2n(1+\varepsilon)\sqrt{w_{+}} - n\lambda -n\lambda_{0})+o(1)=o(1)
   \end{split} 
\end{equation}
where $\lambda \coloneqq K_{G}(n^{-\frac{1}{2}}\log{n} + n^{-1}\theta\log{n})\sqrt{w_{+}}$ is, as $\lambda_{0}$, also of order $o(\sqrt{w_{+}})$ and then the result follows by the same reasoning as of previous sections.

We conclude that 
\begin{equation}
    \begin{split}
         \P(\hat K_n > K) &= \P(\text{rejected for all $K_0 \leq K$})\\
          &\leq  \P(\text{rejected for $K_0= K$})\\
          & \leq \P(\hat{T}_{n,K}(W_{G};\varepsilon)=1) = o(1) \\
    \end{split}
\end{equation}
as we wanted.
\end{proof}

Now, we prove the non-underestimation error.

\begin{teo}\label{Theorem7}(\textbf{Non-underestimation}).  For $\hat{K}_{n}$ defined previously we obtain
\begin{equation}
    \P(\hat K_n < K)\to 0 \text{ as } n\to \infty
\end{equation}
under the assumption that we have
\begin{equation}\label{eq:136}
    M_{in}-M_{out}>\min\left\{\frac{2K^{4}(K+1)^{2}(1+\varepsilon)\sqrt{w_{+}}}{n},\frac{K(K-1)\log(2(K-1))\sqrt{w_{+}}}{n}\right\}.
\end{equation}
\end{teo}

\begin{proof}
We need to prove that
    \begin{equation}
    \begin{split}
         \P(\hat K_n < K) &= \P(\text{for some $K_0 \leq K-1$ the test was not rejected})\\
          &= \P(\cup_{K_0=1}^{K-1} \{\hat{T}_{n,K_0}(W_{G};\varepsilon)=0\})\\
          & \leq \sum\limits_{K_0=1}^{K-1}\P(\hat{T}_{n,K_0}(W_{G};\varepsilon)=0) = o(1)\\
    \end{split}
\end{equation}
Consider any $K_{0}<K$. Observe that
\begin{equation} \label{eq:139}
    \begin{split}
        \P(\hat{T}_{n,K_0}(W_{G};\varepsilon)=0) &= \P(\text{SDP}(W_{G}-\hat{M}_{K_0})\leq 2n(1+\varepsilon)\sqrt{w_{+}})\\
        &= \P(\text{SDP}(W_{G} -M_{K}+M_{K} -\hat{M}_{K_0})\leq
        2n(1+\varepsilon)\sqrt{w_{+}})\\
        &\leq \P(\text{SDP}(M_{K} -\hat{M}_{K_0}) -\text{SDP}(M_{K}-W_{G})\leq
        2n(1+\varepsilon)\sqrt{w_{+}})
    \end{split}
\end{equation}
The problem now becomes to find \textbf{a lower bound for} $\text{SDP}(M_{K} -\hat{M}_{K_0})$.

We can write
\begin{equation}
    \begin{split}
        M_{K} &= M_{out}\mathbb{1}\mathbb{1}^{T} + (M_{in}-M_{out})B_{K}\\
        \hat{M}_{K_{0}} &= \hat{M}_{K_{0},out}\mathbb{1}\mathbb{1}^{T} + (\hat{M}_{K_{0},in}-\hat{M}_{K_{0},out})\hat{B}_{K_{0}}
    \end{split}
\end{equation}
where $B_{K}$ and $B_{K_{0}}$ are membership matrices. We can rewrite
\begin{equation}
    \begin{split}
        \hat{M}_{K_{0}} &= (\hat{M}_{K_{0},out}-M_{out}+M_{out})\mathbb{1}\mathbb{1}^{T} + (\hat{M}_{K_{0},in}-M_{in}+M_{in}-\hat{M}_{K_{0},out}+M_{out}-M_{out})\hat{B}_{K_{0}}\\
        &= M_{out}\mathbb{1}\mathbb{1}^{T} + (M_{in}-M_{out})\hat{B}_{K_{0}} + (\hat{M}_{K_{0},out}-M_{out})\mathbb{1}\mathbb{1}^{T}+ (\hat{M}_{K_{0},in}-M_{in}-(\hat{M}_{K_{0},out}-M_{out}))\hat{B}_{K_{0}}\\
        &=: \hat{D}_{K_{0}} + \hat{C}_{K_{0}}
    \end{split}
\end{equation}
where
\begin{equation}
    \begin{split}
    \hat{D}_{K_{0}}&\coloneqq M_{out}\mathbb{1}\mathbb{1}^{T} + (M_{in}-M_{out})\hat{B}_{K_{0}}\\
    \hat{C}_{K_{0}}&\coloneqq (\hat{M}_{K_{0},out}-M_{out})\mathbb{1}\mathbb{1}^{T}+(\hat{M}_{K_{0},in}-M_{in}-(\hat{M}_{K_{0},out}-M_{out}))\hat{B}_{K_{0}}.
    \end{split}
\end{equation}
We obtain then
\begin{equation}
\begin{split}
    \text{SDP}(M_{K}-\hat{M}_{K_{0}}) &= \text{SDP}(M_{K}-\hat{D}_{K_{0}}-\hat{C}_{K_{0}})\\
    &\geq \text{SDP}(M_{K}-\hat{D}_{K_{0}}) - \text{SDP}(\hat{C}_{K_{0}})
\end{split}
\end{equation}
and for the first term we have the lower bound of appendix \hyperref[sec:appB]{B}. 

The goal becomes \textbf{to find an upper bound for} $\text{SDP}(\hat{C}_{K_{0}})$. Let $\epsilon_{0} > 0$, we divide this task into cases.

\underline{\textbf{Case 1:} $|\hat{M}_{K_{0},out} - M_{out}|<\epsilon_{0}$ and $|\hat{M}_{K_{0},in}-M_{in}|<\epsilon_{0}$.} 
In this case we have
\begin{equation}
    \begin{split}
        \text{SDP}(\hat{C}_{K_{0}}) &= \text{SDP}((\hat{M}_{K_{0},out}-M_{out})\mathbb{1}\mathbb{1}^{T}+(\hat{M}_{K_{0},in}-M_{in}-(\hat{M}_{K_{0},out}-M_{out}))\hat{B}_{K_{0}})\\
        &\leq \text{SDP}((\hat{M}_{K_{0},out}-M_{out})\mathbb{1}\mathbb{1}^{T}) + \text{SDP}((\hat{M}_{K_{0},in}-M_{in}-(\hat{M}_{K_{0},out}-M_{out}))\hat{B}_{K_{0}})\\
        & \leq \epsilon_{0} \text{SDP}(\mathbb{1}\mathbb{1}^{T}) + 2\epsilon_{0} \text{SDP}(\hat{B}_{K_{0}})\\
        &= \epsilon_{0} n^{2} + 2\epsilon_{0} \frac{n^{2}}{K_{0}}
    \end{split}
\end{equation}
where in the last inequality we calculated the SDP's exactly considering the number of ones in the matrices.

Then, using that we already know from appendix \ref{sec:appB} the result
\begin{equation}
    \text{SDP}(M_{K}-\hat{D}_{K_{0}}) \geq \frac{2n^{2}(M_{in}-M_{out})}{K^{2}K_{0}^{2}}
\end{equation}
we want to choose $\epsilon_{0}$ such that
\begin{equation}
    \begin{split}
        \text{SDP}(M_{K}-\hat{D}_{K_{0}}) - \text{SDP}(\hat{C}_{K_{0}}) \geq \frac{2n^{2}(M_{in}-M_{out})}{K^{2}K_{0}^{2}} -\epsilon_{0} n^{2} -2\epsilon_{0} \frac{n^{2}}{K_{0}} \geq cn^{2}
    \end{split}
\end{equation}
with $c>0$ to have a positive lower bound for $\text{SDP}(M_{K} -\hat{M}_{K_0})$ of order $O(n^{2})$.

Therefore, we need to have
\begin{equation}\label{eq:138}
\begin{split}
    &\frac{2(M_{in}-M_{out})}{K^{2}K_{0}^{2}} - \epsilon_{0} - \frac{2\epsilon_{0}}{K_{0}} > 0 \\
    &\epsilon_{0} < \frac{2(M_{in}-M_{out})}{K^{2}K_{0}(K_{0}+2)}
\end{split}
\end{equation}

\underline{\textbf{Case 2:} $|\hat{M}_{K_{0},in}-M_{in}|>\epsilon_{0}$ only or both $>\epsilon_{0}$.} First of all, we can show that we have asymptotically with high probability 
\begin{equation}
M_{in}\geq\hat{M}_{K_{0},in} \quad\text{and}\quad\hat{M}_{K_{0},out}\geq M_{out}
\end{equation}
Intuitively this is justified when we think that, for the estimator with $K_{0}<K$, when we average out $\frac{n^{2}}{K_{0}}$ entries in the definition of $\hat{M}_{K_{0},in}$, as $\frac{n^{2}}{K_{0}}>\frac{n^{2}}{K}$ we are also averaging entries coming from the distribution with mean $M_{out}<M_{in}$ such that, together with concentration holding, the average decreases.

In fact, writing 
\begin{equation}
    \hat{M}_{K_{0},in} = \frac{1}{\hat{n}_{K_{0},in}}\sum_{i,j=1}^{n}(W_{Gij}-\E W_{Gij})\mathbb{1}\{\hat{z}_{K_{0},i}=\hat{z}_{K_{0},j}\} + \frac{1}{\hat{n}_{K_{0},in}}\sum_{i,j=1}^{n}\E W_{Gij}\mathbb{1}\{\hat{z}_{K_{0},i}=\hat{z}_{K_{0},j}\}
\end{equation}
we can bound the second term by
\begin{equation}
\begin{split}
    \frac{1}{\hat{n}_{K_{0},in}}\sum_{i,j=1}^{n}\E W_{Gij}\mathbb{1}\{\hat{z}_{K_{0},i}=\hat{z}_{K_{0},j}\} &\leq \frac{K_{0}}{n^{2}}\left[\frac{n^{2}}{K}M_{in}+\left(\frac{n^{2}}{K_{0}}-\frac{n^{2}}{K}\right)M_{out}\right]\\
    &= \frac{K_{0}}{K}M_{in} + \left(1-\frac{K_{0}}{K}\right)M_{out}
\end{split}
\end{equation}
where we used that $\frac{n^{2}}{K_{0}}>\frac{n^{2}}{K}$ and $\frac{n^{2}}{K}$ is the true number of $M_{in}$ entries.

For the first term we can use sub-gamma concentration in the following sense. We know that for any $t>0$
\begin{equation}
\begin{split}
    &\P\left(\exists \hat{z}_{K_{0}}:\left|\sum_{i,j=1}^{n}(W_{Gij}-\E W_{Gij})\mathbb{1}\{\hat{z}_{K_{0},i}=\hat{z}_{K_{0,j}}\}\right|> t\hat{n}_{K_{0},in}\right)\\
    & \leq\sum_{\hat{z}\in\{1,...,K_{0}\}^{n}}\P\left(\left|\sum_{i,j=1}^{n}(W_{Gij}-\E W_{Gij})\mathbb{1}\{\hat{z}_{K_{0},i}=\hat{z}_{K_{0,j}}\}\right|> t\hat{n}_{K_{0},in}\right)\\
    &\leq 2\exp{\left(\frac{-t^2\hat{n}_{K_{0},in}/2}{\sum_{ij}\nu_{ij}+ct\hat{n}_{K_{0},in}}\right)}\leq \delta
\end{split}
\end{equation}
if we take
\begin{equation}
    \begin{split}
        t &\coloneqq \frac{1}{\hat{n}_{K_{0},in}}\sqrt{2\sum_{ij}\nu_{ij}\log{\left(\frac{1}{\delta}\right)}} + \frac{1}{\hat{n}_{K_{0},in}}c\log{\left(\frac{1}{\delta}\right)}\\
        &\leq \frac{1}{\hat{n}_{K_{0},in}}\sqrt{2nw_{+}\log{\left(\frac{1}{\delta}\right)}} + \frac{1}{\hat{n}_{K_{0},in}}\theta\sqrt{w_{+}}\log{\left(\frac{1}{\delta}\right)}\\
        &\leq K_{0}[n^{-\frac{3}{2}}\sqrt{w_{+}\log{(1/\delta)}}+n^{-2}\theta\sqrt{w_{+}}\log(1/\delta)]
    \end{split}
\end{equation}
where we used Assumption \ref{assum:3}. Choosing $\delta = (2K_{0})^{-n}$ we obtain
\begin{equation}
    t \leq K_{0}\frac{\sqrt{w_{+}}}{n}\log{(2K_{0})}
\end{equation}
such that
\begin{equation}
\begin{split}
    &\P\left(\exists \hat{z}_{K_{0}}: \frac{1}{\hat{n}_{K_{0},in}}\left|\sum_{i,j=1}^{n}(W_{Gij}-\E W_{Gij})\mathbb{1}\{\hat{z}_{K_{0},i}=\hat{z}_{K_{0,j}}\}\right|> \frac{C_{0}\sqrt{w_{+}}}{n}\right)\\
    &\leq \sum_{\hat{z}\in\{1,...,K_{0}\}^{n}}\P\left( \frac{1}{\hat{n}_{K_{0},in}}\left|\sum_{i,j=1}^{n}(W_{Gij}-\E W_{Gij})\mathbb{1}\{\hat{z}_{K_{0},i}=\hat{z}_{K_{0,j}}\}\right|> \frac{C_{0}\sqrt{w_{+}}}{n}\right)\leq 2^{-n+1}.
\end{split}
\end{equation}

Then, asymptotically with high probability we obtain
\begin{equation}
\begin{split}
    &\hat{M}_{K_{0},in} \leq \frac{C_{0}\sqrt{w_{+}}}{n} + \frac{K_{0}}{K}M_{in} + \left(1-\frac{K_{0}}{K}\right)M_{out}\\
    \Rightarrow &\hat{M}_{K_{0},in}-M_{in} \leq \left(\frac{K_{0}}{K}-1\right)M_{in} + \left(1-\frac{K_{0}}{K}\right)M_{out} + \frac{C_{0}\sqrt{w_{+}}}{n}\\
    \Rightarrow &M_{in}-\hat{M}_{K_{0},in} \geq \left(1-\frac{K_{0}}{K}\right)(M_{in}-M_{out})-\frac{C_{0}\sqrt{w_{+}}}{n}
\end{split}
\end{equation}
and $M_{in}\geq \hat{M}_{K_{0},in}$ if the following condition holds
\begin{equation}
    \left(1-\frac{K_{0}}{K}\right)(M_{in}-M_{out}) > \frac{C_{0}\sqrt{w_{+}}}{n}.
\end{equation}
where $C_{0}=K_{0}\log{(2K_{0})}$. Considering the worst case of $K_{0}=K-1$ we need the condition
\begin{equation}
    M_{in} - M_{out} > \frac{K(K-1)\sqrt{w_{+}}\log(2(K-1))}{n}.
\end{equation}

Similarly, we have $\hat{M}_{K_{0},out}>M_{out}$ because asymptotically with high probability
\begin{equation}
    \hat{M}_{K_{0},out} = \frac{K_{0}}{n^{2}(K_{0}-1)}\langle W_{G}, \mathbb{1}\mathbb{1}^{T}-\hat{B}_{K_{0}}\rangle \geq \frac{K_{0}}{n^{2}(K_{0}-1)}\left[\frac{(K_{0}-1)n^{2}}{K_{0}}(M_{out}+t)\right] = M_{out} + t > M_{out}.
\end{equation}

In this case, we calculate directly $\text{SDP}(M_{K}-\hat{M}_{K_{0}})$ (without using $\hat{C}_{K_{0}}$ and $\hat{D}_{K_{0}}$). We can choose $\check{Z} \in$ PSD$_{1}$ to get a lower bound on this SDP as follows
\[
  \setlength{\arraycolsep}{0pt}
  \setlength{\delimitershortfall}{0pt}
  \check{Z} =\begin{pmatrix}
    \,\resizebox{!}{1em}{\fbox{$1$}} &  &  & \,  \\
    \, & \resizebox{!}{1em}{\fbox{$1$}} &  & 0\, \\
    \,0&  & \ddots & \, \\
    \, &  &  & \resizebox{!}{1em}{\fbox{$1$}}\, \\
  \end{pmatrix}
\]
where we have $K$ blocks of size $\frac{n^{2}}{K^{2}}$. Then, a lower bound not reached by any possible configuration $\hat{M}^{*}_{K_{0}}$ (because they can not have that block structure) is reached when we put the entries $M_{in}-\hat{M}_{K_{0},in}$ of any possible configuration in the same position as the $1$'s blocks above. This is a lower bound because it is the minimum bound in the bigger set where we consider the entries of $M_{K}-\hat{M}_{K_{0}}$ ``free" from the block structure of each one separately.

\begin{obs}
    We have enough $\hat{M}_{K_{0},in}$, we have $\frac{n^{2}}{K_{0}}>\frac{n^{2}}{K}$ of them.
\end{obs}

In this case, the lower bound becomes
\begin{equation}
    \text{SDP}(M_{k}-\hat{M}_{K_{0}}) \geq \frac{n^{2}}{K}(M_{in}-\hat{M}_{K_{0},in}) > \frac{n^{2}}{K}\epsilon_{0}
\end{equation}
\underline{\textbf{Case 3:} $|\hat{M}_{K_{0},out}-M_{out}|>\epsilon_{0}$ and $|\hat{M}_{K_{0},in}-M_{in}|<\epsilon_{0}$.} In this case, we choose $Z'\in \text{PSD}_{1}$ as the matrix $\check{Z}$ but with also subdiagonals blocks with $-1$'s. Again, we put a non-reachable configuration where the entries $\hat{M}_{K_{0},in}$ are in the $+1$'s blocks and in the blocks $-1$'s it does not matter whether we put $\hat{M}_{K_{0},in}$ or $\hat{M}_{K_{0},out}$ because in these cases we get entries $-(M_{out}-\hat{M}_{K_{0},in})$ and $-(M_{out}-\hat{M}_{K_{0},out})$ both greater than or equal to $\hat{M}_{K_{0},out} - M_{out} > \epsilon_{0}$.

Then we obtain
\begin{equation}
\begin{split}
    \text{SDP}(M_{K}-\hat{M}_{K_{0}}) &\geq \langle M_{K} - \hat{M}_{K_{0}},Z' \rangle\\
    &>\frac{n^{2}}{K}(M_{in}-\hat{M}_{K_{0},in}) +  \floor[\bigg]{\frac{K}{2}}\frac{n^{2}}{K}(\hat{M}_{K_{0},out}-M_{out})\\
    &> \floor[\bigg]{\frac{K}{2}}\frac{n^{2}}{K}\epsilon_{0}
\end{split}
\end{equation}
because $M_{in}-M_{K_{0},in}>0$ as proved in the last case. 

Choosing $\epsilon_{0} \coloneqq \frac{2(M_{in}-M_{out})}{K^{2}(K_{0}+2)^{2}} $ we obtain lower bounds for each case, respectively, as
\begin{equation}
    \begin{split}
        &\frac{8(M_{in}-M_{out})n^{2}}{K^{2}K_{0}^{2}(K_{0}+2)^{2}}, \, \frac{2(M_{in}-M_{out})n^{2}}{K^{3}(K_{0}+2)^{2}}, \, \frac{2\left\lfloor \frac{K}{2} \right\rfloor(M_{in}-M_{out})n^{2}}{K^{3}(K_{0}+2)^{2}}
    \end{split}
\end{equation}
then, choosing the following lower bound for them
\begin{equation}
     \frac{2(M_{in}-M_{out})n^{2}}{K^{4}(K_{0}+2)^{2}}
\end{equation}
Considering the worst case of $K_{0}=K-1$ we define
\begin{equation}
    \gamma \coloneqq \frac{2(M_{in}-M_{out})n^{2}}{K^{4}(K+1)^{2}}
\end{equation}

Thus, we need the condition
\begin{equation}
\begin{split}
    &\gamma > 4n(1+\varepsilon)\sqrt{w_{+}}\\
    &M_{in}-M_{out}> \frac{2K^{4}(K+1)^{2}(1+\varepsilon)\sqrt{w_{+}}}{n}
\end{split}
\end{equation}
that is, the condition of the theorem statement.

Finally, continuing the calculation of \ref{eq:139} 
\begin{equation} 
    \begin{split}
        \P(\hat{T}_{n,K_0}(W_{G};\varepsilon)=0) &= \P(\text{SDP}(W_{G}-\hat{M}_{K_0})\leq 2n(1+\varepsilon)\sqrt{w_{+}})\\
        &= \P(\text{SDP}(W_{G} -M_{K}+M_{K} -\hat{M}_{K_0})\leq
        2n(1+\varepsilon)\sqrt{w_{+}}\\
        &\leq \P(\text{SDP}(M_{K} -\hat{M}_{K_0}) -\text{SDP}(M_{K}-W_{G})\leq
        2n(1+\varepsilon)\sqrt{w_{+}})\\
        &\leq \P(\gamma -\text{SDP}(M_{K}-W_{G}) \leq 2n(1+\varepsilon)\sqrt{w_{+}})+o(1)\\
        &\leq \P(\text{SDP}(\sqrt{w_{+}}B) > 2n(1+\varepsilon)\sqrt{w_{+}}-n\lambda_{0}) + o(1) = o(1)
    \end{split}
\end{equation}
where again in the last line we used the same reasoning as done in section \ref{sec:communitydet}.

Then we obtain
 \begin{equation}
         \P(\hat K_n < K) \leq \sum\limits_{K_0=1}^{K-1}\P(\hat{T}_{n,K_0}(W_{G};\varepsilon)=0) = o(1)
\end{equation}
\end{proof}

\section{Example: Zero-Inflated Gaussian model}\label{sec:application}

We work here with the interesting case of a zero-inflated Gaussian distribution, that is, a Gaussian times an independent Bernoulli distribution. For the weighted Stochastic Block Model we consider the following distributions
\begin{equation}\label{eq:model_sparse_gaussian}
    W_{G{ij}} \sim \begin{cases}
        Ber(\frac{a_n}{n})\mathcal{N}(\mu_{in},\tau_{1}^{2}),  \quad \text{if  $i \sim j$}\\
        Ber(\frac{b_n}{n})\mathcal{N}(\mu_{out},\tau_{2}^{2}), \quad \text{if  $i \not\sim j$}
    \end{cases}
\end{equation}
where $a_n,b_n\xrightarrow{n\to \infty}\infty$ but both are bounded by $Cn$ for some constant $C$ (we will see below why this is necessary). In other words, we are considering the more difficult case of a sparse random graph but with average degree going to infinity with $n$, similar to what was considered in \cite{montanari}.

First, we observe that 
\begin{equation}
    \E[W_{G{ij}}] \sim \begin{cases}
        \frac{a_n}{n}\mu_{in},  \quad \text{if  $i \sim j$}\\
        \frac{b_n}{n}\mu_{out}, \quad \text{if  $i \not\sim j$}
    \end{cases}
\end{equation}
such that to obey \textbf{Assumption \ref{assum:2}} we need to impose the following equation 
\begin{equation}
    a_n\mu_{in}> b_n\mu_{out}.
\end{equation}

Now, we check that both distribution for $W_{Gij}$, when centered, are sub-gamma distributions. Let us calculate for the case $i\sim j$, the other case is done similarly. We define the log moment-generating function for a centered random variable $\bar{X} \coloneqq X- \E X$ as
\begin{equation}
    \psi_{\bar{X}}(\lambda) = \log{\E[e^{\lambda\bar{X}}]} \quad \lambda\in \mathbb{R}.
\end{equation}
A two-sided sub-gamma distribution is characterized by (see section 2.4 of \cite{boucheron2013concentration})
\begin{equation}
    \psi_{\bar{X}}(\lambda) \leq \frac{\lambda^{2}\nu}{2(1-c|\lambda|)} \quad \forall\lambda \;\text{with} \; 0<|\lambda|<\frac{1}{c}
\end{equation}
where $\nu>0$ is called the \textbf{variance factor} and $c>0$ is the \textbf{scale parameter} and in this case we write $\bar{X} \sim sub\Gamma(\nu,c)$.

For the case where $i\sim j$ we are considering a random variable $X = BN$ where $B\sim Ber(\frac{a_{n}}{n})$ and $N \sim \mathcal{N}(\mu_{in},\tau_{1}^{2})$ and $B\perp N$. We calculate below the moment generating function of $\bar{X}$ and take the $\log$ in the end to find the correct parameters $\nu,c$ for this situation. Let $X'$ be a random variable with the same distribution as $X$ and $X'\perp X$, we calculate
\begin{equation}
\begin{split}
    \E[e^{\lambda(X-\E[X])}] &=\E[e^{\lambda\E(X-X'|X)}] \leq \E[\E[e^{\lambda(X-X')}|X]] = \E[e^{\lambda(X-X')}]\\
    &=\E[e^{2\lambda\left(\frac{1}{2}X-\frac{1}{2}X'\right)}] \leq \frac{1}{2}\E[e^{2\lambda X}+e^{-2\lambda X'}]
\end{split}
\end{equation}
where in the first inequality we used the conditional Jensen's inequality. Continuing
\begin{equation}
\begin{split}
    \E[e^{\lambda(X-\E[X])}] &\leq \frac{1}{2}\left[\sum_{j\geq 0}\frac{(2\lambda)^{j}}{j!}\E[X^{j}] + \frac{(-1)^{j}(2\lambda)^{j}}{j!}\E[X^{j}]\right] = \sum_{\substack{j\geq 0\\ j \, \text{even}}} \frac{(2\lambda)^{j}}{j!}\E[(BN)^{j}]\\
    &= \sum_{\substack{j\geq 0\\ j \, \text{even}}} \frac{(2\lambda)^{j}}{j!}\E[N^{j}]\E[B^{j}] = \frac{a_n}{n} \sum_{\substack{j\geq 0\\ j \, \text{even}}} \frac{(2\lambda)^{j}}{j!}\E[(\mu_{in}+\tau_{1}Z)^{j}]\\
    &\leq \frac{a_n}{n} \sum_{\substack{j\geq 0\\ j \, \text{even}}} \frac{(2\lambda)^{j}}{j!}\E[(|\mu_{in}|+\tau_{1}Z)^{j}] = \frac{a_n}{n} \sum_{\substack{j\geq 0\\ j \, \text{even}}} \frac{(2\lambda)^{j}}{j!}\left[\sum_{k=0}^{j} \binom{j}{k}|\mu_{in}|^{j-k}\tau_{1}^{k}\E[Z^{k}]\right]
\end{split}
\end{equation}
where $Z\sim \mathcal{N}(0,1)$ and we used the fact that $\E[B^{j}]=\frac{a_{n}}{n}$ for all $j \in \mathbb{N}$. Using that
\begin{equation}
    \E[Z^{k}] \sim \begin{cases}
        0,  \quad \text{if $k$ odd}\\
        \frac{k!}{
    \frac{k}{2}!2^{\frac{k}{2}}
}, \quad \text{if $k$ even}
    \end{cases}
\end{equation}
and observing that this is increasing in $k$ we obtain
\begin{equation}
\begin{split}
    \E[e^{\lambda(X-\E[X])}] &\leq \frac{a_n}{n} \sum_{\substack{j\geq 0\\ j \, \text{even}}} \frac{(2\lambda)^{j}}{j!}\left[\sum_{k=0}^{j} \binom{j}{k}|\mu_{in}|^{j-k}\tau_{1}^{k}\right]\frac{j!}{
    \frac{j}{2}!2^{\frac{j}{2}}}\\
    &=\frac{a_n}{n}\sum_{\substack{j\geq 0\\ j \, \text{even}}} \frac{\left(\frac{2}{\sqrt{2}}\lambda\right)^{j}}{\frac{j}{2}!}(|\mu_{in}| + \tau_{1})^{j}\\
    &\leq \frac{a_n}{n}\sum_{\substack{j\geq 0\\ j \, \text{even}}}\left(\frac{2}{\sqrt{2}}\lambda\right)^{j}(|\mu_{in}|+\tau_{1})^{j}\\
    &= \frac{a_n}{n} + \frac{a_n}{n}\sum_{\substack{j\geq 2\\ j \, \text{even}}}\left(\frac{2}{\sqrt{2}}\lambda\right)^{j}(|\mu_{in}|+\tau_{1})^{j}\\
    &= \frac{a_n}{n} + \frac{a_n}{n}\left(\frac{2}{\sqrt{2}}\lambda\right)^{2}(|\mu_{in}|+\tau_{1})^{2}\sum_{\substack{j\geq 0\\ j \, \text{even}}}\left(\frac{2}{\sqrt{2}}\lambda\right)^{j}(|\mu_{in}|+\tau_{1})^{j}\\
    &\leq 1 + \left(\frac{2}{\sqrt{2}}\lambda\right)^{2}(|\mu_{in}|+\tau_{1})^{2}\frac{\frac{a_n}{n}}{1-\left(\frac{2}{\sqrt{2}}\lambda\right)(|\mu_{in}|+\tau_{1})}\\
    &= 1 + \frac{4\lambda^{2}\frac{a_n}{n}(|\mu_{in}|+\tau_{1})^{2}}{2(1-\sqrt{2}(|\mu_{in}|+\tau_{1})\lambda)}
\end{split}
\end{equation}
where we supposed that
\begin{equation}
    \left(\frac{2}{\sqrt{2}}\lambda\right)(|\mu_{in}|+\tau_{1}) < 1 \iff \lambda < \frac{1}{\sqrt{2}(|\mu_{in}|+\tau_{1})}
\end{equation}
or the same substituting $\lambda$ by $|\lambda|$ such that bounding $\lambda$ by its absolute value, in the last line on the previous equation we end up with $|\lambda|$ in the denominator.

Finally, taking the logarithm we obtain
\begin{equation}
    \psi_{\bar{X}}(\lambda) \leq \frac{4\lambda^{2}\frac{a_n}{n}(|\mu_{in}|+\tau_{1})^{2}}{2(1-\sqrt{2}(|\mu_{in}|+\tau_1)\lambda)}
\end{equation}
such that in our case the \textbf{variance factor} is $\nu_{in} \coloneqq 4\frac{a_n}{n}(|\mu_{in}|+\tau_{1})^{2} $ and the \textbf{scale parameter} is $c_{in} = \sqrt{2}(|\mu_{in}|+\tau_{1})$.

Given the above we can also check \textbf{Assumption \ref{assum:1}} for $p=2$, that is, $\V[X]\leq \nu_{in}$. In fact, using the general inequality for independent random variables $X_{1}$ and $X_{2}$
\begin{equation}
    \V[X_{1}X_{2}] = \V[X_{1}]\V[X_{2}]+\V[X_{1}]\E[X_{2}^{2}] + \V[X_{2}]\E[X_{1}^{2}]
\end{equation}
we obtain
\begin{equation}
\begin{split}
    \V[X] = \V[BN] &= \frac{a_n}{n}\left(1-\frac{a_n}{n}\right)\tau_{1}^{2} + \frac{a_{n}}{n}\left(1-\frac{a_n}{n}\right)(\mu_{in}^{2}+\tau_{1}^{2}) + \tau_{1}^{2}\left[\frac{a_n}{n}\left(1-\frac{a_n}{n}\right)+ \frac{a_{n}^{2}}{n^{2}}\right]\\
    &= \frac{a_n}{n}\left(1-\frac{a_n}{n}\right)\left[\mu_{in}^{2}+3\tau_{1}^{2}\right] + \frac{a_{n}^{2}}{n^{2}}\tau_{1}^{2}\\
    &\leq \frac{a_n}{n}(\mu_{in}^{2}+4\tau_{1}^{2}) \leq 4\frac{a_n}{n}(|\mu_{in}|+\tau_{1})^{2} = \nu_{in}
\end{split}
\end{equation}
as we wanted.

We also check \textbf{Assumption \ref{assum:1}} for $p\geq 3$, that is
\begin{equation}
\E[|\bar{W}_{Gij}|^{p}] \leq \frac{c_{ij}^{p-2}\nu_{ij}p!}{2}.
\end{equation}
Let us write the right-hand side of the condition (RHS) with the $\nu_{ij}$ and $c_{ij}$ of this case. Again, we write considering we have the distribution coming from the inside of communities edges, with the ``out" case having the same bound. We have
\begin{equation}\label{eq:176}
\begin{split}
    \frac{c_{ij}^{p-2}\nu_{ij}p!}{2} &= \frac{[\sqrt{2}(|\mu_{in}|+\tau_{1})]^{p-2}4(\frac{a_n}{n})(|\mu_{in}|+\tau_{1})^2p!}{2}\\
    &= \frac{2^{\frac{p}{2}-1}4}{2}\frac{a_n}{n}(|\mu_{in}|+\tau_{1})^pp!\\
    &=2^{\frac{p}{2}}\frac{a_n}{n}(|\mu_{in}|+\tau_{1})^pp!.
\end{split}
\end{equation}
We can also calculate the left-hand side (LHS) obtaining
\begin{equation}\label{eq:177}
\begin{split}
    \E[|\bar{W}_{Gij}|^p] &= \E[|BN-\E[BN]|^p]\\
    &\leq 2^p(\E[|BN|^p]+\E[|\E[BN]|^p])\\
    &\leq 2^{p+1}\E[|BN|^p] = 2^{p+1}(\E[|B|^p]\E[|N|^p])\\
    &= 2^{p+1}\frac{a_n}{n}\E[|N|^p]
\end{split}  
\end{equation}
Considering $Z \sim \mathcal{N}(0,1)$, we obtain
\begin{equation}
    \begin{split}
        \E[|N|^p] &= \E[(|\mu_{in}+\tau_{1}Z|^p)]\\
        &\leq \E[(|\mu_{in}|+\tau_{1}|Z|)^p]\\
        &= \sum_{k=0}^{p}\binom{p}{k}|\mu_{in}|^{p-k}\tau_{1}^k\E[|Z|^k]\\
        &\leq \sum_{k=0}^{p}\binom{p}{k}|\mu_{in}|^{p-k}\tau_{1}^k\left[\frac{2^{\frac{p}{2}}}{\sqrt{\pi}}\Gamma\left(\frac{p+1}{2}\right)\right] = (|\mu_{in}|+\tau_{1})^p\left[\frac{2^{\frac{p}{2}}}{\sqrt{\pi}}\Gamma\left(\frac{p+1}{2}\right)\right]
    \end{split}
\end{equation}
where we used the fact that the moments $\E[|Z|^k]$ are increasing such that we maintained only the term $k=p$. We split the calculation in the case where we have $p$ even and odd.

\underline{Case $p$ even:} In this case, $\frac{p}{2}\in \mathbb{N}$ such that
\begin{equation}
\frac{2^{\frac{p}{2}}}{\sqrt{\pi}}\Gamma\left(\frac{p}{2}+\frac{1}{2}\right) = \frac{2^{\frac{p}{2}}}{\sqrt{\pi}}\frac{p!\sqrt{\pi}}{\frac{p}{2}!2^p} = \frac{p!}{\frac{p}{2}!2^{\frac{p}{2}}}
\end{equation}
and equation \ref{eq:177} becomes
\begin{equation}
\begin{split}
    \E[|\bar{W}_{Gij}|^p] &\leq 2^{p+1}\frac{a_n}{n}\frac{p!}{\frac{p}{2}!2^{\frac{p}{2}}}(|\mu_{in}|+\tau_{1})^p.
\end{split}
\end{equation}
Comparing with equation \ref{eq:176}, to have Assumption \ref{assum:1} it suffices that
\begin{equation}
    \frac{2^{p+1}}{\frac{p}{2}!2^{\frac{p}{2}}} \leq 2^{\frac{p}{2}}
\end{equation}
for $p\geq 4$ (remember we are in the case $p$ is even).

For $p=4$ we have
\begin{equation}
    \frac{2^{5}}{2!2^2}=4 \leq 4 = 2^{\frac{4}{2}}
\end{equation}
such that the relation holds. For $p=6$ we also have it holding
\begin{equation}
    \frac{2^7}{3!2^3} = \frac{16}{6} \leq 8 = 2^\frac{6}{2}.
\end{equation}
Now, for $p\geq8$, we have $\frac{p}{2}\geq4$ such that we can use $\frac{p}{2}!>2^{\frac{p}{2}}$ such that
\begin{equation}
    \frac{2^{p+1}}{\frac{p}{2}!2^{\frac{p}{2}}}< \frac{2^{p+1}}{2^{\frac{p}{2}}2^{\frac{p}{2}}} = 2 \leq 2^{\frac{p}{2}}
\end{equation}
and the relation also holds.

\underline{Case $p$ odd:} In this case $\frac{p+1}{2}\in \mathbb{N}$, such that
\begin{equation}
    \Gamma\left(\frac{p+1}{2}\right) = \left(\frac{p-1}{2}\right)!
\end{equation}
and then to have equation \ref{eq:176} holding it suffices to have
\begin{equation}
    2^{p+1}\left(\frac{p-1}{2}\right)! \leq 2^{\frac{p}{2}}p!.
\end{equation}
As $p\geq 3$ and is odd we can rewrite it as $p=2m+1$ with $m\geq1$, then we need
\begin{equation}\label{eq:187}
\begin{split}
    &2^{2m+2}m!\leq 2^{m+\frac{1}{2}}(2m+1)!\\
    \Leftrightarrow &2^{m+\frac{3}{2}} \leq (2m+1)(2m)(2m-1)...(m+2)(m+1) =: \ell   
\end{split}
\end{equation}
Observing that $\ell$ is the product of $m+1$ terms which are $\geq m$, we have for $m\geq 3$
\begin{equation}
    2^{m+\frac{3}{2}}\leq m^{m+1} \leq \ell
\end{equation}
such that condition \ref{eq:187} holds. It remains to check for $m=1,2$. For $m=1$, we calculate and obtain
\begin{equation}
\begin{split}
    &2^{2+2}1! \leq 2^{1+\frac{1}{2}}3!\\
    \Leftrightarrow &2^4 \leq 16.9 \leq 2\sqrt{2}3! 
\end{split}
\end{equation}
that is, the condition holds. We also have it holding
\begin{equation}
    2^{4+2}2! = 128 \leq 678,8 \leq 4\sqrt{2}5!
\end{equation}
and this finishes the checking of Assumption \ref{assum:1}.

\begin{obs}
   We could have first checked Assumption \ref{assum:1}, defined $\nu_{ij}$ and $c_{ij}$ in the same way, and then concluded that the weights are sub-gamma with these parameters via Bernstein's inequality (Section 2.8 of \cite{gabor}). However, we chose our approach of finding the parameters directly because then we avoid the need to guess them.
\end{obs}

We now check that \textbf{Assumption \ref{assum:3}} holds. Denoting the variance factor for the entry $W_{Gij}$ as $\nu_{ij}$ we have that
\begin{equation}
    n\max_{ij} \nu_{ij} = n \max\left\{4\frac{a_n}{n}(|\mu_{in}|+\tau_{1})^{2},4\frac{b_n}{n}(|\mu_{out}|+\tau_{2})^{2}\right\}
\end{equation}
such that choosing 
\begin{equation}\label{eq:w_gaussian_bernoulli}
    w_{+} \coloneqq \max\left\{4a_{n}(|\mu_{in}|+\tau_{1})^{2}, 4b_{n}(|\mu_{out}|+\tau_{2})^{2}\right\}
\end{equation}
satisfies the first equation of Assumption \ref{assum:3} and also $w_{+} \in (4,n\cdot const)$, $w_{+}\xrightarrow{n\to \infty}\infty$ by the conditions on $a_{n},b_{n}$. It remains to check that
\begin{equation}
     \theta \coloneqq \frac{\max_{1\leq i \leq j \leq n}c_{ij}}{\sqrt{w_{+}}} \xrightarrow{n\to \infty} 0,
\end{equation}
in fact, in this case we have
\begin{equation}
    \theta = \frac{\max\{\sqrt{2}(|\mu_{in}|+\tau_{1}), \sqrt{2}(|\mu_{out}|+\tau_{2})\}}{\sqrt{\max\left\{4a_{n}(|\mu_{in}|+\tau_{1})^{2}, 4b_{n}(|\mu_{out}|+\tau_{2})^{2}\right\}}} \xrightarrow{n\to \infty} 0
\end{equation}
because of the conditions $a_n,b_n\xrightarrow{n\to \infty}\infty$ and the fact that the other terms are of order $O(1)$.

To complete the section, we also need to show that the SBM with zero-inflated Gaussian distribution with the above parameters, satisfies the conditions of Theorems \ref{Theorem2},\ref{Theorem4} and \ref{Theorem7}. Actually, we only need to check one of the condition as the others are of the same order. Let us then check that the condition of Theorem \ref{Theorem2}, that is, \textbf{equation \ref{eq:29}} is satisfied. The right-hand side of it gives
\begin{equation}
    \frac{n^{2}}{2}(M_{in}-M_{our})=\frac{n^{2}}{2}\left[\frac{a_n}{n}\mu_{in}-\frac{b_{n}}{n}\mu_{out}\right] = \left(\frac{a_n\mu_{in}-b_n\mu_{out}}{2}\right)n 
\end{equation}
and the left-hand side
\begin{equation}
    4n(1+\delta)\sqrt{\max\left\{4a_{n}(|\mu_{in}|+\tau_{1})^{2}, 4b_{n}(|\mu_{out}|+\tau_{2})^{2}\right\}}
\end{equation}
such that we for $n$ large enough we have the condition satisfied as the RHS is of order $n\max\{{a_{n},b_{n}}\}$ and the left hand side of order $n\sqrt{\max\{{a_{n},b_{n}}\}}$.

\section{Computational simulations of the tests and estimations}\label{sec:simulations}

In this section, we empirically study the performance of the community detection method proposed in this work on a finite sample network. Our goal is to compare the proposed method, based on a SDP solution, with methods grounded in other principles, such as likelihood-based and spectral approaches. To this end, we selected three different methods: (1) spectral clustering (SC), originally proposed for binary networks by \cite{lei2015consistency} and applied directly to the weighted network; (2) the pseudo-likelihood method (Pseudo) for weighted Gaussian networks introduced in \cite{cerqueira2023}; and (3) a discretization-based approach (DB), which discretizes the network weights before applying a spectral clustering algorithm, as proposed by \cite{xu2018optimal}. The discretization level used in the DB approach is set to $\lfloor 0.4 (\log\log n)^4\rfloor$, and for the initialization of the Pseudo method, we used the SC. 

The metric used to compare the community detection methods evaluates the total discrepancy between the estimated and true membership matrices, capturing the alignment accuracy of the estimated communities as defined in \eqref{eq:def_error}. This metric takes values between zero and $n^2$. The error is computed as the average over 100 replications of the model.

To implement the community detection method based on the SDP, we utilized the \texttt{CVXR} package in R to solve the SDP problem in \eqref{def:SDP} \cite{package_cvxr}.

We simulated networks from the zero-inflated Gaussian model, defined in Equation \eqref{eq:model_sparse_gaussian}, with $K=4$, $\mu_{in}=5$, $\mu_{out}=2$, $\tau_{1}^{2}=\tau_{2}^{2}=1$ and $a_n/n=b_n/n=\rho$. In this setting, the parameter $\rho$ controls the sparsity of the network, with smaller values of $\rho$ corresponding to a sparser network. Varying the values of $\rho$ for different network sizes, $n=100,200$, we observe in Figure \ref{fig:error_sparse_gaussian} that, as expected, as $\rho$ increases (i.e., as the network becomes less sparse), the estimation error decreases. Fixing $\rho=0.4$, we observe in Figure \ref{fig:error_sparse_gaussian_min} that the estimation error decreases as the difference $\mu_{in} - \mu_{out}$ increases. It is important to emphasize that the DB method depends on the choice of the discretization level, which can be influenced by both $\rho$ and $\mu_{in} - \mu_{out}$. Selecting the optimal value for this parameter is beyond the scope of this work; therefore, the DB method did not perform well in our analysis, as we used a fixed discretization level for all simulation scenarios.

\begin{figure}[h]
\begin{subfigure}{.5\textwidth}
  \centering
  \includegraphics[scale=0.55]{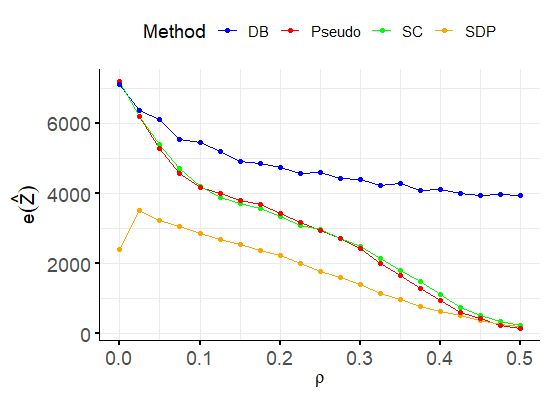}
  \caption{$n=100$.}
\end{subfigure}%
\begin{subfigure}{.5\textwidth}
  \centering
  \includegraphics[scale=0.55]{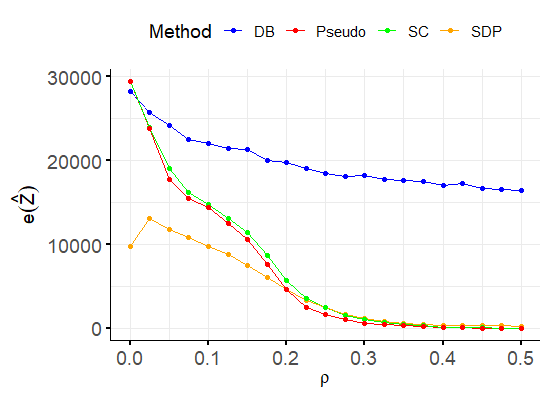}
  \caption{$n=200$.}
\end{subfigure}
\caption{Estimation error computed for different community detection approaches as function of the sparsity parameter of the model $\rho$ with $\mu_{in}-\mu_{out}=3$.}
\label{fig:error_sparse_gaussian}
\end{figure}

\begin{figure}[h]
\begin{subfigure}{.5\textwidth}
  \centering
  \includegraphics[scale=0.55]{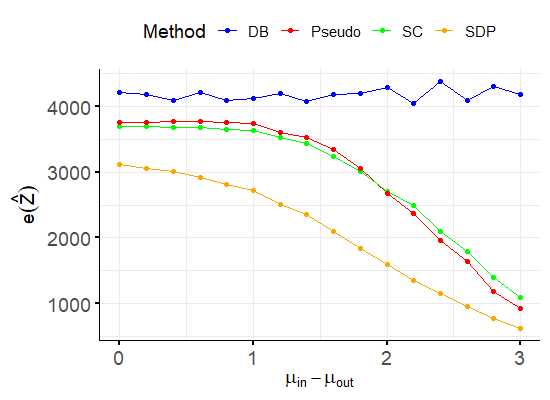}
  \caption{$n=100$.}
\end{subfigure}%
\begin{subfigure}{.5\textwidth}
  \centering
  \includegraphics[scale=0.55]{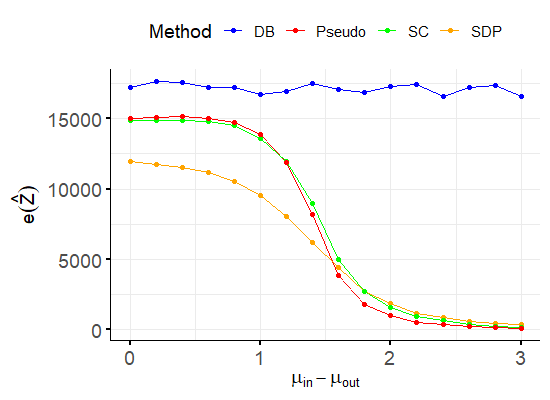}
  \caption{$n=200$.}
\end{subfigure}
\caption{Estimation error computed for different community detection approaches as function of the difference of the mean weights $\mu_{in}-\mu_{out}$ with $\rho=0.4$. }
\label{fig:error_sparse_gaussian_min}
\end{figure}

We also study the performance of the estimator of the number of communities based on the sequential hypothesis test \eqref{def:hat_k}. We replace $w_{+}$ in Equation \eqref{eq:w_gaussian_bernoulli} with its value obtained by plugging in the estimates of $\mu_{in}$, $\mu_{out}$, and $\rho$ from the spectral clustering method. For all simulations we fixed $\tau_{1}^{2}=\tau_{2}^{2}=1$. We observe from Figure \ref{fig:error_khat_sparse_gaussian} that, by fixing the difference of the mean weights $\mu_{in} - \mu_{out} = 4$, the convergence to the true number of communities is faster for $K = 3$ than for $K = 4$. Moreover, as the number of nodes in the network increases, convergence holds for sparser networks. This reinforces the necessary assumption in Theorem \ref{Theorem7}, which states a lower bound for the difference $\mu_{in} - \mu_{out}$, depending on the true number of communities and the sparsity parameter $\rho$ through $w_{+}$. This necessary assumption can also be observed in Figure \ref{fig:error_khat_sparse_gaussian_min}, where the convergence to the true number of communities is faster for $K=3$ with $\rho=0.5$ than for $K=4$ with $\rho=0.8$, as the difference $\mu_{in} - \mu_{out}$ varies.

\begin{figure}[h]
\begin{subfigure}{.5\textwidth}
  \centering
  \includegraphics[scale=0.5]{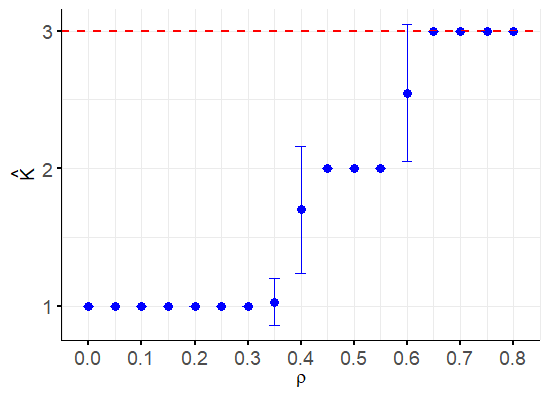}
  \caption{$n=100$.}
\end{subfigure}%
\begin{subfigure}{.5\textwidth}
  \centering
  \includegraphics[scale=0.5]{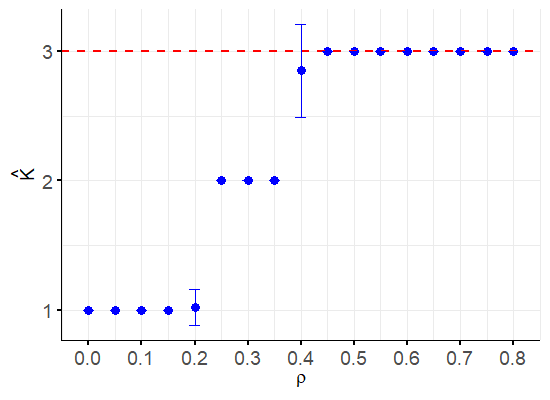}
  \caption{$n=200$.}
\end{subfigure}
\begin{subfigure}{.5\textwidth}
  \centering
  \includegraphics[scale=0.5]{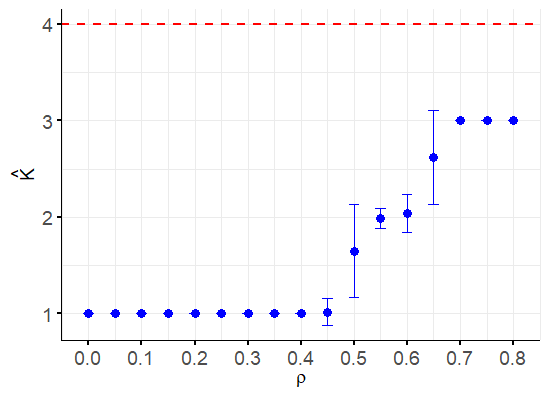}
  \caption{$n=100$.}
\end{subfigure}%
\begin{subfigure}{.5\textwidth}
  \centering
  \includegraphics[scale=0.5]{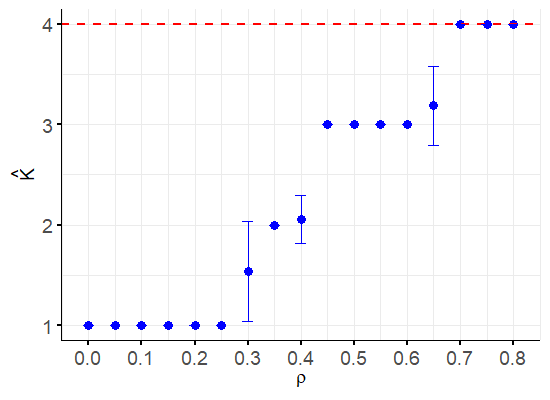}
  \caption{$n=200$.}
\end{subfigure}
\caption{The mean and the one standard deviation error bars for the estimated
number of communities as function of the sparsity parameter for $| \mu_{in} - \mu_{out}| =4$. (A) and (B) model with three communities and (C) and (D) model with four communities. }
\label{fig:error_khat_sparse_gaussian}
\end{figure}

\begin{figure}[h]
\begin{subfigure}{.5\textwidth}
  \centering
  \includegraphics[scale=0.5]{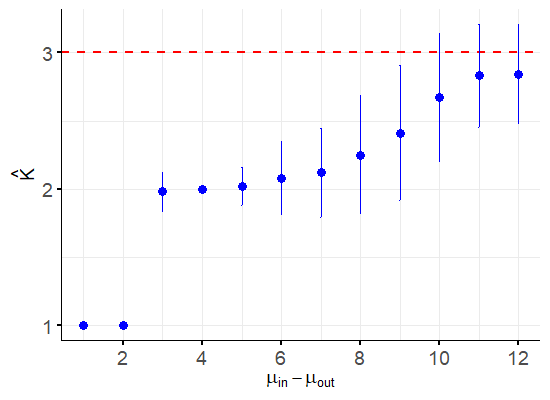}
  \caption{$n=100$.}
\end{subfigure}%
\begin{subfigure}{.5\textwidth}
  \centering
  \includegraphics[scale=0.5]{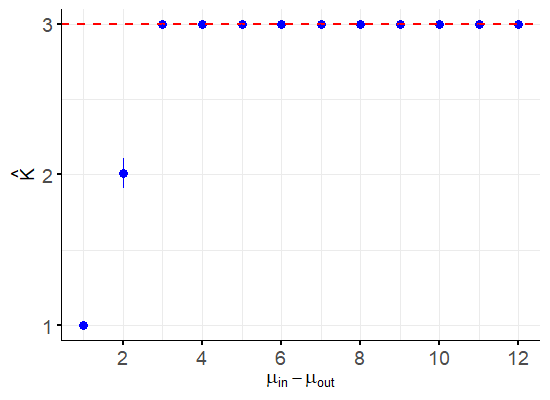}
  \caption{$n=200$.}
\end{subfigure}
\begin{subfigure}{.5\textwidth}
  \centering
  \includegraphics[scale=0.5]{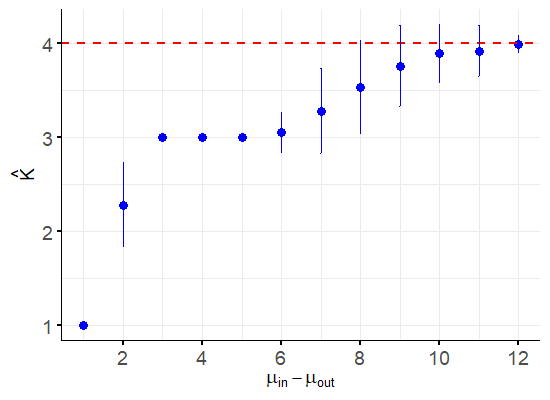}
  \caption{$n=100$.}
\end{subfigure}%
\begin{subfigure}{.5\textwidth}
  \centering
  \includegraphics[scale=0.5]{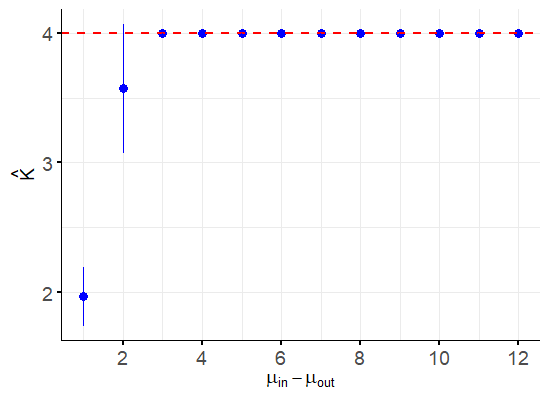}
  \caption{$n=200$.}
\end{subfigure}
\caption{The mean and the one standard deviation error bars for the estimated
number of communities as function of the difference of the mean weights $\mu_{in}-\mu_{out}$ . (A) and (B) model with three communities with $\rho=0.5$ and (C) and (D) model with four communities with $\rho=0.8$. }
\label{fig:error_khat_sparse_gaussian_min}
\end{figure}

\section{Discussion and final considerations}\label{sec:discussion}

In this paper, we explored the use of semidefinite programming in the context of community detection. Specifically, we addressed the problem of distinguishing between any two possible different numbers of communities present in a balanced weighted Stochastic Block Model and the recovery of community memberships. Our work is an extension of the work in \cite{montanari} in the sense that we considered the weighted version of the problem and tests that can distinguish not only between a homogeneous Erdős–Rényi graph and the presence of two communities (as well as generalizations to more than two communities) but also between any two arbitrary numbers of communities. Additionally, the sequential use of the tests that we built to estimate the number of communities follows the same idea as in \cite{lei2016_test}, however, we do not use a statistic coming from a GOE matrix directly but from the SDP function of it. 

As a separate point, the universality result we established for the SDP is more general. It is different from the similar result in \cite{montanari} in the sense that they proved an approximation of the SDP function of a Bernoulli matrix with a specific GOE equivalent matrix. In our case, though, we proved the approximation between a symmetric random matrix with sub-gamma entry distribution with its equivalent Gaussian matrix.

Finally, we illustrated our findings with simulation studies using the zero-inflated Gaussian model. We compared the performance of the proposed community detection estimators with other known approaches in literature, using spectral and likelihood-based methods. We observe that the SDP-based method outperform the other ones for almost all sparsity levels (keeping in mind the fine-tuning needed for the DB method). We also showed that the error decreases as the difference between within-community and between-community weights increases, as expected from Theorem \ref{Theorem5}. For the problem of estimating the number of communities, we also obtained results consistent with Theorems \ref{Theorem6} and \ref{Theorem7}, showing that estimating the true number of communities can be challenging when this number is large, as it requires a greater difference between within-community and between-community weights, and that the challenge also depends on the network size.

While our results are applicable to more general setting than similar previous results, there are some further possible generalizations. For example, we could ask what happens in case the means are not homogeneous inside and outside communities, or what we need to change in our analysis in case the communities are not balanced. We could also ask how we should modify our analysis in the case the number of communities depends on the number of vertices. All these considerations approximates the theoretical model to real-world data and then are desired to be studied and analyzed. We hope the results here inspire new research in this area.

\appendix

\section{Universality of the SDP function} \label{sec:appA}
We prove the following theorem from the first section.

\begin{teo}
Let $X \in \mathbb{R}^{n\times n}$ be a symmetric random matrix with independent centered entries $X_{ij} \sim sub\Gamma(\nu_{ij},c_{ij})$  with variance $\sigma_{ij}^2$ and let $D \in \mathbb{R}^{n \times n}$ be independent of $X$ and the analogous Gaussian matrix, \textit{i.e.}, a symmetric random matrix with independent entries $D_{ij}$ of distribution $\mathcal{N}(0,\sigma_{ij}^{2})$. Then, there exists $C>0$ such that for $8\leq k \leq n$, $\epsilon \in (0,\frac{1}{2})$, and $\beta>0$, with probability at least $1-\max(4^{-n},n^{-\delta})$,  we have
\begin{equation}\label{eq:147}
\begin{split}
    \frac{1}{n}|\text{SDP}(X) - \text{SDP}(D)| &\lesssim \left[\frac{1}{k-1}+\epsilon\right]\left(\sqrt{\sum_{i\leq j}\frac{\nu_{ij}}{n}} + c\right) + \frac{k-1}{2\beta}\log{\left(\frac{C}{\epsilon}\right)}\\
    &+ \sqrt{\sum_{i\leq j}\nu_{ij}n^{\delta-2}} + \frac{\beta^{2}}{n}\sum_{i\leq j}(c_{ij}\nu_{ij}+\nu_{ij}^{\frac{3}{2}})
\end{split}
\end{equation}
where $c \coloneqq \max_{1\leq i \leq j}c_{ij}$ and $\lesssim$ means that the inequality is true up to a multiplicative universal constant.
\end{teo} 

To prove the theorem, we will need to define the following low-rank approximation of SDP.
\begin{defn}Let $M \in \mathbb{R}^{n \times n}$, and $k\in \mathbb{N}$, we define
\begin{equation}
    \text{OPT$_{k}$}(M) := \max{\{\langle M,X \rangle : X \succeq 0,  X_{ii}=1 \forall i \in [n], \text{rank}(X)\leq k\}}.
\end{equation}
\end{defn}

We need also to define the following function known as the ``log-partition function" from statistical mechanics.
\begin{defn} Let $M \in \mathbb{R}^{n \times n}$,$\beta > 0$ and $k \in \mathbb{N}$, we define
\begin{equation}
    \Phi(\beta,k;M) := \frac{1}{\beta} \log{\left(\int\exp\left(\beta \sum_{i,j=1}^{n} M_{ij} \langle \sigma_{i},\sigma_{j}\rangle\right)d\nu(\mathbf{\sigma})\right)}
 \end{equation}
where $\sigma=(\sigma_1, \sigma_2, ..., \sigma_n) \in (\mathbb{S}^{k-1})^{n}$ and $d\nu$ is the uniform and normalized measure on $(\mathbb{S}^{k-1})^{n}$.
\end{defn}
The purpose of introducing this function is to have a smooth approximation of  the OPT$_{k}$ function so that we can use a already known universality result for smooth functions in one of the steps of the proof.

Using the triangle inequality
\begin{equation}\label{eq:54}
    \begin{split}
        |\text{SDP}(D)- \text{SDP}(X)| &\leq |\text{SDP}(D)-\text{OPT}_{k}(D)|\\
        &+|\text{OPT}_{k}(D)-\Phi(\beta,k;D)|\\
          &+ |\Phi(\beta,k;D) - \E[\Phi(\beta,k;D)]|\\
         &+ |\E[\Phi(\beta,k;D)] - \E[\Phi(\beta,k;X)]|\\
         &+ |\E[\Phi(\beta,k;X)]-\Phi(\beta,k;X)|\\
         &+ |\Phi(\beta,k;X) - \text{OPT}_{k}(X)|\\
         &+|\text{OPT}_{k}(X)-\text{SDP}(X)|
    \end{split}
\end{equation}
we apply the following results to obtain the proof of the theorem. 

\begin{obs}
In what follows, we always consider $M\in \{X,D\}$, which is a matrix with sub-gamma entries because $X$ has sub-gamma entries $X_{ij}\sim sub\Gamma(\nu_{ij},c_{ij})$ and $D$ has the corresponding Gaussian entries, which are also sub-gamma but with $c_{ij}=0$ for all entries.
\end{obs}

First, to approximate SDP$(M)$ and OPT$_{k}(M)$ we use the following lemma.

\begin{lema}\label{lemma:1}
With probability at least $1-4^{-n}$, we have
\begin{equation}
    |\text{SDP}(M)-\text{OPT}_{k}(M)| \lesssim \frac{1}{k-1}\left(\sqrt{\sum_{i\leq j}\nu_{ij}n} + cn\right)
\end{equation}
where $c\coloneqq \max_{i\leq j}c_{ij}$
\end{lema}
\begin{proof}
Using equation $2$ from \cite{sdpgrothendieck} we obtain
\begin{equation}
    |\text{SDP}(M)-\text{OPT}_{k}(M)| \leq \frac{1}{k-1}(\text{SDP}(M)+\text{SDP}(-M)).
\end{equation}

Now, by fact $3.2$ and equation $3.3$ from \cite{guedon2016community} we also obtain
\begin{equation}\label{eq:154}
    \text{SDP}(M) \leq K_{G}\lVert M \rVert_{\infty \to 1}
\end{equation}
where
\begin{equation}
    \lVert M \rVert_{\infty \to 1} =\sup_{x,y \in \{-1,+1\}^{n}} \langle x,My\rangle.
\end{equation}
Then
\begin{equation}
    |\text{SDP}(M)-\text{OPT}_{k}(M)| \leq \frac{2K_{G}}{k-1}\lVert M \rVert_{\infty \to 1}.
\end{equation}
Given that the entries of $M$ are sub-gamma random variables with parameters $\nu_{ij}$ and $c_{ij}$, this means
\begin{equation}
    \psi_{M_{ij}} \leq \frac{\lambda^{2}\nu_{ij}}{2(1-c_{ij}|\lambda|)} 
\end{equation}
where this is valid for all $0<|\lambda|<\frac{1}{c_{ij}}$ and $\psi$ is the $\log$ of the moment-generating function. Now, let us calculate this function for $\langle x,My\rangle$ for fixed $x,y \in \{-1,+1\}^{n}$. Using the symmetry of $M$ and the independence of its entries we get
\begin{equation}
    \begin{split}
        \psi_{\langle x,My\rangle}(\lambda) &= \psi_{\sum_{i<j}M_{ij}(x_{i}y_{j}+x_{j}y_{i})+\sum_{i}M_{ii}x_{i}y_{i}}(\lambda)\\
        &= \sum_{i<j} \psi_{M_{ij}}(\lambda(x_{i}y_{j}+x_{j}y_{i})) + \sum_{i}\psi_{M_{ii}}(\lambda(x_{i}y_{i}))
    \end{split}
\end{equation}
then using that $x_{i}y_{j}+x_{j}y_{i} \in [-2,2]$, $x_{i}y_{i} \in [-1,1]$ and the definition of $c$ we obtain
\begin{equation}
    \begin{split}
        \psi_{\langle x,My\rangle} \leq \sum_{i<j}\frac{4\lambda^{2}\nu_{ij}}{2(1-c2|\lambda|)} + \sum_{i}\frac{\lambda^{2}\nu_{ii}}{2(1-c|\lambda|)} \leq \frac{2\lambda^{2}\left(\sum_{i\leq j}\nu_{ij}\right)}{1-2c|\lambda|}.
    \end{split}
\end{equation}
Then 
\begin{equation}
    \langle x,My \rangle \sim sub\Gamma\left(2\sum_{i\leq j}\nu_{ij},2c\right)
\end{equation}
for all $x,y \in \{-1,+1\}^{n}$.

Finally, using equation $5.5$ from \cite{zhang2020concentration} we get
\begin{equation}
    \begin{split}
        \P(\lVert M \rVert_{\infty\to 1}\geq t) &= \P\left(\sup_{x,y\in\{-1,+1\}^{n}} \langle x,My \rangle \geq t \right)\\
        &\leq \P\left(\cup_{x,y \in \{-1,+1\}^{n}} \langle x,My \rangle \geq t\right)\\
        &\leq \sum_{x,y\in\{-1,+1\}^{n}}\P\left(\langle x,My \rangle \geq t\right)\\
        &\leq 4^{n}\exp{\left(\frac{-t^{2}}{4\sum_{i\leq j}\nu_{ij}+4ct}\right)} \leq 4^{-n}
    \end{split}
\end{equation}
choosing $t=C\left(\sqrt{\sum_{i\leq j}\nu_{ij}n}+cn\right)$ for a appropriate constant $C$ the result follows.
\end{proof}

Second, we have the following approximations of $\text{OPT}_{k}$ by the log-partition modification of lemma $3.2$ in \cite{montanari} where we were able to bound the difference by $\lVert M\rVert_{\infty \to 1}$ instead of $\lVert M \rVert_{\infty \to 2}$ using a PAC-Bayes bound.

\begin{lema}\label{lemma:2} For $8\leq k \leq n$, $\epsilon \in(0,\frac{1}{2})$ and $\beta>0$ we have with probability at least $1-4^{-n}$
\begin{equation}
    \left| \Phi(\beta,k;M)-\text{OPT$_{k}$}(M) \right| \lesssim \frac{n(k-1)}{2\beta}\log{\left(\frac{C}{\epsilon}\right)} + \epsilon\left(\sqrt{\sum_{i\leq j}\nu_{ij}n}+cn\right)
\end{equation}
\end{lema}

For this proof we will need two preliminary results. One is the following.

\begin{prop}\label{prop:norm} There exists a universal constant $C>0$ such that for any $n,k\in\mathbb{N}$, any $M\in \mathbb{R}^{n\times n}$ and any ${\bf \sigma}\in (\mathbb{S}^{k-1})^n$, 
\begin{equation}
    \left|\sum_{i,j=1}^nM_{i,j}\langle \sigma_i,\sigma_j\rangle\right|\leq C\,\|M\|_{\infty\to 1}.
\end{equation} 
\end{prop}
\begin{proof}Given $\sigma$,  Lemma 4.2 in \cite{alon2004approximating} implies that there one can choose random vectors $\xi,\eta\in \{-1,+1\}^n$ such that for all pairs $1\leq i,j\leq n$
\begin{equation}
    C\E\,\xi_i\eta_j = \langle \sigma_i,\sigma_j\rangle,
\end{equation}
with 
\begin{equation}
    C=\frac{\pi}{2\ln(1+\sqrt{2})}\in (0,+\infty).
\end{equation}
Linearity of expectation implies
\begin{equation}
    \left|\sum_{i,j=1}^nM_{i,j}\langle \sigma_i,\sigma_j\rangle\right| = C\left|\E\left[\sum_{i,j=1}^nM_{i,j}\,\xi_i\eta_j\right]\right|\leq C\,\E\left[\left|\sum_{i,j=1}^nM_{i,j}\,\xi_i\eta_j\right|\right]
\end{equation}

The expression inside the expectation is upper bounded by
\begin{equation}
\left|\sup_{x,y\in \{-1,+1\}^{n}}x^TMy\right| = \|M\|_{\infty\to 1}.
\end{equation}
\end{proof}

\begin{prop}\label{prop:random}Given ${\bf \sigma}\in (\mathbb{S}^{k-1})^n$ and $\epsilon\in (0,1/2]$ with $k\geq 8$. Let $\nu$ be the uniform measure over $(\mathbb{S}^{k-1})^n$ and let $\nu_{\sigma,\epsilon}$ denote $\nu$ conditioned to the set
\begin{equation}
    A_{\sigma,\epsilon}:=\{\eta\in (\mathbb{S}^{k-1})^n\,:\, \forall i\in[n],\, \langle \eta_i,\sigma_i\rangle\geq 1-\epsilon\}.
\end{equation}
Then 
\begin{equation}
    \nu(A_{\sigma,\epsilon})\geq \frac{\epsilon^{\frac{(k-1)n}{2}}}{(6\sqrt{k})^n}
\end{equation}
and 
\begin{equation}
    \E_{\eta\sim \nu_{\sigma,\epsilon}}\,\langle \eta_i,\eta_j \rangle= \alpha^2\,\langle \sigma_i,\sigma_j\rangle + (1-\alpha^{2})\mathbb{1}_{i=j}\mbox{ with }\alpha\in [1-\epsilon,1].
\end{equation}
\end{prop}
\begin{proof} We have that $\nu$ is a product measure and $A_{\sigma,\epsilon}$ also has a product structure:
\begin{equation}
    A_{\sigma,\epsilon} = \times_{i=1}^n\{x\in \mathbb{S}^{k-1}\,:\,\langle x,\sigma_i\rangle\geq 1-\epsilon\}.
\end{equation}
The right hand side is a Cartesian product of spherical caps, and the measures of these caps (as subsets of $\mathbb{S}^{k-1}$) do not depend on the choice of $\sigma_1$. Therefore,
\begin{equation}
    \nu(A_{\sigma,\epsilon}) = (\nu_1\{x\in \mathbb{S}^{k-1}\,:\,\langle x,\sigma_1\rangle\geq 1-\epsilon\})^n
\end{equation}
where $\nu_1$ is the uniform measure on $\mathbb{S}^{k-1}$. The inequality for $\nu(A_{\sigma,\epsilon})$ follows from 
\begin{equation}
    \nu_1\{x\in \mathbb{S}^{k-1}\,:\,\langle x,\sigma_1\rangle\geq 1-\epsilon\}\geq \frac{(2\epsilon - \epsilon^2)^{\frac{(k-1)}{2}}}{6\sqrt{k}\,(1-\epsilon)},
\end{equation}
which the lower bound shown in the second part of \cite[Lemma 2.1]{brieden2001deterministic} (this inequality requires $1-\epsilon\geq \sqrt{2/k}$, which is guaranteed by $k\geq 8$ and $\epsilon\leq 1/2$; we also omitted a few terms from the lower bound for simplicity).  

The product structure of $A_{\sigma,\epsilon}$ implies that the conditional measure $\nu_{\sigma,\epsilon}$ is also a product measure. In particular, if $\eta\sim \nu_{\sigma,\epsilon}$, then
\[\eta_1,\dots,\eta_n\mbox{ are independent.}\]
Additionally,
\begin{equation}
    \eta_i = a_i\,\sigma_i + \sigma_i^\perp
\end{equation}
where $\sigma_i^\perp\in \mathbb{S}^{k-1}$ is a random vector perpendicular to $\sigma_i$ and $a_i  :=\langle \eta_i,\sigma_i\rangle\in [1-\epsilon,1)$ almost surely. Since the law of $\eta_i$ is invariant by rotations of $\mathbb{S}^{k-1}$ that leaves $\sigma_i$ fixed, then for $i \neq j$ we have 
\begin{equation}
    \E_{\eta\sim \nu_{\sigma,\epsilon}}[\sigma_i^\perp]=0\mbox{ and }\E_{\eta\sim \nu_{\sigma,\epsilon}}\,\langle \eta_i,\eta_j\rangle = \E_{\eta\sim \nu_{\sigma,\epsilon}}[a_i]\E_{\eta\sim \nu_{\sigma,\epsilon}}[a_j]\,\langle \sigma_i,\sigma_j\rangle,
\end{equation}
and (by rotation invariance) $\alpha  = \E_{\eta\sim \nu_{\sigma,\epsilon}}[a_i]$ is independent of $i$. For $i=j$ we have
\begin{equation}
\begin{split}
    \E_{\eta\sim \nu_{\sigma,\epsilon}}\,\langle \eta_i,\eta_i\rangle &= \E_{\eta\sim \nu_{\sigma,\epsilon}}[a_i^{2}]\,\langle \sigma_i,\sigma_i\rangle\\
    1 &= \E_{\eta\sim \nu_{\sigma,\epsilon}}[a_i^{2}]\cdot1
\end{split}
\end{equation}
such that $\E_{\eta\sim \nu_{\sigma,\epsilon}}[a_i^{2}]= 1 = (1-\alpha^{2})+\alpha^{2}$.
\end{proof}

\begin{proof}[Proof of Lemma \ref{lemma:2}]Choose $\sigma\in (\mathbb{S}^{k-1})^n$ such that 
\begin{equation}
    {\rm OPT}_k(M)=\sum_{i,j=1}^n M_{i,j}\langle \sigma_i,\sigma_j\rangle.
\end{equation}
By Proposition \ref{prop:random} and Jensen's inequality,
\begin{equation}
    \alpha^2\,{\rm OPT}_k(M) + (1-\alpha^{2})\tr{(M)} = \E_{\eta\sim \nu_{\sigma,\epsilon}}\left[\sum_{i,j=1}^n M_{i,j}\langle \eta_i,\eta_j\rangle\right] \leq \frac{1}{\beta}\log\left(\E_{\eta\sim \nu_{\sigma,\epsilon}}\left[\exp\left(\beta\sum_{i,j=1}^n M_{i,j}\langle \eta_i,\eta_j\rangle\right)\right]\right).
\end{equation}
 Now notice that, since the exponential is nonnegative and $\nu_{\sigma,\epsilon} = \nu(\cdot\mid A_{\sigma,\epsilon})$,
\begin{equation}
\begin{split}
    \E_{\eta\sim \nu_{\sigma,\epsilon}}\left[\exp\left(\beta\sum_{i,j=1}^n M_{i,j}\langle \eta_i,\eta_j\rangle\right)\right] & = \frac{\E_{\eta\sim \nu}\left[\exp\left(\beta\sum_{i,j=1}^n M_{i,j}\langle \eta_i,\eta_j\rangle\right){\bf 1}_{A_{\sigma,\epsilon}}(\eta)\right]}{\nu(A_{\sigma,\epsilon})}\\ &\leq \frac{\E_{\eta\sim \nu}\left[\exp\left(\beta\sum_{i,j=1}^n M_{i,j}\langle \eta_i,\eta_j\rangle\right)\right]}{\nu(A_{\sigma,\epsilon})}.
\end{split}
\end{equation}
Taking logs and plugging into the previous display, we arrive at
\begin{equation}
    \alpha^2{\rm OPT}_k(M) + (1-\alpha^{2})\tr{(M)}\leq \Phi(k,\beta,M) + \frac{\log(1/\nu(A_{\sigma,\epsilon}))}{\beta}\leq \Phi(k,\beta,M) + \frac{n\,(k-1)}{2\beta}\log((6\sqrt{k})^{\frac{2}{k-1}}/\epsilon).
\end{equation}
On the other hand $\alpha\in [1-\epsilon,1]$ implies $\alpha^2\in [1-2\epsilon+\epsilon^2,1]\subset [1-2\epsilon,1]$, so
\begin{equation}
    {\rm OPT}_k(M) = \alpha^2\,{\rm OPT}_k(M) + (1-\alpha^2)\,{\rm OPT}_k(M)\leq \alpha^2\,{\rm OPT}_k(M) + 2\epsilon\,\,C\|M\|_{\infty\to 1}.
\end{equation}
where we used Proposition \ref{prop:norm}. 

Therefore, 
\begin{equation}
    {\rm OPT}_k(M) \leq \Phi(k,\beta,M) + \frac{n\,(k-1)}{2\beta}\log(c/\epsilon) - (1-\alpha^{2})\tr{(M)} + 2\epsilon\,\,C\|M\|_{\infty\to 1}
\end{equation}
where $c =\sup_{k\geq 8}(6\sqrt{k})^{\frac{2}{k-1}}<+\infty$. 

Considering that $\alpha^2\in [1-2\epsilon+\epsilon^2,1]$ and $\epsilon^2<\epsilon$ for $\epsilon \in (0,\frac{1}{2})$, we obtain $0 \leq 1-\alpha^{2} \leq 2\epsilon$. We also have that exists a constant $C'>0$ such that $|\tr{(M)}|\leq C'\lVert M \rVert_{\infty \to 1}$, because
\begin{equation}
    |\tr(M)| = |\tr(IM)| \leq |\text{SDP}(M)| \leq C' \lVert M \rVert_{\infty\to 1}
\end{equation}
and we used Proposition \ref{prop:norm} again with $k=n$ and $\sigma$ attaining the maximum. We conclude that $-\tr{(M)} \leq C'\lVert M \rVert_{\infty \to 1}$ and then
\begin{equation}
    {\rm OPT}_k(M) \leq \Phi(k,\beta,M) + \frac{n\,(k-1)}{2\beta}\log(c/\epsilon) + (2\epsilon C +\epsilon C')\|M\|_{\infty\to 1}
\end{equation}

The desired upper bound on ${\rm OPT}_k(M)$ follows from adjusting the constants $C,C'$. The lower bound follows because $\beta>0$,
\begin{equation}
    \exp(\beta\,\Phi(k,\beta,M)) = \E_{\eta\sim \nu}\left[\exp\left(\beta\sum_{i,j=1}^n M_{i,j}\langle \eta_i,\eta_j\rangle\right)\right]\leq \exp\left({\beta\,\sup_{\eta\in (\mathbb{S}^{k-1})^n}\sum_{i,j=1}^n M_{i,j}\langle \eta_i,\eta_j\rangle}\right)
\end{equation}
and the RHS of this inequality is precisely $\exp(\beta\,{\rm OPT}_k(M))$.

To complete the proof, we use the already calculated bound for the probability of $\lVert M \rVert_{\infty \to 1}$ being large.
\end{proof}

Now, we have the following calculation to be used in the concentration parts of equation \ref{eq:54}.

\underline{Bounded variation of $\Phi$:}
Let $M'$ be the same matrix as $M$ except by the entry $ij$ that is substituted by $M_{ij}'$. Then, we have
\begin{equation}
    \begin{split}
        &|\Phi(\beta,k;M)-\Phi(\beta,k;M')| = \Bigg| \frac{1}{\beta} \log\Bigg(\int \exp{\Bigg(\beta \sum_{\substack{k,l=1\\k\neq i,l \neq j}}^{n}M_{kl}\langle \sigma_k,\sigma_l \rangle + M_{ij}\langle \sigma_{i},\sigma_{j}\rangle \Bigg)}d\nu(\sigma)\Bigg)\\
        &- \frac{1}{\beta} \log\Bigg(\int \exp{\Bigg(\beta \sum_{\substack{k,l=1\\k\neq i,l \neq j}}^{n}M_{kl}\langle \sigma_k,\sigma_l \rangle + M'_{ij}\langle \sigma_{i},\sigma_{j}\rangle \Bigg)}d\nu(\sigma)\Bigg) \Bigg| \\
        &=\Bigg| \frac{1}{\beta} \log\Bigg(\frac{\int \exp(\beta \sum M_{kl}\langle \sigma_{k},\sigma_{l}\rangle + M_{ij}\langle \sigma_i,\sigma_j \rangle)d\nu(\sigma)}{\int \exp(\beta \sum M_{kl}\langle \sigma_{k},\sigma_{l}\rangle + M'_{ij}\langle \sigma_i,\sigma_j \rangle)d\nu(\sigma)}\Bigg)\Bigg|\\
        &=\Bigg| \frac{1}{\beta} \log\Bigg(\frac{\int \exp(\beta \sum M_{kl}\langle \sigma_{k},\sigma_{l}\rangle + M'_{ij}\langle \sigma_i,\sigma_j \rangle)\exp(\beta(M_{ij}-M'_{ij})\langle \sigma_{i},\sigma_{j}\rangle) d\nu(\sigma)}{\int \exp(\beta \sum M_{kl}\langle \sigma_{k},\sigma_{l}\rangle + M'_{ij}\langle \sigma_i,\sigma_j \rangle)d\nu(\sigma)}\Bigg)\Bigg|\\
        &\leq \Bigg| \frac{1}{\beta} \log\Bigg(\frac{\int \exp(\beta \sum M_{kl}\langle \sigma_{k},\sigma_{l}\rangle + M'_{ij}\langle \sigma_i,\sigma_j \rangle)d\nu(\sigma) \max_{\sigma}\exp(|\beta(M_{ij}-M'_{ij})\langle \sigma_{i},\sigma_{j}\rangle|)}{\int \exp(\beta \sum M_{kl}\langle \sigma_{k},\sigma_{l}\rangle + M'_{ij}\langle \sigma_i,\sigma_j \rangle)d\nu(\sigma)}\Bigg)\Bigg|\\
        &\leq \Bigg| \frac{1}{\beta}\log\Bigg(\exp(\beta|M_{ij}-M'_{ij}|)\Bigg)\Bigg|\\
        &= |M_{ij}-M'_{ij}| \\
    \end{split}
\end{equation}
where in the first inequality below we used $\log$ is increasing, Hölder inequality, $\exp$ increasing and in the second inequality that $\langle \sigma_{i},\sigma_{j}\rangle \leq 1$ and $\exp,\log$ are increasing.

The concentration part is given by the following lemma.

\begin{lema}
    With probability at least $1-n^{-\delta}$ where $\delta>0,\,\beta>0$ and $k>0$, we have
    \begin{equation}
        |\Phi(\beta,k;M)-\E[\Phi(\beta,k,M)]| \lesssim \sqrt{\sum_{i\leq j}\nu_{ij}n^{\delta}}
    \end{equation}
\end{lema}
\begin{proof}
By the Efron-Stein inequality and the above bounded variation calculation we have
\begin{equation}
\begin{split}
    \V(\Phi(\beta,k;M)) &\leq \sum_{i,j=1}^{n}\E[\Phi(\beta,k;M)-\Phi(\beta,k;M^{(ij)})]\\
    &\leq 2\sum_{i\leq j}\E[|M_{ij}-M'_{ij}|^{2}]\\
    &= 2\sum_{i\leq j}\E[|M_{ij}-\E[M_{ij}]+\E[M_{ij}]-M'_{ij}|^{2}]\\
    &\leq 4\sum_{i \leq j}\V(M_{ij}) \leq C\sum_{i\leq j}\nu_{ij}
\end{split}
\end{equation}
where the superscript $(ij)$ means that we are considering the same matrix but with a independent copy of the $ij$-th entry and the last inequality is due to the fact that for a distribution $Y\sim sub\Gamma(\nu,c)$ we have $\V(Y) \lesssim \nu$.

Now, for any $\varepsilon>0$ by the Chebyshev inequality we obtain
\begin{equation}
    \P(|\Phi(\beta,k;M)-\E[\Phi(\beta,k;M)]|\geq \varepsilon) \leq \frac{\V(\Phi(\beta,k;M))}{\varepsilon^{2}} \leq \frac{2\sum_{i\leq j}\nu_{ij}}{\varepsilon^{2}}
\end{equation}
choosing $\varepsilon = \sqrt{C\sum_{i\leq j}\nu_{ij}n^{\delta}}$ the result follows.
\end{proof}

Finally, for the universality result of the expectation, we use lemma $E.3$ from \cite{montanari}.
\begin{lema}\label{lemma:4} Let $F:\mathbb{R}^{N} \rightarrow \mathbb{R}$ be a three times continuously differentiable function and $Y=(Y_{1},Y_{2},...,Y_{N})$, $Z=(Z_{1},Z_{2},...,Z_{N})$ be independent random vectors such that $\E[Y_{i}]=\E[Z_{i}]$, $\E[Y_{i}^2]=\E[Z_{i}^2]$ for $1 \leq i\leq N$, then
\begin{equation}
    |\E[F(Y)]-\E[F(Z)]| \leq \frac{1}{6}\left[\sum_{i=1}^{N} \{\E[|Y_{i}|^{3}] + \E[|Z_{i}|^{3}]\}\right]\max_{i \in [N]}{\lVert \partial_{i}^{3}F\rVert}_{\infty}.
\end{equation}
\end{lema}

\begin{proof} In our case, we use the matrix $Y=X$ as a vector of its entries and the same for $Z=D$ and they are already independent. Moreover, we have $N=n^2$ and $F(M)=\Phi(\beta,k;M)$. 

Let us now calculate $\lVert \partial_{ij}^{3}\Phi \rVert_{\infty}$. Let $\partial_{ij} = \partial / \partial M_{ij}$. Defining the Gibbs measure
\begin{equation}
    \mu_{M}(\sigma) \coloneqq \frac{\exp{(\beta H_{M}(\sigma))}}{\int \exp{(\beta H_{M}(\tau))}d\nu(\tau)} d\nu(\sigma)
\end{equation}
where $H_{M}(\sigma)= \langle \sigma, M \sigma \rangle = \sum_{i,j=1}^{n}M_{ij}\langle \sigma_{i},\sigma_{j} \rangle$ and $\sigma$ is already defined. We calculate the derivatives.

The first derivative gives
\begin{equation}
\begin{split}
    \partial_{ij}\Phi(\beta,k;M) &= \frac{1}{\beta} \frac{\partial}{\partial M_{ij}} \log{\left(\int \exp{\left(\beta \sum_{i,j=1}^{n}M_{ij}\langle \sigma_{i},\sigma_{j}\rangle \right)} d\nu(\sigma)\right)}\\
    &= \frac{1}{\beta}\int \frac{\beta \langle \sigma_{i}, \sigma_{j} \rangle \exp{(\beta H_{M}(\sigma))}}{\int \exp{(\beta H_{M}(\tau))}d\nu(\tau)} d\nu(\sigma) =: \mu_{M}(\langle \sigma_{i}, \sigma_{j} \rangle)
\end{split}
\end{equation}
where we use the notation that $\mu_{M}(f(\sigma))$ represents the expectation of $f$ in relation to the measure $\mu_{M}$.
The second derivative gives

\begin{align}
    \partial_{ij}^{2}\Phi(\beta,k;M) &= \frac{\partial}{\partial M_{ij}} \int \frac{\langle \sigma_{i}, \sigma_{j} \rangle \exp{\left(\beta \sum_{i,j=1}^{n} M_{ij}\langle \sigma_{i},\sigma_{j} \rangle\right)}}{\int \exp{\left(\beta \sum_{i,j=1}^{n}M_{ij}\langle \tau_{i},\tau_{j} \rangle\right)}d\nu(\tau)}d\nu(\sigma)\\
    &=\beta \int \frac{\langle \sigma_{i},\sigma_{j} \rangle^{2}\exp{\left(\beta \sum_{i,j=1}^{n}M_{ij} \langle \sigma_{i},\sigma_{j} \rangle\right)}}{\int \exp{\left(\beta \sum_{i,j=1}^{n}M_{ij} \langle \tau_{i},\tau_{j} \rangle\right)}d\nu(\tau)}d\nu(\sigma)\\
    &- \beta\frac{1}{\left(\int \exp{\left(\beta \sum_{i,j=1}^{n}M_{ij} \langle \tau_{i},\tau_{j} \rangle\right)}d\nu(\tau)\right)^2}\times\\
    &\times \int \langle \tau_{i},\tau_{j} \rangle \exp{\left(\beta \sum_{i,j=1}^{n}M_{ij} \langle \tau_{i},\tau_{j} \rangle\right)}d\nu(\tau)\\
    &\times\int \langle \sigma_{i},\sigma_{j} \rangle \exp{\left(\beta \sum_{i,j=1}^{n}M_{ij} \langle \sigma_{i},\sigma_{j} \rangle\right)}d\nu(\sigma)\\
    &= \beta \int \frac{\langle \sigma_{i},\sigma_{j} \rangle^{2}\exp{\left(\beta \sum_{i,j=1}^{n}M_{ij} \langle \sigma_{i},\sigma_{j} \rangle\right)}}{\int \exp{\left(\beta \sum_{i,j=1}^{n}M_{ij} \langle \tau_{i},\tau_{j} \rangle\right)}d\nu(\tau)}d\nu(\sigma)\\
    &-\beta \left(\int \frac{\langle \sigma_{i},\sigma_{j} \rangle\exp{\left(\beta \sum_{i,j=1}^{n}M_{ij} \langle \sigma_{i},\sigma_{j} \rangle\right)}}{\int \exp{\left(\beta \sum_{i,j=1}^{n}M_{ij} \langle \tau_{i},\tau_{j} \rangle\right)}d\nu(\tau)}d\nu(\sigma)\right)^{2}\\
    &= \beta(\mu_{M}(\langle \sigma_{i},\sigma_{j}\rangle^2) - \mu_{M}(\langle \sigma_{i}, \sigma_{j} \rangle)^2)
\end{align}

Finally, the third one gives
\begin{align}
    &\partial_{ij}^{3}\Phi(\beta,k;M) = \frac{\partial}{\partial M_{ij}} \beta \int \frac{\langle \sigma_{i},\sigma_{j} \rangle^{2}\exp{\left(\beta \sum_{i,j=1}^{n}M_{ij} \langle \sigma_{i},\sigma_{j} \rangle\right)}}{\int \exp{\left(\beta \sum_{i,j=1}^{n}M_{ij} \langle \tau_{i},\tau_{j} \rangle\right)}d\nu(\tau)}d\nu(\sigma)\\ 
    &-\frac{\partial}{\partial M_{ij}}\beta \left(\int \frac{\langle \sigma_{i},\sigma_{j} \rangle\exp{\left(\beta \sum_{i,j=1}^{n}M_{ij} \langle \sigma_{i},\sigma_{j} \rangle\right)}}{\int \exp{\left(\beta \sum_{i,j=1}^{n}M_{ij} \langle \tau_{i},\tau_{j} \rangle\right)}d\nu(\tau)}d\nu(\sigma)\right)^{2}\\
    &= \beta^{2} \int \frac{\langle \sigma_{i},\sigma_{j}\rangle^3 \exp{\left(\beta \sum_{i,j=1}^{n}M_{ij}\langle \sigma_{i}, \sigma_{j}\rangle\right)}}{\int \exp{\left(\beta \sum_{i,j=1}^{n}M_{ij}\langle \tau_{i}, \tau_{j}\rangle\right)}d\nu(\tau)}d\nu(\sigma)\\
    &-\beta^{2} \int\frac{\langle \tau_{i}, \tau_{j}\rangle \exp{\left(\beta \sum_{i,j=1}^{n}M_{ij}\langle \tau_{i}, \tau_{j}\rangle\right)} d\nu(\tau)}{\int \exp{\left(\beta \sum_{i,j=1}^{n}M_{ij}\langle \tau_{i}, \tau_{j}\rangle\right)}d\nu(\tau)}\times \\
    &\times \int\frac{\langle \sigma_{i}, \sigma_{j}\rangle^{2} \exp{\left(\beta \sum_{i,j=1}^{n}M_{ij}\langle \sigma_{i}, \sigma_{j}\rangle\right)} d\nu(\sigma)}{\int \exp{\left(\beta \sum_{i,j=1}^{n}M_{ij}\langle \tau_{i}, \tau_{j}\rangle\right)}d\nu(\tau)}\\
    &-2\beta\int \frac{\langle \sigma_{i},\sigma_{j} \rangle\exp{\left(\beta \sum_{i,j=1}^{n}M_{ij} \langle \sigma_{i},\sigma_{j} \rangle\right)}}{\int \exp{\left(\beta \sum_{i,j=1}^{n}M_{ij} \langle \tau_{i},\tau_{j} \rangle\right)}d\nu(\tau)}d\nu(\sigma)\times\\
    &\times[\beta \int \frac{\langle \sigma_{i},\sigma_{j} \rangle^{2}\exp{\left(\beta \sum_{i,j=1}^{n}M_{ij} \langle \sigma_{i},\sigma_{j} \rangle\right)}}{\int \exp{\left(\beta \sum_{i,j=1}^{n}M_{ij} \langle \tau_{i},\tau_{j} \rangle\right)}d\nu(\tau)}d\nu(\sigma)\\
    &-\beta \left(\int \frac{\langle \sigma_{i},\sigma_{j} \rangle\exp{\left(\beta \sum_{i,j=1}^{n}M_{ij} \langle \sigma_{i},\sigma_{j} \rangle\right)}}{\int \exp{\left(\beta \sum_{i,j=1}^{n}M_{ij} \langle \tau_{i},\tau_{j} \rangle\right)}d\nu(\tau)}d\nu(\sigma)\right)^{2}]\\
    &= \beta^{2}(\mu_{M}(\langle \sigma_{i},\sigma_{j} \rangle^3) -3\mu_{M}(\langle \sigma_{i},\sigma_{j} \rangle^2)\mu_{M}(\langle\sigma_{i},\sigma_{j}\rangle) + 2\mu_{M}(\langle \sigma_{i},\sigma_{j}\rangle^3))
\end{align}

using that $|\langle \sigma_{i},\sigma_{j}\rangle|\leq 1$ and that $\mu_{M}(\sigma)$ is a probability measure we obtain that all expectations in the previous line are in $[-1,1]\subset \mathbb{R}$, so that
\begin{equation}
    \lVert \partial_{ij}^{3} \Phi \rVert_{\infty} \leq 6\beta^{2}.
\end{equation}

Then, applying the theorem we obtain
\begin{equation}
    |\E[\Phi(\beta,k;X)]-\E[\Phi(\beta,k;D)]| \leq \beta^{2}\left[\sum_{1 \leq i \leq j \leq n }\E(|X_{ij}|^{3}) + \sum_{1 \leq i \leq j \leq n}\E(|D_{ij}|^{3})\right]
\end{equation}

From \cite{regina} we have that
\begin{equation}
    \E(|D_{ij}|^{3}) = (\mu_{ij}^{3} + 2\sigma_{ij}^{2})\left(\frac{2}{\sqrt{2\pi}}\sigma_{ij}e^{-\frac{1}{2}\theta_{ij}^2}\right)-\mu_{ij}(1-2F(\theta_{ij})) - \mu_{ij}\sigma_{ij}^{2}(1-2F(\theta_{ij}))
\end{equation}
where
\begin{equation}
\begin{split}
    \theta_{ij}&=\frac{\mu_{ij}}{\sigma_{ij}}\\
    F(a)&=\frac{1}{\sqrt{2\pi}} \int_{-\infty}^{\infty} e^{-\frac{1}{2}y^{2}}dy
\end{split}
\end{equation}
then, considering we are dealing with the centered case we get
\begin{equation}
    \E(|D_{ij}|^{3}) \lesssim \sigma_{ij}^{3}. 
\end{equation} 
Remembering the Assumption \ref{assum:1} we are using, we conclude 
\begin{equation}
    |\E[\Phi(\beta,k;X)]-\E[\Phi(\beta,k;D)]| \lesssim \beta^{2} \sum_{i\leq j}^{N} (c_{ij}\nu_{ij} + \nu_{ij}^{\frac{3}{2}}),
\end{equation}
and this completes the proof of the theorem.
\end{proof}

\section{Lower bound of $\text{SDP}(M_{r}-M_{s})$}\label{sec:appB}

In this appendix, we find a lower bound for $\text{SDP}(M_{r}- M_{s})$ where $M_{i}$ represents the matrix of edge's means in the case we have $i$ communities. We will prove the following theorem.

\begin{teo} Let $r,s \in \mathbb{N}$ with $r>s\geq 2$ and define $m \coloneqq \frac{n}{rs}$, we also assume $m\geq 2$. Let $M_{r}$ and $M_{s}$ be the matrices of edge's means in the hypothesis where we have $r$ and $s$ communities, respectively. We suppose also that in each hypothesis the communities have the same size and the means of edges inside and outside communities are homogeneous. These means are given by $M_{in}$ and $M_{out}$, respectively, and $M_{in} > M_{out}$. The following is true for all matrices $M_{r}$ and $M_{s}$ 
\begin{equation}
    \text{SDP}(M_{r}-M_{s}) \geq \frac{2n^2(M_{in}-M_{out})}{r^2s^2}.
\end{equation}
\end{teo}

\begin{proof}
First, observe that we can fix $M_{r}$ because this only means that we are fixing a enumeration for the vertices and communities, which are irrelevant for community detection. We then fix $M_{r}$ such that the vertices are enumerated increasingly within the enumeration of the communities, that is: the first community $C_{1}$ has the first $|C_{1}|$ vertices, the second community $C_{2}$ has the following $|C_{2}|$ vertices and so on. With that, $M_{r}$ has the following shape
\[
        \setlength{\arraycolsep}{0pt}
  \setlength{\delimitershortfall}{0pt}
  M_{r} = \begin{pmatrix}
    \,\fbox{$M_{in}$} &  &  & \,  \\
    \, & \fbox{$M_{in}$} &  & M_{out}\, \\
    \,M_{out} &  & \ddots & \, \\
    \, &  &  & \fbox{$M_{in}$}\, \\
  \end{pmatrix}
\]
where the $r$ communities have the same size by assumption so that each block has size $\frac{n}{r}\times \frac{n}{r} =sm \times sm$, remembering $m=\frac{n}{rs}$.

For simplicity, we try to focus our analysis in what happens inside the $r$ smaller communities. We do this by choosing $Z\in \text{PSD}_{1}(n)$ with non-zero entries only within the diagonal blocks of $M_{r}$. When this is not possible we create a ``perturbed" version $\tilde{Z} \in \text{PSD}_{1}(n)$ to do the analysis. With that in mind, we divide the proof into the following cases.
\begin{itemize}
    \item[$\diamond$] \textbf{Case 1:} When $s\nmid r$, the configuration of communities $\mathcal{C}$ where each of the $r$ smaller communities are entirely contained in a big community is impossible. In this case, we do a special construction that allows us to find a lower bound in the right order with the $Z$ mentioned;
    \item[$\diamond$] \textbf{Case 2:} When $s \mid r$, the above configuration is possible. Then, we will need to subdivide again into cases to deal with this. 
        \begin{itemize}
        \item[$\star$] \textbf{Subcase 2.1:} If the configuration is the previous mentioned or ``near" it we need to choose the mentioned $\tilde{Z}$ to ensure the bound is of the right order.
        \item[$\star$] \textbf{Subcase 2.2:} If the configuration is ``far" enough from the previous mentioned, we can choose the same $Z$ as in the case $s\nmid r$ and have the right order bound.
    \end{itemize}
\end{itemize}
With that, the bound will be the lowest obtained from all cases.

\begin{figure}[H]
    \centering
    \begin{tikzpicture}

    \tikzstyle{group} = [ellipse, draw=cyan, thick, minimum width=3cm, minimum height=2cm, fill=cyan!10, inner sep=0pt]
    \tikzstyle{dot} = [circle, fill=black, minimum size=2pt, inner sep=0pt] 
    \tikzstyle{subgroup} = [ellipse, draw=cyan, thick, minimum width=1cm, minimum height=0.8cm, inner sep=0pt]

    \node[group, label=above:1] (circle1) at (0, 2) {}; 
    \node[group, label=above:2] (circle2) at (3, 2) {}; 
    \node[group, label=below:3] (circle3) at (1.5, 0) {}; 

    \node[subgroup] (subgroup1a) at (-0.6, 2.0) {}; 
    \node[subgroup] (subgroup1b) at (0.6, 2.0) {}; 

    \node[subgroup] (subgroup2a) at (2.4, 2.0) {}; 
    \node[subgroup] (subgroup2b) at (3.6, 2.0) {}; 

    \node[subgroup] (subgroup3a) at (0.9, 0.0) {}; 
    \node[subgroup] (subgroup3b) at (2.1, 0.0) {}; 

    \foreach \x/\y in {-0.7/2.2, -0.5/2.2, -0.7/2.0, -0.5/2.0, -0.7/1.8, -0.5/1.8} {
        \node[dot] at (\x,\y) {}; 
    }

    \foreach \x/\y in {0.5/2.2, 0.7/2.2, 0.5/2.0, 0.7/2.0, 0.5/1.8, 0.7/1.8} {
        \node[dot] at (\x,\y) {}; 
    }

    \foreach \x/\y in {2.3/2.2, 2.5/2.2, 2.3/2.0, 2.5/2.0, 2.3/1.8, 2.5/1.8} {
        \node[dot] at (\x,\y) {}; 
    }

    \foreach \x/\y in {3.5/2.2, 3.7/2.2, 3.5/2.0, 3.7/2.0, 3.5/1.8, 3.7/1.8} {
        \node[dot] at (\x,\y) {}; 
    }

    \foreach \x/\y in {0.8/0.2, 1.0/0.2, 0.8/0.0, 1.0/0.0, 0.8/-0.2, 1.0/-0.2} {
        \node[dot] at (\x,\y) {}; 
    }

    \foreach \x/\y in {2.0/0.2, 2.2/0.2, 2.0/0.0, 2.2/0.0, 2.0/-0.2, 2.2/-0.2} {
        \node[dot] at (\x,\y) {}; 
    }

    \end{tikzpicture}
    \caption{Configuration $\mathcal{C}$ in the case $r=6$, $s=3$ and $m=2$}
    \label{fig:configurationC}
\end{figure}

\begin{obs}
    What we call ``right order" is a bound of order $m^{2}$, which makes sense in the calculations of the Type I and II errors in the text.
\end{obs}

We start with the case where $s$ does not divide $r$.

\underline{\textbf{Case 1. $s\nmid r$}:}
The idea is to fix a positive semidefinite matrix $Z$ with diagonal elements equal to $1$, that is, a $Z\in \text{PSD}_{1}(n)$, and use the definition of SDP function to find a configuration of $s$ communities such that its matrix of edge's means $M^{*}_{s}$ minimizes the inner product of $M_{r}-M^{*}_{s}$ with that specific $Z$. With that, we would have
\begin{equation}\label{eq:183}
    \text{SDP}(M_{r}-M_{s}) \geq \langle M_{r}-M_{s}, Z \rangle \geq \langle M_{r}-M^{*}_{s}, Z \rangle.
\end{equation}
for all possible $M_{s}$, \textit{i.e.}, matrix of edge's means in the case we have $s$ communities.  

We fix $Z$ to be the following $n\times n$ matrix 
\[
  \setlength{\arraycolsep}{0pt}
  \setlength{\delimitershortfall}{0pt}
  Z \coloneqq \begin{pmatrix}
    \,\resizebox{!}{1em}{\fbox{$1$}} &  &  & \,  \\
    \, & \resizebox{!}{1em}{\fbox{$1$}} &  & 0\, \\
    \,0&  & \ddots & \, \\
    \, &  &  & \resizebox{!}{1em}{\fbox{$1$}}\, \\
  \end{pmatrix}
\]
where the blocks, as in $M_{r}$, are of size $sm\times sm$. 

Actually, we do not find a configuration of $s$ communities with matrix of edge's means $M^{*}_{s}$ defined in equation \ref{eq:183} but find a lower bound on the inner product valid for all possible matrices $M_{s}$.

We observe that, since $s$ does not divide $r$, we can put  at maximum $\left\lfloor \frac{rm}{sm} \right\rfloor =\left\lfloor \frac{r}{s} \right\rfloor$ communities of the $r$ smaller ones (of size $sm$) entirely within each one of the $s$ bigger communities (of size $rm$). Then, in each one of the bigger communities it remains
 \begin{equation}
     rm - sm\floor[\bigg]{\frac{r}{s}}
 \end{equation}
vertices which are not inside any smaller communities. To construct the bound, we thus suppose that there is a ``ghost" community $\mathcal{G}$ outside the largest ones where the remainning elements will be part of. We denote this ``not allowed" configuration of bigger communities as $C^{*}$, where what we mean by ``allowed" configuration is one where the vertices are all inside one and only one of the $s$ bigger communities (which gives a corresponding $M_{s}$ matrix).

For example, let us suppose that $r=8$, $s=3$, $m=2$ such that the $s$ bigger communities have size $rm=16$ and the $r$ smaller ones have size $sm=6$ vertices. The image below represents the configuration in this case where the $4$ remaining vertices in each bigger community are part of the denoted $\mathcal{G}$ community.

\begin{figure}[H]
    \centering
    \begin{tikzpicture}
        \tikzstyle{group} = [ellipse, draw=cyan, thick, minimum width=3cm, minimum height=2cm, fill=cyan!10, inner sep=0pt]
        \tikzstyle{dot} = [circle, fill=black, minimum size=2pt, inner sep=0pt] 
        \tikzstyle{center} = [circle, dashed, draw=cyan, thick, minimum size=2cm]
        \tikzstyle{subgroup} = [ellipse, draw=cyan, thick, minimum width=1cm, minimum height=0.8cm, inner sep=0pt]

        \node[center, label=above:$\mathcal{G}$] (G) at (0, 0.5) {}; 

        \draw[thick, cyan] (G) -- (-2.8, 2); 
        \draw[thick, cyan] (G) -- (2.8, 2);  
        \draw[thick, cyan] (G) -- (0, -1.5); 

        \node[group, label=above:1] (group1) at (-2, 2) {};  
        \node[subgroup] (subgroup1a) at (-2.5, 2.3) {}; 
        \node[subgroup] (subgroup1b) at (-2.5, 1.7) {}; 
        
        \foreach \x/\y in {-2.6/2.5, -2.4/2.5, -2.6/2.3, -2.4/2.3, -2.7/2.4, -2.3/2.4} {
            \node[dot] at (\x,\y) {};
        }
        \foreach \x/\y in {-2.6/1.7, -2.4/1.7, -2.6/1.5, -2.4/1.5, -2.5/1.6, -2.3/1.6} {
            \node[dot] at (\x,\y) {};
        }
        
        \foreach \x/\y in {-1.0/2.5, -1.0/1.5, -1.7/2.0, -1.3/2.0} {
            \node[dot] at (\x,\y) {};
        }

        \node[group, label=above:2] (group2) at (2, 2) {};  
        \node[subgroup] (subgroup2a) at (2.5, 2.3) {}; 
        \node[subgroup] (subgroup2b) at (2.5, 1.7) {}; 
        
        \foreach \x/\y in {2.4/2.5, 2.6/2.5, 2.4/2.3, 2.6/2.3, 2.7/2.4, 2.3/2.4} {
            \node[dot] at (\x,\y) {};
        }
        \foreach \x/\y in {2.6/1.7, 2.4/1.7, 2.6/1.5, 2.4/1.5, 2.5/1.6, 2.3/1.6} {
            \node[dot] at (\x,\y) {};
        }
        
        \foreach \x/\y in {1.0/2.5, 1.0/1.5, 1.7/2.0, 1.3/2.0} {
            \node[dot] at (\x,\y) {};
        }

        \node[group, label=below:3] (group3) at (0, -2) {};  
        \node[subgroup] (subgroup3a) at (-0.5, -2.3) {}; 
        \node[subgroup] (subgroup3b) at (0.5, -2.3) {}; 
        
        \foreach \x/\y in {-0.6/-2, -0.4/-2, -0.6/-2.4, -0.4/-2.4, -0.5/-2.2, -0.7/-2.2} {
            \node[dot] at (\x,\y) {};
        }
        \foreach \x/\y in {0.6/-2, 0.4/-2, 0.6/-2.4, 0.4/-2.4, 0.5/-2.2, 0.7/-2.2} {
            \node[dot] at (\x,\y) {};
        }
        
        \foreach \x/\y in {0.7/-1.5, 0.2/-1.5, -0.3/-1.5, -0.8/-1.5} {
            \node[dot] at (\x,\y) {};
        }
    \end{tikzpicture}
    \caption{Example of a ghost community $\mathcal{G}$}
    \label{fig:ghost-community}
\end{figure}
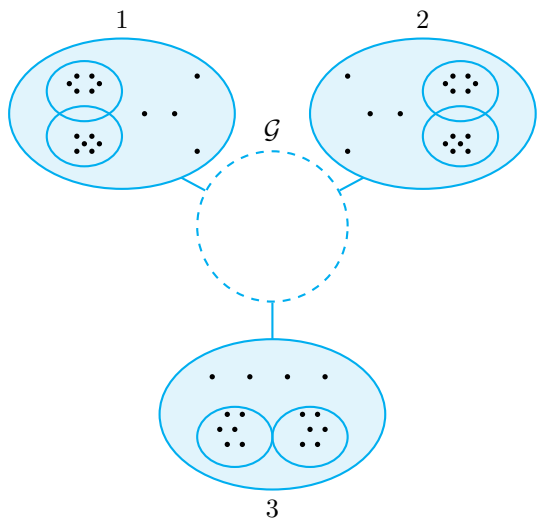

The goal becomes to show that this configuration of vertices, that gives us a configuration of $M_{in}$'s and $M_{out}$'s on the edges gives us a positive and correct order lower bound for equation \ref{eq:183}.

As the positive semidefinite matrix $Z$ has non-zero entries only on the blocks diagonal and these are the same ones where we have $M_{in}$ in $M_{r}$, for the inner product we need to focus only on this part of the matrices. We can have a $M_{in}$ or $M_{out}$ entry of $M_{s}$ in these diagonal blocks. In the first case, the corresponding entry of $M_{r}-M_{s}$ would be $M_{in}-M_{in}=0$ and in the second case it would be $M_{in}-M_{out}$. Then, to ensure we have a non-zero lower bound, it suffices to ensure that we have a positive number of entries of the second kind while still minimizing its number. For this purpose, first we show the following. 

\begin{claim}\label{claim:1}
\begin{equation}\label{eq:133}
    r(sm)^2 \geq s\floor[\bigg]{\frac{r}{s}}(sm)^2 + s\left(rm-sm\floor[\bigg]{\frac{r}{s}}\right)^2
\end{equation}
\end{claim}
where the term on the LHS is the number of $M_{in}$ entries of $M_{r}$, the first term on the RHS is the number of $M_{in}$ entries of the $C^{*}$ configuration of bigger matrices related to vertices inside the smaller communities, and the second term is the number of $M_{in}$ entries of $C^{*}$ related to the vertices that are left out of these smaller communities but are united in the ``ghost" community $\mathcal{G}$. In other words, this inequality says that even if we paired the $M_{in}$ entries of $M_{r}$ to the $M_{in}$ entries of the configuration $C^{*}$ it would still remain sufficient $M_{in}$'s from $M_{r}$ to be paired with the $M_{out}$'s from the same configuration $\mathcal{C^{*}}$. The bound then will become
\begin{equation}
B \coloneqq(LHS-RHS)(M_{in}-M_{out}).
\end{equation}

\begin{proofclaim}
Calculating
\begin{equation}
    \begin{split}
    &r(sm)^2 \geq s\floor[\bigg]{\frac{r}{s}}(sm)^2 + s\left(rm-sm\floor[\bigg]{\frac{r}{s}}\right)^2 \\
    \Leftrightarrow \quad &r(sm)^2 \geq s\floor[\bigg]{\frac{r}{s}}(sm)^2 + sr^2m^2 -2rs^2m^2\floor[\bigg]{\frac{r}{s}} + s^3m^2\floor[\bigg]{\frac{r}{s}}^2\\
    \Leftrightarrow \quad &rs \geq s^2\floor[\bigg]{\frac{r}{s}}\left(1+ \floor[\bigg]{\frac{r}{s}}\right) + r^2 -2rs\floor[\bigg]{\frac{r}{s}}
    \end{split}
\end{equation}
now, we remember that by the Euclidean division algorithm, we can write
\begin{equation}\label{eq:remainder}
    r = s \floor[\bigg]{\frac{r}{s}} + t \Rightarrow \floor[\bigg]{\frac{r}{s}} = \frac{r-t}{s}
\end{equation}
with the remainder $t\leq s-1$. Substituting $\left\lfloor \frac{r}{s} \right\rfloor$ above we obtain
\begin{equation}
    \begin{split}
        &rs \geq s^2\left(\frac{r-t}{s}\right) + s^2\left(\frac{r-t}{s}\right)^2 + r^2 -2rs\left(\frac{r-t}{s}\right)\\
        \Leftrightarrow \quad &rs - sr + st - r^2 +2rt - t^2 - r^2 +2r^2 -2rt \geq 0\\
        \Leftrightarrow \quad &st-t^2 \geq 0 \\
        \Leftrightarrow \quad &(s-t)t \geq 0
    \end{split}
\end{equation}
which is valid for $t\leq s-1$. \claimqed
\end{proofclaim}

Now, we show that $B$ is in fact a bound, that is, all allowed configurations of vertices in $M_{s}$ gives an inner product $\langle M_{r} - M_{s}, Z \rangle$ greater than $B$.

Let us call the bigger communities $s_{i}$ for $i \in \{1,2,...,s\}$ and the smaller ones $r_{i}$ for $i \in \{1,2,...,r\}$. Consider one of the $s$ biggest communities, for example, $s_{1}$. We want to have its $rm$ vertices configured such that it is maximized the number of $M_{in}$ entries related to these vertices in the diagonal blocks we are considering. In other words, we want to minimize $M_{in}-M_{out}$ (because it is maximized the entries $M_{in}-M_{in}=0$) on these blocks as mentioned on the paragraph above Claim \ref{claim:1}. To achieve such minimization, we first need the following step.
\begin{claim}
    If we have $a_{i}$ vertices of $s_{1}$ in $r_{i}$ for $i \in \{1,2,...,r\}$, then we have $a_{i}^{2}$ edges with mean $M_{in}$ coming from $s_{1}$ in the diagonal block corresponding to $r_{i}$.
\end{claim}
\begin{proofclaim}

Let us consider $r_{1}$ for concreteness and enumerate the vertices inside $s_{1}$ which are also inside $r_{1}$ as $v_{1},v_{2},...,v_{a_{1}}$. Consider $v_{1}$, by definition, we have entries $M_{in}$ of $s_{1}$ coming from this vertice in $a_{1} + (a_{1} -1)$ positions of the the block: $a_{1}$ counting the connection of $v_{1}$ with itself and the other vertices and, by symmetry, $a_{1}-1$ counting the connections of the other vertices to $v_{1}$ minus the double count of $v_{1}$. For $v_{2}$ we have the same counting but excluding the connection with $v_{1}$ already counted, that is, we count $a_{1} -1 + (a_{1} -2)$ entries $M_{in}$. Repeating the argument, we conclude that we have  $a_{1}^{2}$ elements $M_{in}$ of $s_{1}$ in the diagonal block of $r_{1}$. In fact, considering the contribution of all $a_{1}$ vertices we obtain
\begin{equation}
\begin{split}
    a_{1} + (a_{1}-1) + a_{1}-1 + (a_{1}-2) + ...+ (1) + 1 &=  a_{1} + a_{1}-1 + ...+ 1 + (a_{1}-1+ a_{1}-2 + ...+ 1)\\
    &= \frac{a_{1}(a_{1}+1)}{2} + \frac{(a_{1}-1)a_{1}}{2} = a_{1}^2.
\end{split}
\end{equation}
The same reasoning is valid for the other smaller communities $r_{2},r_{3},...,r_{r}$ and the claim follows. \claimqed
\end{proofclaim}

With that, the problem of maximizing the $M_{in}-M_{in}$ (minimizes $M_{in}-M_{out}$) entries related to $s_1$ becomes: to intersect the $rm$ vertices from $s_{1}$ with the vertices of the $r$ smaller communities such that the sum 
\begin{equation} \label{eq:191}
    S_{1}(C) := A_{1}^2 + A_{2}^2 + ... + A_{r}^2
\end{equation}
is maximized, where $A_{i}$ represents the number of vertices from $s_1$ in the $i$-th smaller community and $C$ represents a configuration of vertices in $s$ communities (allowed or not). In other terms, by the previous claim, this sum represents the number of edges with mean $M_{in}$ of $s_{1}$ also related to a connection inside the $r$ smaller communities.

We argue that the configuration $C^{*}$ with sum 
\begin{equation}
    S_{1}(C^{*}) = (sm)^2 + ... + (sm)^2 + \left(rm-sm\floor[\bigg]{\frac{r}{s}}\right)^2 + 0+ ...+0=\floor[\bigg]{\frac{r}{s}}(sm)^2 + \left(rm-sm\floor[\bigg]{\frac{r}{s}}\right)^2
\end{equation}
is the optimal one (observe that this is the $RHS$ of equation \ref{eq:133} divided by $s$ because it is considering only the first bigger community $s_{1}$). In other words, the best configuration is the one where we fill the \textbf{slots} $A_{i}$ in the maximum capacity $(sm)^2$ until it is left a number of vertices less than $sm$ which are also put in the same slot together, that is, $\left(rm-sm\floor[\bigg]{\frac{r}{s}}\right)^2$ edges inside the ``ghost" community $\mathcal{G}$. The $0+...+0$ terms indicate how many smaller communities do not intersect with $s_{1}$. For example, on the configuration of Figure \ref{fig:ghost-community} where $r=8$ we would have $8-3=5$ terms zero because we have $2$ smaller communities entirely contained on $s_1$ plus a intersection with ``ghost" community, that is, $2+1=3$ non-zero terms. 

\begin{obs}
    Observe that even though we have written the fuller smaller communities first in the sum, they do not necessarily correspond to the first enumerate smaller communities.
\end{obs}

In fact, beginning in \textbf{any} configuration $C$ (allowed or not) we can end up in the configuration $C^{*}$ putting more vertices together one at a time. The inverse operations to this, have two possible forms. To conclude that $C^{*}$ maximized the sum we want thus to show that doing these inverse operations can only makes our sum smaller. The operations are the following.

\textbf{Add a vertex from a non-empty slot to a empty one:} In this case, by the elementary inequality
\begin{equation}\label{eq:193}
    (x_{1}+x_{2}+...+x_{N})^{2} \geq x_{1}^{2} + x_{2}^{2} + ... + x_{N}^{2}
\end{equation}
for $x_{i}\in \mathbb{R}^{+}$ with $i \in [N]$ for any $N\in \mathbb{N}$, the sum decreases. This is because, with the $LHS$ representing the non-empty slot, if we put $x_{1}=1$ element of it alone on a empty slot we would obtain
\begin{equation}
    (1+x_{2}+...+x_{N})^{2} \geq 1^{2} + (x_{2}^{2} + ... + x_{N}^{2}) 
\end{equation}
and then the sum, in fact, decreases.

\textbf{Add a vertex from a non-empty slot to another non-empty slot with fewer vertices:} In this case, we need to show that for all $y\geq 1$ integer we have
\begin{equation} \label{eq:195}
    x^2 + (x-y)^2 \geq (x-1)^2 + (x-y+1)^2.
\end{equation}
Calculating:
\begin{equation}
    \begin{split}
        &x^2 + (x-y)^2 \geq (x-1)^2 +(x-y+1)^2\\
        \Leftrightarrow \quad &x^2 + x^2 - 2xy + y^2 \geq x^2 -2x + 1 + x^2 -2xy + 2x + y^2 -2y +1\\
        \Leftrightarrow \quad &-2y+1 \leq 0\\
        \Leftrightarrow \quad &-y \leq -1\\
        \Leftrightarrow \quad &y \geq 1
    \end{split}
\end{equation}

The same reason above is valid for the other $s-1$ bigger communities, such that
\begin{equation}
    \max_{C}{S_{1}(C)} = \max_{C}{S_{2}(C)} = ... = \max_{C}{S_{s}(C)} = S_{i}(C^{*})
\end{equation}
for any $i \in \{1,2,...,s\}$. Then
\begin{equation}
    \max_{C}{S_{1}(C)} + ... + \max_{C}{S_{s}(C)} = \max_{C}{(S_{1}(C)+...+S_{s}(C))} \geq \max_{C' \text{allowed}}{S_{1}(C')+...+S_{s}(C')} 
\end{equation}
where the equality is due to the previous equation and the last maximum is restricted to allowed configurations $C'$.

Finally, we obtain that the bound we had is in fact a bound, that is, from the $rs^{2}m^{2}$ entries with $M_{in}$ of $M_{r}$, we can subtract the $M_{in}$ entries of any allowed configuration of $M_{s}$ and obtain the positive lower bound
\begin{equation}\label{eq:199}
\begin{split}
     rs^2m^2 - \max_{C' \text{allowed}}{(S_{1}(C')+...+S_{s}(C'))} &\geq rs^2m^2 - s\max_{C}{S_{1}(C)}> 0\\
     &= rs^{2}m^{2} - sS_{1}(C^{*})>0
\end{split}
\end{equation}
for all valid configurations $C'$ where the last inequality comes from Claim \ref{claim:1}.

The bound we achieve is then
\begin{equation}
    \text{SDP}(M_{r}-M_{s}) \geq \left(rs^2m^{2} - s\left(s^2\floor[\bigg]{\frac{r}{s}}m^{2} - rm^{2} + s\floor[\bigg]{\frac{r}{s}}m^{2}\right) \right)(M_{in}-M_{out})
\end{equation}
or in term of the remainder $t$ defined in \ref{eq:remainder}
\begin{equation} \label{eq:138}
    \boxed{\text{SDP}(M_{r}-M_{s}) \geq [s^2t-st^2]m^2(M_{in}-M_{out}) \quad\text{$\forall M_{s}$ \, with \,$s\nmid r$}.}
\end{equation}

\begin{obs}
    The most important idea of this part was to relax the calculation to include the not allowed configurations. If we had not done that, we would have had to worry about all possible relative configurations of vertices in $r$ and $s$ communities at the same time to find the bound, which would be a very difficult problem.
\end{obs}

\underline{\textbf{Case 2. $s\mid r$}:} We divide it further into two subcases.

\underline{\textbf{Subcase 2.1:}} We have $s$ divides $r$ and the $r$ smaller communities are totally contained in the $s$ bigger ones, i.e., we have configuration $\mathcal{C}$ or $r-2$ smaller communities are totally contained in the $s$ bigger ones and in the two smaller communities that remain there are less than $m$ elements of each one of them in some of the bigger community. 

\begin{obs}
    By the division condition it is easy to see that $r-2$ is the second largest possible number of smaller communities totally contained in the biggest ones.
\end{obs}
For example, for $s=3$, $r=6$, $m=2$ we can have as in figure \ref{fig:subcase1} where there is $1<2=m$ element of each of the two smaller communities (not totally contained in the bigger ones) in the bigger communities $1$ and $2$.
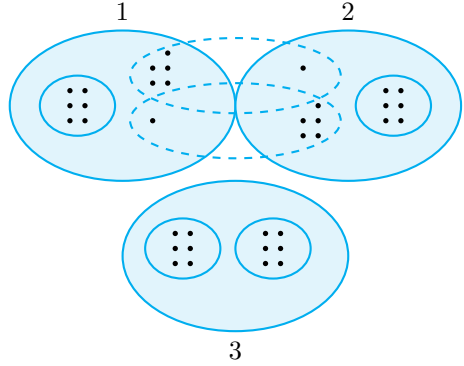
\begin{figure}[H]
    \centering
    \begin{tikzpicture}

    \tikzstyle{group} = [ellipse, draw=cyan, thick, minimum width=3cm, minimum height=2cm, fill=cyan!10, inner sep=0pt]
    \tikzstyle{dot} = [circle, fill=black, minimum size=2pt, inner sep=0pt] 
    \tikzstyle{subgroup} = [ellipse, draw=cyan, thick, minimum width=1cm, minimum height=0.8cm, inner sep=0pt]

    \node[group, label=above:1] (circle1) at (0, 2) {}; 
    \node[group, label=above:2] (circle2) at (3, 2) {}; 
    \node[group, label=below:3] (circle3) at (1.5, 0) {}; 

    \node[subgroup] (subgroup1a) at (-0.6, 2.0) {}; 
    \node[subgroup] (subgroup2b) at (3.6, 2.0) {}; 
    \node[subgroup] (subgroup3a) at (0.8, 0.1) {}; 
    \node[subgroup] (subgroup3b) at (2.0, 0.1) {}; 

    \foreach \x/\y in {-0.7/2.2, -0.5/2.2, -0.7/2.0, -0.5/2.0, -0.7/1.8, -0.5/1.8} {
        \node[dot] at (\x,\y) {};
    }

    \foreach \x/\y in {3.5/2.2, 3.7/2.2, 3.5/2.0, 3.7/2.0, 3.5/1.8, 3.7/1.8} {
        \node[dot] at (\x,\y) {};
    }
    
    \foreach \x/\y in {0.7/0.3, 0.9/0.3, 0.7/0.1, 0.9/0.1, 0.7/-0.1, 0.9/-0.1} {
        \node[dot] at (\x,\y) {};
    }
    
    \foreach \x/\y in {1.9/0.3, 2.1/0.3, 1.9/0.1, 2.1/0.1, 1.9/-0.1, 2.1/-0.1} {
        \node[dot] at (\x,\y) {};
    }

    \foreach \x/\y in {2.4/2.5, 2.6/2.0, 2.4/1.8, 2.6/1.8, 2.4/1.6, 2.6/1.6} {
        \node[dot] at (\x,\y) {};
    }

    \foreach \x/\y in {0.4/2.5, 0.6/2.5, 0.4/2.3, 0.6/2.3, 0.6/2.7, 0.4/1.8} {
        \node[dot] at (\x,\y) {};
    }

    \draw[cyan, thick, dashed] (1.5, 2.4) ellipse [x radius=1.4cm, y radius=0.5cm];
    \draw[cyan, thick, dashed] (1.5, 1.8) ellipse [x radius=1.4cm, y radius=0.5cm];

    \end{tikzpicture}
    \caption{Example of subcase $2.1$}
    \label{fig:subcase1}
\end{figure}

For the configurations of vertices of the subcase $2.1$, instead of the previous $Z$ we consider another $\tilde{Z} \in \text{PSD}_{1}(n)$ of the following shape: $\tilde{Z}$ is the matrix with $s$ diagonal big blocks formed by sub-blocks of size $sm$ and these sub-blocks alternating $+1$ and $-1$. The other elements are zero.

For example, for $r=6$, $s=2$, $m=2$ and $n=24$ we have
\begin{equation}\label{eq:ztilde}
   \tilde{Z} = \setlength{\arraycolsep}{0pt}
  \setlength{\delimitershortfall}
  {0pt}\setlength{\fboxsep}{1pt}
  \begin{pmatrix}
    \,\begin{matrix}
        \framebox[1cm][c]{\,$\begin{matrix} 1 & 1 \\
        1 & 1 \end{matrix}$\,} & \framebox[1cm][c]{\,$\begin{matrix} -1 & -1 \\
        -1 & -1 \end{matrix}$\,} & \framebox[1cm][c]{\,$\begin{matrix} 1 & 1 \\
        1 & 1 \end{matrix}$\,}\\
        \framebox[1cm][c]{\,$\begin{matrix} -1 & -1 \\
        -1 & -1 \end{matrix}$\,} & \framebox[1cm][c]{\,$\begin{matrix} 1 & 1 \\
        1 & 1 \end{matrix}$\,}& \framebox[1cm][c]{\,$\begin{matrix} -1 & -1 \\
        -1 & -1 \end{matrix}$\,}\\
        \framebox[1cm][c]{\,$\begin{matrix} 1 & 1 \\
        1 & 1 \end{matrix}$\,} & \framebox[1cm][c]{\,$\begin{matrix} -1 & -1 \\
        -1 & -1 \end{matrix}$\,} & \framebox[1cm][c]{\,$\begin{matrix} 1 & 1 \\
        1 & 1 \end{matrix}$\,}
  \end{matrix} &0  \,  \\
  
    \, 0& \begin{matrix}
        \framebox[1cm][c]{\,$\begin{matrix} 1 & 1 \\
        1 & 1 \end{matrix}$\,} & \framebox[1cm][c]{\,$\begin{matrix} -1 & -1 \\
        -1 & -1 \end{matrix}$\,} & \framebox[1cm][c]{\,$\begin{matrix} 1 & 1 \\
        1 & 1 \end{matrix}$\,}\\
        \framebox[1cm][c]{\,$\begin{matrix} -1 & -1 \\
        -1 & -1 \end{matrix}$\,} & \framebox[1cm][c]{\,$\begin{matrix} 1 & 1 \\
        1 & 1 \end{matrix}$\,}& \framebox[1cm][c]{\,$\begin{matrix} -1 & -1 \\
        -1 & -1 \end{matrix}$\,}\\
        \framebox[1cm][c]{\,$\begin{matrix} 1 & 1 \\
        1 & 1 \end{matrix}$\,} & \framebox[1cm][c]{\,$\begin{matrix} -1 & -1 \\
        -1 & -1 \end{matrix}$\,} & \framebox[1cm][c]{\,$\begin{matrix} 1 & 1 \\
        1 & 1 \end{matrix}$\,}
  \end{matrix} 
  \end{pmatrix}
\end{equation}
Such $\tilde{Z}$ is positive semidefinite because it is the following sum of rank one matrices with positive coefficients
\begin{equation}
\begin{split}
    \tilde{Z} &= \begin{pmatrix*}[r] 1 \\-1 \\1 \\\vdots \\1 \\0 \\\vdots \\0 \end{pmatrix*} \begin{pmatrix*}
    \smash[b]{\block{\frac{r}{s} \, \text{times}}}  & 0 & \cdots & 0  
\end{pmatrix*}\\
&+ \underbrace{\cdots}_{\text{total of} \, s \, \text{terms}} + \begin{pmatrix*}[r] 0 \\\vdots \\0 \\1 \\\vdots \\1 \\-1 \\1 \end{pmatrix*} \begin{pmatrix*}
    0 & \cdots & 0 & 1 & \cdots & 1 & -1 & 1 
\end{pmatrix*}
\end{split}
\end{equation}

\begin{claim}
We can consider, without loss of generality, that in the case of the communities totally contained (configuration $\mathcal{C}$), the enumeration of the vertices in the biggest communities is consecutive (in the smaller communities it already is because of the fixed $M_{r}$ we are using). This means consider that the $\frac{r}{s}$ first smaller communities are in the first biggest one and, generalizing, the $\frac{rn}{s}$-th smaller communities are in the n-th biggest one. The figure below illustrate this.
\end{claim}

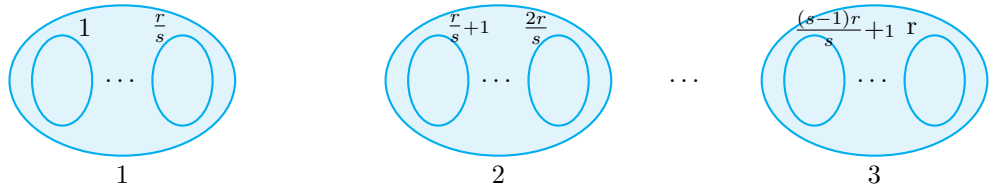
\begin{figure}[H]
    \centering
    \begin{tikzpicture}
        \tikzstyle{group} = [ellipse, draw=cyan, thick, minimum width=3cm, minimum height=2cm, fill=cyan!10, inner sep=0pt]
        \tikzstyle{subgroup} = [ellipse, draw=cyan, thick, minimum width=0.8cm, minimum height=1.2cm, inner sep=0pt]

        \node[group, label=below:1] (ellipse1) at (0, 0) {}; 
        \node[group, label=below:2] (ellipse2) at (5, 0) {}; 
        \node[group, label=below:3] (ellipse3) at (10, 0) {}; 

        \node[subgroup] (sub1) at (-0.8, 0) {}; 
        \node[subgroup] (sub2) at (0.8, 0) {}; 
        
        \node[subgroup] (sub3) at (4.2, 0) {}; 
        \node[subgroup] (sub4) at (5.8, 0) {}; 
        
        \node[subgroup] (sub5) at (9.2, 0) {}; 
        \node[subgroup] (sub6) at (10.8, 0) {}; 
        
        \node at (-0.5, 0.7) {1};  
        \node at (0.5, 0.7) {$\frac{r}{s}$}; 
        
        \node at (4.6, 0.7) {$\frac{r}{s}{\scriptstyle +1}$};  
        \node at (5.5, 0.7) {$\frac{2r}{s}$}; 
        
        \node at (9.6, 0.7) {$\frac{(s-1)r}{s}{+\scriptstyle 1}$};  
        \node at (10.5, 0.7) {r}; 

        \node at (7.5, 0) {\dots};

        \node at (0, 0) {\dots};
        \node at (5, 0) {\dots};
        \node at (10, 0) {\dots};
    \end{tikzpicture}
    \caption{We consider this enumeration without loss of generality}
    \label{fig:enumeration}
\end{figure}

\begin{proofclaim}
In fact, considering this relative numbering of the smaller communities inside the bigger ones (the communities themselves, not its elements) as the identity element of the symmetric group $S_{r}$, we can get from any other enumeration to this one by a finite number of transpositions (just do it one community at a time). This results in permuting the rows and columns of $M_{s}$ and consequently of $M_{r}-M_{s}$ by sets of $sm$ elements.
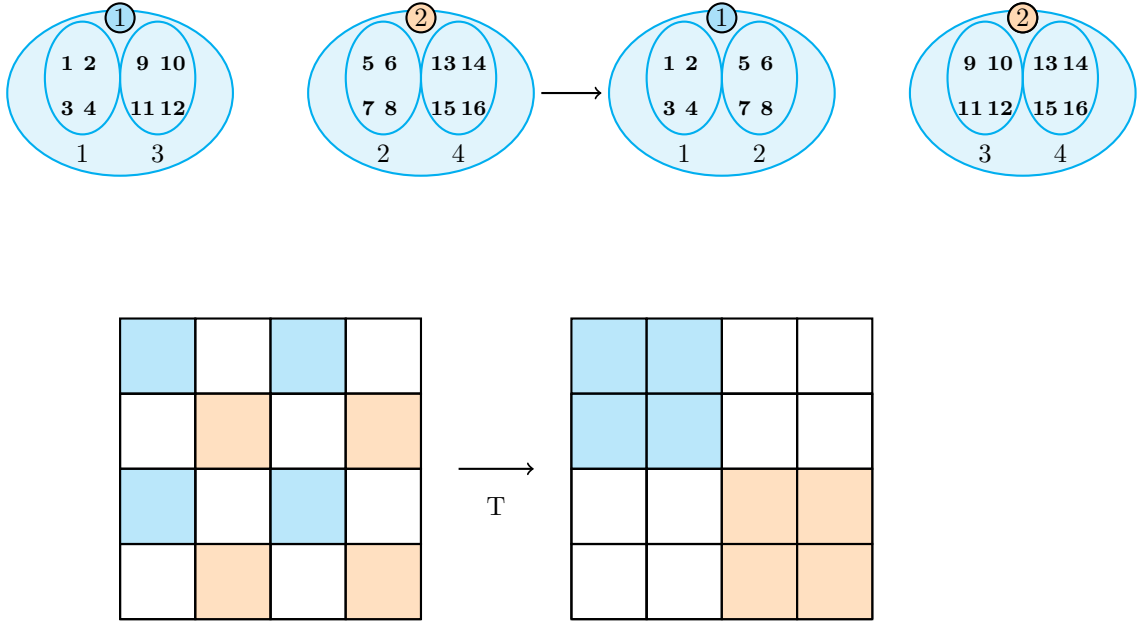
\begin{figure}[H]
    \centering
    \begin{tikzpicture}

        \tikzstyle{group} = [ellipse, draw=cyan, thick, minimum width=3cm, minimum height=2.2cm, fill=cyan!10, inner sep=0pt]
        \tikzstyle{subgroup} = [ellipse, draw=cyan, thick, minimum width=1cm, minimum height=1.5cm, inner sep=0pt]
        \tikzstyle{labelcirc} = [draw, circle, inner sep=1pt, thick]
        \tikzstyle{matrixcell} = [draw, thick, minimum size=1cm]
        \tikzstyle{highlight} = [draw, thick, fill opacity=0.6]
        \tikzstyle{numstyle} = [font=\bfseries\footnotesize] 

        \node[group] (group1) at (-4, 2) {};
        \node[subgroup] (sub1a) at (-4.5, 2.2) {}; 
        \node[subgroup] (sub1b) at (-3.5, 2.2) {}; 
        \node[labelcirc, fill=cyan!30] at (-4, 3) {1};

        \node[numstyle] at (-4.7, 2.4) {1};
        \node[numstyle] at (-4.4, 2.4) {2};
        \node[numstyle] at (-4.7, 1.8) {3};
        \node[numstyle] at (-4.4, 1.8) {4};
        \node[numstyle] at (-3.7, 2.4) {9};
        \node[numstyle] at (-3.3, 2.4) {10};
        \node[numstyle] at (-3.7, 1.8) {11};
        \node[numstyle] at (-3.3, 1.8) {12};
        
        \node at (-4.5, 1.2) {1};
        \node at (-3.5, 1.2) {3};
        
        \node[group] (group2) at (0, 2) {};
        \node[subgroup] (sub2a) at (-0.5, 2.2) {}; 
        \node[subgroup] (sub2b) at (0.5, 2.2) {}; 
        \node[labelcirc, fill=orange!30] at (0, 3) {2};
        
        \node[numstyle] at (-0.7, 2.4) {5};
        \node[numstyle] at (-0.4, 2.4) {6};
        \node[numstyle] at (-0.7, 1.8) {7};
        \node[numstyle] at (-0.4, 1.8) {8};
        \node[numstyle] at (0.3, 2.4) {13};
        \node[numstyle] at (0.7, 2.4) {14};
        \node[numstyle] at (0.3, 1.8) {15};
        \node[numstyle] at (0.7, 1.8) {16};
        
        \node at (-0.5, 1.2) {2};
        \node at (0.5, 1.2) {4};
        
        \node[group] (group3) at (4, 2) {};
        \node[subgroup] (sub3a) at (3.5, 2.2) {}; 
        \node[subgroup] (sub3b) at (4.5, 2.2) {}; 
        \node[labelcirc, fill=cyan!30] at (4, 3) {1};
        
        \node[numstyle] at (3.3, 2.4) {1};
        \node[numstyle] at (3.6, 2.4) {2};
        \node[numstyle] at (3.3, 1.8) {3};
        \node[numstyle] at (3.6, 1.8) {4};
        \node[numstyle] at (4.3, 2.4) {5};
        \node[numstyle] at (4.6, 2.4) {6};
        \node[numstyle] at (4.3, 1.8) {7};
        \node[numstyle] at (4.6, 1.8) {8};

        \node at (3.5, 1.2) {1};
        \node at (4.5, 1.2) {2};
        
        \node[group] (group4) at (8, 2) {};
        \node[subgroup] (sub4a) at (7.5, 2.2) {}; 
        \node[subgroup] (sub4b) at (8.5, 2.2) {}; 
        \node[labelcirc, fill=cyan!30] at (8, 3) {1};
        
        \node[numstyle] at (7.3, 2.4) {9};
        \node[numstyle] at (7.7, 2.4) {10};
        \node[numstyle] at (7.3, 1.8) {11};
        \node[numstyle] at (7.7, 1.8) {12};
        \node[numstyle] at (8.3, 2.4) {13};
        \node[numstyle] at (8.7, 2.4) {14};
        \node[numstyle] at (8.3, 1.8) {15};
        \node[numstyle] at (8.7, 1.8) {16};
        
        \node at (7.5, 1.2) {3};
        \node at (8.5, 1.2) {4};
        \node[labelcirc, fill=orange!30] at (8, 3) {2};

        \draw[thick, ->] ($(group3)!.6!(group2)$) -- ($(group3)!.4!(group2)$);

        \draw[matrixcell] (-4,-5) grid (0,-2);

        \draw[highlight, fill=cyan!40] (-4,-2) rectangle (-3,-1); 
        \draw[highlight] (-3,-2) rectangle (-2,-1); 
        \draw[highlight, fill=cyan!40] (-2,-2) rectangle (-1,-1); 
        \draw[highlight] (-1,-2) rectangle (0,-1); 

        \draw[highlight] (-4,-3) rectangle (-3,-2);
        \draw[highlight,fill=orange!40] (-3,-3) rectangle (-2,-2);
        \draw[highlight] (-2,-3) rectangle (-1,-2); 
        \draw[highlight,fill=orange!40] (-1,-3) rectangle (0,-2); 
        
        \draw[highlight, fill=cyan!40] (-4,-4) rectangle (-3,-3); 
        \draw[highlight] (-3,-4) rectangle (-2,-3); 
        \draw[highlight, fill=cyan!40] (-2,-4) rectangle (-1,-3); 
        \draw[highlight] (-1,-4) rectangle (0,-3);

        \draw[highlight] (-4,-5) rectangle (-3,-4); 
        \draw[highlight, fill=orange!40] (-3,-5) rectangle (-2,-4); 
        \draw[highlight] (-2,-5) rectangle (-1,-4); 
        \draw[highlight, fill=orange!40] (-1,-5) rectangle (0,-4);
        
        \draw[matrixcell] (2,-5) grid (6,-2);

        \draw[highlight, fill=cyan!40] (2,-2) rectangle (3,-1); 
        \draw[highlight, fill=cyan!40] (3,-2) rectangle (4,-1);
        \draw[highlight] (4,-2) rectangle (5,-1); 
        \draw[highlight] (5,-2) rectangle (6,-1); 

        \draw[highlight, fill=cyan!40] (2,-3) rectangle (3,-2);
        \draw[highlight, fill=cyan!40] (3,-3) rectangle (4,-2);
        \draw[highlight] (4,-3) rectangle (5,-2); 
        \draw[highlight] (5,-3) rectangle (6,-2); 

        \draw[highlight] (2,-4) rectangle (3,-3);
        \draw[highlight] (3,-4) rectangle (4,-3);
        \draw[highlight, fill=orange!40] (4,-4) rectangle (5,-3); 
        \draw[highlight, fill=orange!40] (5,-4) rectangle (6,-3);
        
        \draw[highlight] (2,-5) rectangle (3,-4);
        \draw[highlight] (3,-5) rectangle (4,-4);
        \draw[highlight, fill=orange!40] (4,-5) rectangle (5,-4); 
        \draw[highlight, fill=orange!40] (5,-5) rectangle (6,-4);

        \draw[thick,->] (0.5,-3) -- (1.5,-3);
        \node at (1,-3.5) {T};

    \end{tikzpicture}
    \caption{Operation done to the matrix $M_{s}$, the color represents the $M_{in}$ entries of it}
    \label{fig:matrix-operation}
\end{figure}

Doing the same permutation in the corresponding blocks of $\tilde{Z}$, this operation maintains the positive semidefinite property of $\tilde{Z}$ because we permute equally rows and columns (and, of course, the property of having diagonal $1$). Then, calculating the inner product of the permuted matrices we end up with the same inner product, that is, the same bound. \claimqed
\end{proofclaim}

We consider, by now, just the configuration $\mathcal{C}$ where all $r$ smaller communities are totally contained in the bigger ones and with the enumeration described above. 
\begin{claim} In the case we have all smaller communities contained in the bigger ones, we claim that the bound is
    \begin{equation}\label{eq:203}
    \boxed{\text{SDP}(M_{r}-M_{s}) \geq \langle M_{r}-M_{s}, \tilde{Z}\rangle = 2s^3\left[\sum_{i=1}^{\frac{r}{s}-1}\left(\frac{r}{s}-i\right)(-1)^{i+1}\right]m^2(M_{in}-M_{out})}
\end{equation}
\end{claim}

\begin{obs}
It is important to notice that the reasoning here is slightly different from the previous case. In the previous case, we were considering for the fixed $Z$ all the configurations where $r \nmid s$. Now, for the fixed matrix $\tilde{Z}$, given that we also fixed the enumeration \textit{w.l.o.g.}, we need to find the bound just for one case: the case where the smaller communities are totally contained in the bigger ones. Then, the bound is just the inner product  $\langle M_{r}-M_{s}, \tilde{Z}\rangle$ with $M_{s}$ representing this specific configuration. We do the same thing after for the remaining configurations in this subcase $2.1$ where $r-2$ smaller communities are totally contained in the $s$ bigger ones and in the two communities that remain there are less than $m$ elements of each one of them in some of the bigger community. 
\end{obs}

\begin{proofclaim}In fact, in this case we have 
\begin{equation}\label{eq:205}
  M_{r} - M_{s} = \setlength{\arraycolsep}{1.5pt}
  \begin{pmatrix}
  \, \fbox{\,$\begin{matrix}
    \,\framebox[0.5cm][c]{$0$} & \framebox[0.5cm][c]{$-1$} & \cdots & \framebox[0.5cm][c]{$-1$}\,  \\
    \, \framebox[0.5cm][c]{$-1$} & \framebox[0.5cm][c]{$0$} & \cdots & \framebox[0.5cm][c]{$-1$}\, \\
    \, \cdots & \framebox[0.5cm][c]{$-1$} & \cdots & \vdots \, \\
    \, \framebox[0.5cm][c]{$-1$} & \framebox[0.5cm][c]{$-1$} & \cdots & \framebox[0.5cm][c]{$0$} \, \\
  \end{matrix}$ \,} & & 0 \, \\
     & \ddots & \\
    0 & & \fbox{\,$\begin{matrix}
    \,\framebox[0.5cm][c]{$0$} & \framebox[0.5cm][c]{$-1$} & \cdots & \framebox[0.5cm][c]{$-1$}\,  \\
    \, \framebox[0.5cm][c]{$-1$} & \framebox[0.5cm][c]{$0$} & \cdots & \framebox[0.5cm][c]{$-1$}\, \\
    \, \cdots & \framebox[0.5cm][c]{$-1$} & \cdots & \vdots \, \\
    \, \framebox[0.5cm][c]{$-1$} & \framebox[0.5cm][c]{$-1$} & \cdots & \framebox[0.5cm][c]{$0$} \,
  \end{matrix} $\,} 
  \end{pmatrix} (M_{in}-M_{out})
\end{equation}
where the smaller blocks have size $sm\times sm$ and repeats $\frac{r}{s}$ times inside the bigger one, and we have $s$ bigger blocks of size $rm\times rm$. The diagonal sub-blocks represent the $M_{in}$ entries of $M_{r}$ which are subtracted by the $M_{in}$ entries of $M_{s}$, given blocks with $0$ entries, the remaining sub-blocks represents the $M_{out}$ entries of $M_{r}$ subtract still by the $M_{in}$ entries of $M_{s}$ giving blocks with $-(M_{in}-M_{out})$ entries, and finally all entries outside the bigger blocks are $M_{out}-M_{out}=0$.

Let us consider what happens with the inner product of one of these big blocks with the corresponding block in $\tilde{Z}$ and only above the diagonal.
\begin{equation}\label{eq:206}
 \setlength{\arraycolsep}{2pt}
\begin{pmatrix}
    \,\framebox[0.5cm][c]{$0$} & \framebox[0.5cm][c]{$-1$} & \framebox[0.5cm][c]{$-1$} & \cdots & \framebox[0.5cm][c]{$-1$}\,  \\
    \, \framebox[0.5cm][c]{$-1$} & \framebox[0.5cm][c]{$0$} & \framebox[0.5cm][c]{$-1$} & \cdots & \framebox[0.5cm][c]{$-1$}\, \\
    \, \framebox[0.5cm][c]{$-1$} & \framebox[0.5cm][c]{$-1$} & \framebox[0.5cm][c]{$0$} & \cdots & \vdots \, \\
    \, \vdots & \vdots & \cdots & \cdots  & \framebox[0.5cm][c]{$-1$}\, \\
    \, \framebox[0.5cm][c]{$-1$} & \framebox[0.5cm][c]{$-1$} & \cdots & \cdots & \framebox[0.5cm][c]{$0$} \, \\
  \end{pmatrix}
  \cdot \setlength{\arraycolsep}{2pt}
\begin{pmatrix}
    \,\framebox[0.5cm][c]{$1$} & \framebox[0.5cm][c]{$-1$} & \framebox[0.5cm][c]{$1$} & 
    \cdots & \,  \\
    \, \framebox[0.5cm][c]{$-1$} & \framebox[0.5cm][c]{$1$} & \framebox[0.5cm][c]{$-1$} & \cdots & \, \\
    \, \framebox[0.5cm][c]{$1$} & \framebox[0.5cm][c]{$-1$} & \framebox[0.5cm][c]{$1$} & \cdots & \vdots \, \\
    \, \vdots & \vdots & \cdots & \cdots  & \framebox[0.5cm][c]{$-1$}\, \\
    \, &  & \cdots & \cdots & \framebox[0.5cm][c]{$1$} \, \\
  \end{pmatrix}(M_{in}-M_{out})
\end{equation}

The first diagonal above the principal one has $\frac{r}{s}-1$ blocks of $(sm)^2$ elements. Doing the inner product in this first diagonal we obtain
\begin{equation}
    \left(\frac{r}{s}-1\right)(-1)^{1+1}(sm)^{2}(M_{in}-M_{out}).
\end{equation}

In the second diagonal above the principal one we have $\frac{r}{s}-2$ elements, but now the multiplication of the corresponding blocks gives a negative number
\begin{equation}
    \left(\frac{r}{s}-2\right)(-1)^{1+2}(sm)^{2}(M_{in}-M_{out}).
\end{equation}

We repeat this until the $\left(\frac{r}{s}-1\right)$ diagonal above the principal one that gives us
\begin{equation}
    \left(\frac{r}{s}-\left(\frac{r}{s}-1\right)\right)(-1)^{1+\left(\frac{r}{s}-1\right)}(sm)^{2}(M_{in}-M_{out}).
\end{equation}
Considering that in equation \ref{eq:203} the $2$ factor comes from the symmetry in relation to the diagonal and one $s$ factor comes from counting all $s$ bigger blocks, summing the above equations give the proof of the claim. \claimqed
\end{proofclaim}

\begin{obs}\label{remark:9}
    We prove together with the next bound, the fact that this bound is positive and of right order.
\end{obs}

Let us now consider $r-2$ smaller communities totally contained in the bigger ones and in the $2$ communities that remain there are $k<m$ elements in each one of them inside some bigger one. Again, using the permutation arguments, we can suppose, without loss of generality, that the vertices of these $2$ smaller communities are in a specific configuration inside the bigger ones described in the next paragraph.

\begin{obs}
    In this case, it is possible that we would have first to make permutations to have the right enumeration of vertices inside the smaller ones relatively to the bigger ones (see next paragraph) and then do the permutation in blocks to have the two bigger communities involved just near each other (see next paragraph). 
\end{obs}

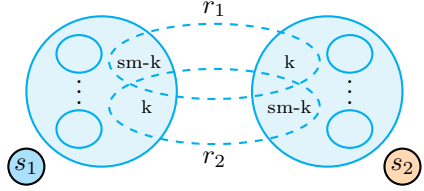
\begin{figure}[H]
    \centering
    \begin{tikzpicture}

        \tikzstyle{group} = [ellipse, draw=cyan, thick, fill=cyan!10, inner sep=0pt]
        \tikzstyle{subgroup} = [ellipse, draw=cyan, thick, fill=cyan!10, inner sep=0pt]
        \tikzstyle{labelcirc} = [draw, circle, inner sep=1pt, thick]

        \node[group] (circle1) at (0, 2) [minimum width=2cm, minimum height=2cm] {}; 
        \node[group] (circle2) at (3, 2) [minimum width=2cm, minimum height=2cm] {}; 

        \node[subgroup] (subgroup1a) at (-0.3, 2.5) [minimum width=0.6cm, minimum height=0.5cm] {}; 
        \node[subgroup] (subgroup1b) at (-0.3, 1.5) [minimum width=0.6cm, minimum height=0.5cm] {}; 

        \node[subgroup] (subgroup2a) at (3.3, 2.5) [minimum width=0.6cm, minimum height=0.5cm] {}; 
        \node[subgroup] (subgroup2b) at (3.3, 1.5) [minimum width=0.6cm, minimum height=0.5cm] {}; 

        \draw[cyan, thick, dashed] (1.5, 2.4) ellipse [x radius=1.4cm, y radius=0.5cm];
        \draw[cyan, thick, dashed] (1.5, 1.8) ellipse [x radius=1.4cm, y radius=0.5cm];

        \node at (0.5, 2.4) {\scriptsize sm-k};
        \node at (0.6, 1.8) {\scriptsize k};
        \node at (2.5, 2.4) {\scriptsize k};
        \node at (2.5, 1.8) {\scriptsize sm-k};

        \node at (-0.3, 2.1) {$\vdots$};
        \node at (3.3, 2.1) {$\vdots$};
        \node at (1.5, 3.1) {$r_1$};
        \node at (1.5, 1.1) {$r_2$};

        \node[labelcirc, fill=cyan!30] at (-1.0, 1) {$s_1$};  
        \node[labelcirc, fill=orange!30] at (4.0, 1) {$s_2$};   

    \end{tikzpicture}
    \caption{Example of subcase $2.1$ but not totally contained}
    \label{fig:example-subcase}
\end{figure}

 Let $s_{1}=s_{2}-1$ be the indices of the $2$ bigger communities that contain the not totally contained smaller ones. Consider the configuration in Figure \ref{fig:example-subcase} with $r_{1}=r_{2}-1$, that is, the $sm-k$ first elements of $r_{1}$ are in $s_{1}$ and the $k$ remaining ones are in $s_{2}$ and the first $k$ elements of $r_{2}$ are in $s_{1}$ and the $sm-k$ remaining ones are in $s_{2}$. In this case, in the diagonal blocks that represent $r_{1}$ and $r_{2}$ we have the scheme of Figure \ref{fig:matrix-visualization}.

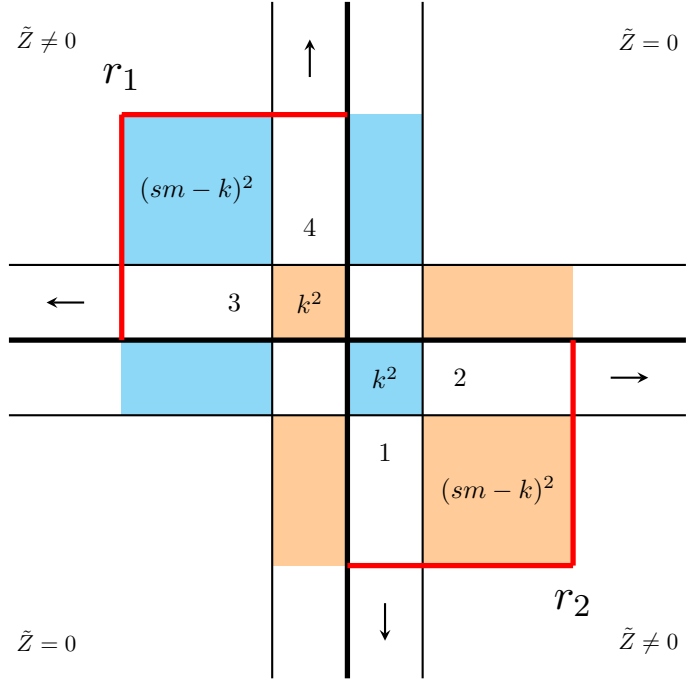
\begin{figure}[H]
    \centering
    \begin{tikzpicture}

        \tikzstyle{matrixcell} = [draw, thick, minimum size=1cm]
        \tikzstyle{axis_label} = [font=\large]

        \fill[cyan!40] (-3,2) rectangle (-2,3);
        \fill[cyan!40] (-2,2) rectangle (-1,3);
        \fill[cyan!40] (0,2) rectangle (1,3);

        \fill[cyan!40] (-3,1) rectangle (-2,2);
        \fill[cyan!40] (-2,1) rectangle (-1,2);
        \fill[cyan!40] (0,1) rectangle (1,2);

        \fill[orange!40] (-1,0) rectangle (0,1);
        \fill[orange!40] (1,0) rectangle (2,1);
        \fill[orange!40] (2,0) rectangle (3,1);

        \fill[cyan!40] (-3,-1) rectangle (-2,0);
        \fill[cyan!40] (-2,-1) rectangle (-1,0);
        \fill[cyan!40] (0,-1) rectangle (1,0);

        \fill[orange!40] (-1,-2) rectangle (0,-1);
        \fill[orange!40] (1,-2) rectangle (2,-1);
        \fill[orange!40] (2,-2) rectangle (3,-1);

        \fill[orange!40] (-1,-3) rectangle (0,-2);
        \fill[orange!40] (1,-3) rectangle (2,-2);
        \fill[orange!40] (2,-3) rectangle (3,-2);

        \draw[line width=0.7mm] (-4.5,0) -- (4.5,0);   
        \draw[thick] (-4.5,1) -- (4.5,1);   
        \draw[thick] (-4.5,-1) -- (4.5,-1);   
        \draw[line width=0.7mm] (0,-4.5) -- (0,4.5);   
        \draw[thick] (-1,-4.5) -- (-1,4.5);   
        \draw[thick] (1,-4.5) -- (1,4.5);   

        \draw[line width=0.7mm, red] (3.0,0.0) -- (3.0,-3.0);
        \draw[thick,->,>=stealth] (3.5,-0.5) -- (4,-0.5); 
        \draw[line width=0.7mm, red] (-3.0,0.0) -- (-3.0,3.0);
        \draw[thick,->,>=stealth] (-3.5,0.5) -- (-4,0.5); 

        \draw[line width=0.7mm, red] (-3.0,3.0) -- (0.0,3.0);
        \draw[thick,->,>=stealth] (-0.5,3.5) -- (-0.5,4); 
        \draw[line width=0.7mm, red] (0.0,-3.0) -- (3.0,-3.0);
        \draw[thick,->,>=stealth] (0.5,-3.5) -- (0.5,-4); 

    \node[axis_label] at (-4.0,4.0) {\scalebox{0.8}{$\tilde{Z} \neq 0$}};  
    \node[axis_label] at (4.0,4.0) {\scalebox{0.8}{$\tilde{Z} = 0$}};   
    \node[axis_label] at (-4.0,-4.0) {\scalebox{0.8}{$\tilde{Z} = 0$}}; 
    \node[axis_label] at (4.0,-4.0) {\scalebox{0.8}{$\tilde{Z} \neq 0$}};  
    
    \node[axis_label] at (-3.0, 3.5) {\scalebox{1.5}{$r_{1}$}};  
    \node[axis_label] at (3.0,-3.5) {\scalebox{1.5}{$r_{2}$}};  

    \node at (-0.5,1.5) {4};  
    \node at (-1.5,0.5) {3};  
    \node at (1.5,-0.5) {2};  
    \node at (0.5,-1.5) {1};  

    \node at (-0.5,0.5) {$k^{2}$};  
    \node at (0.5,-0.5) {$k^{2}$};  

    \node at (-2.0,2.0) {$(sm-k)^{2}$};  
    \node at (2.0,-2.0) {$(sm-k)^{2}$};   

    \end{tikzpicture}
    \caption{What happens inside $r_{1}$ and $r_{2}$ blocks. The cyan represents elements in $s_{1}$ and the orange in $s_{2}$.}
    \label{fig:matrix-visualization}
\end{figure}

\begin{obs}
    The division of $\tilde{Z}=0$ and $\tilde{Z}\neq0$ is done by the thicker horizontal and vertical middle lines. The divisions of $r_{1}$ and $r_{2}$ are done by these lines plus the red lines.
\end{obs}

\begin{obs}
    To help with visualization, the above image represents the middle of the matrix in equation \ref{eq:ztilde} (when considering only the bigger communities $s_{1}$ and $s_{2}$) but with an ``overlap'' between the big blocks which is represented here by the ``mixture'' of the colors cyan and orange.
\end{obs}

\begin{claim}
In the region that matters to us of $\tilde{Z}\neq 0$, the matrix $M_{r}-M_{s}$ is the same as in equation \ref{eq:205} except by the $4$ enumerated and not colored strips on Figure \ref{fig:matrix-visualization}.
\end{claim}
\begin{proofclaim}
In fact, the colored region intersecting $r_{1}$ and $\tilde{Z}\neq0$ represents $M_{in}$ coming from $r_{1}$ (because of the previous remark) and $M_{in}$ coming from $s_{1}$ or $s_{2}$ (because of the coloring), giving a $M_{in}-M_{in}=0$ entry. The same is valid for $r_{2}$. This gives part of the $0$ diagonal sub-blocks of equation \ref{eq:205}. However, the other part of these previous $0$ blocks become $1$'s rectangular blocks given by  the $4$ enumerated and not colored strips on Figure \ref{fig:matrix-visualization} (without considering their extension given by the arrows). In fact, in this case we still have $M_{in}$ coming from $r_1$ and $r_2$ but now $M_{out}$ (not colored) coming from $s_1$ and $s_2$. The continuation of the $4$ strips given by the arrows in the figure also have different entries from the matrix in equation \ref{eq:205}: in the case here these entries extending to the entire matrix is always $M_{out}-M_{out}=0$. \claimqed
\end{proofclaim}

When intersecting with the diagonal blocks of the smaller communities, that is, without the extension by the arrows,  these $4$ strips represents $M_{in}-M_{out}$ as already said. Then, in the bound of \ref{eq:203}, where before we had $0$ summing, now we have to \textbf{sum} 
\begin{equation}\label{eq:sum}
    4k(sm-k)(M_{in}-M_{out})
\end{equation}
where $k(sm-k)$ is the number of elements in each strip also inside the smaller communities blocks.

Outside the diagonal of the smaller communities, the strips represent $0=M_{out}-M_{out}$, as already said, then it cancels part of the bound in \ref{eq:203} where we not necessarily have $0$ entries. In fact, on this previous bound, by equation \ref{eq:206} we see that the strips are on the corners of the bigger blocks associated with $s_{1}$ and $s_{2}$ and the entries associated with these parts of the strip contributed to the bound with the sum
\begin{equation}\label{eq:subtract}
4\left[\sum_{i=1}^{\frac{r}{s}-1} (-1)^{i+1} k(sm)\right](M_{in}-M_{out}) = \begin{cases}
    0, & \frac{r}{s} -1 \quad \text{is even} \\
    4ksm(M_{in}-M_{out}), &  \frac{r}{s} -1 \quad \text{is odd}  \\
\end{cases}
\end{equation}
where the sum is in the vertical or horizontal direction, $k(sm)$ is the number of elements in the strip in each one of the $\frac{r}{s}-1$ smaller blocks remaining on each bigger one. Because now this contribution became $0$, we need to \textbf{subtract} it from the previous bound.

Then, using the next calculations for making explicit the alternating series part coming from the previous bound \ref{eq:203}, we obtain

\textbf{$\frac{r}{s}$ odd (or $\frac{r}{s}-1$ even)}

\begin{equation}
    \begin{split}
    2(sm)^{2}\left(\sum_{i=1}^{\frac{r}{s}-1}\left(\frac{r}{s}-i\right)(-1)^{i+1}\right)s
        &= 2s^{2}m^{2}\left(\frac{r}{s}-1-\left(\frac{r}{s}-2\right) + \frac{r}{s}-3 - \left(\frac{r}{s}-4\right)+ ...\right)\\
        &= 2s^{2}m^{2}(-1+2-3+4-5+...+(r/s -1))\\
        &= 2s^{3}m^{2}\frac{1}{2}\left(\frac{r}{s}-1\right) = rs^{2}m^{2} - s^{3}m^{2} > 0 
    \end{split}
\end{equation}
where the last inequality is because $r>s$.

\textbf{$\frac{r}{s}$ even (or $\frac{r}{s}-1$ odd)}

\begin{equation}
    \begin{split}
        2(sm)^{2}\left(\sum_{i=1}^{\frac{r}{s}-1}\left(\frac{r}{s}-i\right)(-1)^{i+1}\right)s
        &= 2s^{3}m^{2}(-1+2-3+4-...\\
        &+\left(\frac{r}{s} - \left(\frac{r}{s}-3\right)\right) - \left(\frac{r}{s}-\left(\frac{r}{s}-2\right)\right)+ \left(\frac{r}{s}- \left(\frac{r}{s}-1\right)\right))\\
        &= 2s^{3}m^{2}\left(-1+2-3+4-...- (r/s-3)+(r/s-2) +1\right)\\
        &=2s^{3}m^{2}\frac{1}{2}\left(\frac{r}{s}-2\right) + 2s^{3}m^{2} = rs^{2}m^{2} > 0
    \end{split}
\end{equation}

We see that in both cases we have a positive bound proportional to $m^{2}$ such that this confirms remark \ref{remark:9}. For this bound, coming from the previous bound and summing equation \ref{eq:sum} and subtracting equation \ref{eq:subtract}, we finally obtain

\textbf{$\frac{r}{s}$ odd (or $\frac{r}{s}-1$ even):}
\begin{equation}
    \text{SDP}(M_{r}-M_{s}) \geq [rs^{2}m^{2} - s^{3}m^{2} + 4k(sm-k)](M_{in}-M_{out})
\end{equation}
with $k=1,2,...,m-1$. The smaller bound is with $k=1$
\begin{equation}
    \boxed{\text{SDP}(M_{r}-M_{s}) \geq [rs^{2}m^{2} - s^{3}m^{2} + 4sm-4](M_{in}-M_{out})}
\end{equation}
where $4sm-4>0$ because $s,m\geq 2$.

\textbf{$\frac{r}{s}$ even (or $\frac{r}{s}-1$ odd):}
\begin{equation}
    \text{SDP}(M_{r}-M_{s}) \geq [rs^{2}m^{2} + 4k(sm-k) - 4ksm](M_{in}-M_{out}) = [rs^{2}m^{2} - 4k^{2}](M_{in}-M_{out})
\end{equation}
and the smaller bound is now with $k=m-1$
\begin{equation}
    \boxed{\text{SDP}(M_{r}-M_{s}) \geq [rs^{2}m^{2}-4(m-1)^{2}](M_{in}-M_{out}) \geq [rs^{2}m^{2} - 4m^{2}](M_{in}-M_{out}).} 
\end{equation}

\underline{\textbf{Subcase 2.2:}} Remaining cases where $s\mid r$. Here, we use the $Z$ matrix again
\[
  \setlength{\arraycolsep}{0pt}
  \setlength{\delimitershortfall}{0pt}
  Z =\begin{pmatrix}
    \,\resizebox{!}{1em}{\fbox{$1$}} &  &  & \,  \\
    \, & \resizebox{!}{1em}{\fbox{$1$}} &  & 0\, \\
    \,0&  & \ddots & \, \\
    \, &  &  & \resizebox{!}{1em}{\fbox{$1$}}\, \\
  \end{pmatrix}
\]

We can then use the same arguments from case $1$ to prove the following claim.

\begin{claim}
The bound is reached by a configuration with $r-2$ smaller communities totally contained in the bigger ones and in the $2$ remaining ones we have the previous division but now with $k=m$ as the image below. 
\end{claim}

\begin{figure}[H]
    \centering
    \begin{tikzpicture}

        \tikzstyle{group} = [ellipse, draw=cyan, thick, fill=cyan!10, inner sep=0pt]
        \tikzstyle{subgroup} = [ellipse, draw=cyan, thick, fill=cyan!10, inner sep=0pt]

        \node[group] (circle1) at (0, 2) [minimum width=2cm, minimum height=2cm] {}; 
        \node[group] (circle2) at (3, 2) [minimum width=2cm, minimum height=2cm] {}; 

        \node[subgroup] (subgroup1a) at (-0.3, 2.5) [minimum width=0.6cm, minimum height=0.5cm] {}; 
        \node[subgroup] (subgroup1b) at (-0.3, 1.5) [minimum width=0.6cm, minimum height=0.5cm] {}; 

        \node[subgroup] (subgroup2a) at (3.3, 2.5) [minimum width=0.6cm, minimum height=0.5cm] {}; 
        \node[subgroup] (subgroup2b) at (3.3, 1.5) [minimum width=0.6cm, minimum height=0.5cm] {}; 

        \draw[cyan, thick, dashed] (1.5, 2.4) ellipse [x radius=1.4cm, y radius=0.5cm];
        \draw[cyan, thick, dashed] (1.5, 1.8) ellipse [x radius=1.4cm, y radius=0.5cm];

        \node at (0.5, 2.4) {\scriptsize sm-m};
        \node at (0.6, 1.8) {\scriptsize m};
        \node at (2.5, 2.4) {\scriptsize m};
        \node at (2.5, 1.8) {\scriptsize sm-m};

        \node at (-0.3, 2.1) {$\vdots$};
        \node at (3.3, 2.1) {$\vdots$};

        \node at (-1.0, 1) {$s_1$};  
        \node at (4.0, 1) {$s_2$};   

    \end{tikzpicture}
    \caption{Subcase $2.2$ optimal configuration}
    \label{fig:subcase-2-configuration}
\end{figure}

\begin{proofclaim}
As we are using the same matrix $Z$ from \textbf{Case 1}, we can use the same reasoning of looking at one of the biggest community at a time. First, we observe that what we have here as remainder from subcase $2.1$ are the cases with $r-2$ or more smaller communities not totally contained in a bigger one. Then, by same reasoning as case $1$, we obtain that the bound needs to be achieved by a configuration where we have only $r-2$ smaller communities not totally contained in a bigger. This is because if we have more, we could use the operations described in equations \ref{eq:193} and \ref{eq:195} for the corresponding bigger community and improve the bound. 

We prove that the configuration that minimizes the bound is the one from the image above. In fact, for all the configurations where only $r-2$ smaller communities not totally contained in a bigger one,
$m$ represents the smallest possible number of vertices in a smaller community within one of the larger ones (remember: in the previous subcase we had $k<m$ such that now $k\geq m$). Then, using again
\begin{equation}
    x^{2} + (x-y)^{2} \geq  (x-1)^{2} + (x-y+1)^{2}
\end{equation}
for all $y\geq 1$, we obtain looking just to $s_{1}$, the corresponding $A_{r_1}^{2}$ and $A_{r_2}^2$ part of equation \ref{eq:191} as
\begin{equation}
    \begin{split}
        (sm-m)^{2} + m^{2} &\geq (sm-(m+1))^{2} + (m+1)^{2}\\
        &\geq (sm-m-1))^{2} + (m+1)^{2}\\
        &\geq (sm-m-2)^{2} + (m+2)^{2}\\
    \end{split}
\end{equation}
and so on, that is, the sum is larger when we consider $k=m$ than when we consider any other $k>m$. \claimqed

Calculating the bound, if we put all terms of \ref{eq:133} in one side and call it delta and doing the same for our case now we obtain
\begin{equation}
    \begin{split}
    \delta &= r(sm)^{2} - (r-2)(sm)^{2} - 2[(sm-m)^{2}+m^{2}]\\
    &= 2(sm)^{2} - 2[(sm-m)^{2}+m^{2}]\\
    &= 4(s-1)m^2
    \end{split}
\end{equation}

Thus, for $s\geq 2$
\begin{equation}
    \boxed{\text{SDP}(M_{r}-M_{s}) \geq 4(s-1)m^2}
\end{equation}
\end{proofclaim}

Summarizing the bounds obtained we have the following.

\begin{table}[h!]
\centering
\begin{tabular}{|c|c|c|}
\hline
\textbf{Case} & \textbf{Bound} & \textbf{Minimum Value} \\
\hline
\textbf{Case 1} & 
$\text{SDP}(M_{r}-M_{s}) \geq [(s^{2}t-st^{2})m^{2}](M_{in}-M_{out})$ & 
$\text{SDP}(M_{r}-M_{s}) \geq 2m^2(M_{in}-M_{out})$ \\
& & with $s=2$, $t=1$ \\
\hline
\textbf{Case 2, subcase 1, $\frac{r}{s}$ odd} & 
$\text{SDP}(M_{r}-M_{s}) \geq [(rs^{2}-s^{3})m^{2}](M_{in}-M_{out})$ & 
$\text{SDP}(M_{r}-M_{s}) \geq 4m^2(M_{in}-M_{out})$ \\
& & with $s=2$, $r=3$ \\
\hline
\textbf{Case 2, subcase 1, $\frac{r}{s}$ even} & 
$\text{SDP}(M_{r}-M_{s}) \geq [(rs^{2}-4)m^{2}](M_{in}-M_{out})$ & 
$\text{SDP}(M_{r}-M_{s}) \geq 8m^2(M_{in}-M_{out})$ \\
& & with $s=2$, $r=3$ \\
\hline
\textbf{Case 2, subcase 2} & 
$\text{SDP}(M_{r}-M_{s}) \geq [4(s-1)m^{2}](M_{in}-M_{out})$ & 
$\text{SDP}(M_{r}-M_{s}) \geq 4m^2(M_{in}-M_{out})$ \\
& & with $s=2$ \\
\hline
\end{tabular}
\caption{Summary of bounds obtained.}
\end{table}

Then, we conclude that the smallest bound is the minimum value for the first case, as we wanted. 

\end{proof}

\end{document}